\documentclass[preprint,12pt]{elsarticle}

\usepackage{amssymb}

\usepackage{amsmath}
\usepackage{amsthm}
\usepackage{float}
\usepackage{datetime}
\usepackage{hyperref}
\usepackage[all]{xy}
\SelectTips{eu}{}
\usepackage{fullpage}
\usepackage[mathscr]{eucal}
\usepackage{mathtools}
\usepackage{tikz}
\usetikzlibrary{matrix,arrows}
\usepackage{todonotes}
\usepackage{enumerate}

\usepackage{pinlabel}

\makeindex
\allowdisplaybreaks
\numberwithin{equation}{section}
\newtheorem{lemma}[equation]{Lemma}
\newtheorem{theorem}[equation]{Theorem}
\newtheorem{prop}[equation]{Proposition}
\newtheorem{cor}[equation]{Corollary}

\newenvironment{propqed}{\pushQED{\qed}\begin{prop}}{\popQED\end{prop}}

\theoremstyle{definition}
\newtheorem{defn}[equation]{Definition}
\newtheorem{eg}[equation]{Example}
\newtheorem{rk}[equation]{Remark}

\newcommand{\bC}{\mathbb C}
\newcommand{\bQ}{\mathbb Q}
\newcommand{\bR}{\mathbb R}
\newcommand{\bZ}{\mathbb Z}

\newcommand{\fa}{\mathfrak a}
\newcommand{\fb}{\mathfrak b}
\newcommand{\fc}{\mathfrak c}

\newcommand{\fH}{\mathfrak{H}}
\newcommand{\fK}{\mathfrak{K}}

\newcommand{\fX}{\mathfrak{X}}
\newcommand{\fY}{\mathfrak{Y}}

\newcommand{\mra}{\mathrm{a}}
\newcommand{\mrb}{\mathrm{b}}
\newcommand{\mrc}{\mathrm{c}}
\newcommand{\mrd}{\mathrm{d}}
\newcommand{\Aut}{\mathsf{Aut}}
\newcommand{\Diag}{\mathsf{Diag}}
\newcommand{\Ext}{\mathsf{Ext}}

\newcommand{\grph}{\mathsf{graph}}
\newcommand{\mcg}[1]{\Gamma_{#1}}

\newcommand{\Hom}{\mathsf{Hom}}
\newcommand{\Homeo}{\mathsf{Homeo}}

\newcommand{\Orth}{\mathsf{O}}

\newcommand{\SL}{\mathsf{SL}}
\newcommand{\SO}{\mathsf{SO}}
\newcommand{\Sp}{\mathsf{Sp}}

\newcommand{\fsp}{\mathfrak{sp}}

\renewcommand{\leq}{\leqslant}
\renewcommand{\geq}{\geqslant}
\newcommand{\sign}{\mathsf{sign}}
\newcommand{\UU}{\mathsf{U}}

\newcommand{\Tr}{\mathsf{Tr}}

\newcommand{\braid}[6]{\xymatrix@R=15pt@!C=20pt{
  #1 \ar[dr] \ar@/^1.6pc/[rr] && #2 \ar[dr] \ar@/^1.6pc/[rr] && #3 \\
  & #4 \ar[ur] \ar[dr] && #5 \ar[ur] \\ && #6 \ar[ur]}}

\newcommand{\bbZ}[1]{\mathbb Z/{#1}}
 \newcommand{\braidtwo}[6]{\xymatrix@R=15pt@!C=20pt{
  #1 \ar[dr] \ar@/^1.6pc/[rr] && #2 \ar[dr] &\\
  & #3 \ar[ur] \ar[dr] && #4  \\
  #5 \ar[ur] \ar@/_1.6pc/[rr]  && #6 \ar[ur] &}}

\journal{Journal of Pure and Applied Algebra}

\begin{document}

\begin{frontmatter}



\title{Signature cocycles on the mapping class group and symplectic groups}


\author{Dave Benson}
\address{Institute of Mathematics, University of Aberdeen, Aberdeen
  AB24 3UE, Scotland, United Kingdom}
\author{Caterina Campagnolo}
\address{Department of Mathematics, Karlsruhe Institute of Technology,
D-76128 Karlsruhe, Germany}
\author{Andrew Ranicki}
\address{School of Mathematics, University of Edinburgh, Edinburgh EH9
  3FD, Scotland, United Kingdom}
\author{Carmen Rovi}
\address{Excellence Cluster ``Structures", University of Heidelberg, D-69120 Heidelberg, Germany}

\begin{abstract}
Werner Meyer constructed a cocycle in $H^2(\Sp(2g, \bZ);\bZ)$ which computes the signature of a closed oriented surface bundle over a surface. By studying properties of this cocycle, he also showed that the signature of such a surface bundle 
is a multiple of $4$. In this paper, we study signature cocycles both from the geometric and algebraic points of view. We present geometric constructions which are relevant to the signature cocycle and provide an alternative to Meyer's decomposition of a surface bundle. Furthermore, we discuss the precise relation between the Meyer and Wall-Maslov index. The main theorem of the paper, Theorem \ref{th:main}, provides the necessary group cohomology results to analyze the signature of a surface bundle modulo any integer $N$. Using these results, we are able to give a complete answer for $N = 2, 4, \textnormal{ and } 8$, and based on a theorem of Deligne, we show that this is the best we can hope for using this method.

\end{abstract}



\begin{keyword} 
Cocycles, signature, fibre bundle, symplectic group


\MSC 20J06 (primary)~;~ 55R10, 20C33 (secondary)

\end{keyword}

\end{frontmatter}






\section*{Introduction}

Given an oriented $4$-manifold $M$ with boundary, let $\sigma(M)\in \bZ$ be the
signature\index{signature} of $M$. As usual, let $\Sigma_g$ be the standard closed oriented
surface of genus $g$. For the total space $E$ of a 
surface bundle $\Sigma_g \to E \to \Sigma_h$,
it is known from the work of Meyer \cite{Meyer:1973a} that 
$\sigma(E)$ is determined by a cohomology class 
$[\tau_g]\in H^2(\Sp(2g,\bZ);\bZ)$, 
and that both $\sigma(E)$ and $[\tau_g]$ are divisible by four. This raises the question of further divisibility by other multiples of two.

Indeed, the higher divisibility of its signature is strongly related to the monodromy of the surface bundle: it is known since the work of Chern, Hirzebruch and Serre \cite{chern-hirzebruch-serre} that trivial monodromy in $\Sp(2g, \bZ)$ implies signature $0$. Rovi \cite{Rovi:AGT} showed that monodromy in the kernel of $\Sp(2g, \bZ)\rightarrow \Sp(2g, \bZ/4)$ implies signature divisible by $8$, and very recently Benson \cite{Benson:theta} proved that the monodromy lying in the even bigger theta subgroup $\Sp^q(2g, \bZ)$ implies signature divisible by $8$. This settled a special case of a conjecture by Klaus and Teichner (see the introduction of \cite{Hambleton/Korzeniewski/Ranicki:2007a}), namely that if the monodromy lies in the kernel of $\Sp(2g, \bZ)\rightarrow \Sp(2g, \bZ/2)$, the signature is divisible by $8$. This result also follows by work of Galatius and Randal-Williams \cite{galatius-randal-williams}, by completely different methods.

It is interesting to recall that the signature of a $4k$-dimensional compact oriented hyperbolic manifold is equal to $0$, and the signature of a $4$-dimensional smooth spin closed manifolds is divisible by $16$.


In a first paper \cite{BensonCampagnoloRanickiRovi:2017a}, we constructed a cohomology class in the second cohomology group of a finite quotient $\fH$ of $\Sp(2g, \bZ)$ that computes the mod 2 reduction of the signature divided by $4$ for a surface bundle over a surface.

The main results of the present paper are contained in Theorem \ref{th:main}. It consists of an extensive study of the inflations and restrictions of the Meyer class to various quotients of the symplectic group. This provides the necessary group cohomology results to analyze the signature of a surface bundle modulo any integer $N$. Its upshot can be summarized in the diagram below, valid for $g\geq 4$, which illustrates the connection between Meyer's cohomology class, on the bottom
right, and the central extension $\tilde \fH$ of $\fH$ regarded as a cohomology class, on the top left. The group $\fH$ is the smallest quotient of the symplectic group that contains the cohomological information about the signature modulo $8$ of a surface bundle over a surface.

\medskip
\[ \xymatrix@=8mm{
\bZ/2=H^2(\fH;\bZ/2) \ar[r]^\cong \ar[d]^\cong &
H^2(\fH;\bZ/8)=\bZ/2 \ar[d]^\cong \\
\bZ/2=H^2(\Sp(2g,\bZ/4);\bZ/2) \ar[r]^\cong \ar[d]^\cong &
H^2(\Sp(2g,\bZ/4);\bZ/8)=\bZ/2 \ar[d]^4 \\
\bZ/2=H^2(\Sp(2g,\bZ);\bZ/2) \ar[r]^4 &
H^2(\Sp(2g,\bZ);\bZ/8)=\bZ/8 \\
\bZ=H^2(\Sp(2g,\bZ);\bZ) \ar[r]^4 \ar@{->>}[u] &
H^2(\Sp(2g,\bZ);\bZ)=\bZ \ar@{->>}[u]} \]

\medskip

In these new computations, we focus on $g\geq 2$. This is a novelty with respect to \cite{BensonCampagnoloRanickiRovi:2017a}, where we only considered $g\geq 4$. Our reasons for starting with $g=2$ are threefold: when $g=1, 2$, the signature of surface bundles over surfaces is $0$ anyway \cite{Meyer:1973a}, so that we do not need these cases for our purposes. But while $g=2$ compares in a reasonable way to the higher genera, the cohomology of $\Sp(2, \bZ)$ and its quotients behave quite differently: even though part of the properties we investigate are shared by all values of $g$, the common pattern is more often lost when $g=1$. Finally, as it happens that $\Sp(2, \bZ)=\SL(2, \bZ)$, the cohomology of this group has been studied in depth in other contexts \cite{barge-ghys, beyl, kirby-melvin}.

We apply Theorem \ref{th:main} to the study of the signature modulo $N$ for surface bundles. The outcome is the following:

\medskip
\noindent \textbf{Theorem 6.7.}
Let $g\geqslant 4$ and $Q$ be a finite quotient of $\Sp(2g, \bZ)$.
If $c \in H^2(Q;\bZ/N)$ is a cohomology class such that for the monodromy $\bar\chi\colon\pi_1(\Sigma_h)\rightarrow\Sp(2g,\bZ)$ of any surface bundle $\Sigma_g\rightarrow E\rightarrow\Sigma_h$ we have $$\sigma(E) = -\langle\bar\chi^*p^*(c),[\Sigma_h]\rangle \in\bZ/N,$$
then $N=2, 4,$ or $8$. If $N=2$ or $4$, the value $\langle\bar\chi^*p^*(c),[\Sigma_h]\rangle \in\bZ/N$ is always $0$. Here $p\colon\Sp(2g,\bZ)\rightarrow Q$ denotes the quotient map.

\medskip

Furthermore, we generalize this result to arbitrary coefficient rings and show that any cohomology class on a finite quotient of the symplectic group yields at most $2$-valued information on the signature of a surface bundle over a surface since it lies in the $2$-torsion of the arbitrary coefficient ring.
This explains the choice of the study of the signature modulo $8$ by showing that it is the best information a cohomology class on a finite quotient of the symplectic group can provide.


The first sections of this paper are meant as a survey of essential notions: in particular, we review forms, signature, Novikov additivity and Wall non-additivity of the signature, and include a discussion of the relation between the Meyer cocycle and the Maslov cocycle.
Furthermore, we describe the geometric constructions relevant to the signature cocycle. In \cite{Meyer:1973a},  Meyer constructs his cocycle by decomposing the base space of the surface bundle into pairs of pants and then using Novikov additivity of the signature. Here we present an alternative construction in Figure  \ref{fig:wee-monster2}.

In Section \ref{se:lie} we explain the cohomology of Lie groups seen as discrete groups and its relationship to their usual cohomology. Specifically, in Section \ref{se:Cocycles on S1}, we study the case of the unit circle seen as a discrete group and give expressions for the restriction of the Meyer and Maslov cocycles in this setting. In Section \ref{se:Sp2gR}, we examine the homotopy type and cohomology of
$\Sp(2g,\bR)$, and use this to compare the cocycles described in
Sections \ref{se:forms} and \ref{se:Cocycles on S1}. 
Section \ref{se:proofs} contains the main results of the paper.

The paper is complemented by two appendices: \ref{app: co-homology computations} collects computations of the first and second homology and cohomology of the mapping class group, the symplectic group and some of its quotients, as will be often used in the main sections. Finally, \ref{app: biography} is a biography of W. Meyer by W. Scharlau.

\section{The signature}\label{se:forms}

In this section, we start by reviewing fundamental definitions of forms, lagrangians, and the signature. We also give a discussion of the additivity properties (Novikov additivity, Proposition \ref{Novikov}) and the non-additivity properties (Wall non-additivity, Proposition \ref{Wall-non-add}) of the signature.    
In particular we recall the definitions of Wall-Maslov index (Definition \ref{Def:Wall-Maslov}), the Meyer cocycle  (Definition \ref{df:Meyercocycle}), and the Maslov cocycle (Definition \ref{df:Maslovcocycle}).  We illustrate each of these definitions with examples and furthermore recall the connection between these three fundamental notions.

\subsection{Forms}

Let $R$ be a commutative ring. The {\it dual} of an $R$-module $V$ is the $R$-module
$$V^*~=~\Hom_R(V,R)~.$$
The {\it dual} of an $R$-module morphism $f\colon V \to W$ is the $R$-module morphism
$$f^*\colon W^* \longrightarrow V^*~;~g \longmapsto (x \longmapsto g(f(x)))~.$$
As usual, we have an isomorphism of abelian groups
$$V^* \otimes_RV^* \longrightarrow\Hom_R(V,V^*) ~;~f \otimes g \longmapsto (x \longmapsto (y \longmapsto f(x)g(y)))~.$$
If $V$ is f.g. free then so is $V^*$, and the  natural $R$-module morphism 
$$V \longrightarrow V^{**}~;~ x \longmapsto (f \longmapsto f(x))$$
is an isomorphism, in which case it will be used to identify $V=V^{**}$.

A {\it form} $(V,\fb)$ over $R$ is a f.g. free $R$-module
$V$ together with a bilinear pairing $\fb\colon V \times V \rightarrow R$, or equivalently the $R$-module morphism
$$\fb\colon V \longrightarrow V^*~;~ x \longmapsto (y \longmapsto \fb(x,y))~.$$
A {\it morphism} of forms $f\colon (V,\fb) \to (W,\fc)$ over $R$ is an $R$-module
morphism $f\colon V \to W$ such that 
$$f^*\fc f~=~\fb\colon V \longrightarrow V^*$$
or equivalently
$$\fc(f(x),f(y))~=~\fb(x,y) \in R~(x,y \in V)~.$$
\indent
For $\epsilon=\pm 1$ a bilinear form $(V,\fb)$ is {\it $\epsilon$-symmetric} if 
$$\fb(y,x)~=~ \epsilon \fb(x,y) \in R ~(x,y \in V)~,$$ 
or equivalently $\epsilon \fb^*=\fb \in {\rm Hom}_R(V,V^*)$.
For $\epsilon=1$ the form is {\it symmetric}; for $\epsilon=-1$ the form is {\it symplectic}.

Given a form $(V,\fb)$ over $R$ the {\it orthogonal} of a submodule $L \subseteq V$ is the submodule
$$L^{\perp}~=~\{x \in V \,\vert\,\fb(x,y)=0 \in R~\hbox{for all}~y \in L\}~.$$
The {\it radical} of $(V,\fb)$ is the orthogonal of $V$
$$V^{\perp}~=~\{x \in V \,\vert\,\fb(x,y)=0 \in R~\hbox{for all}~y \in V\}~.$$
The form $(V,\fb)$ is {\it nonsingular} if the $R$-module morphism
$$\fb\colon V \longrightarrow V^*~;~ x \longmapsto (y \longmapsto \fb(x,y))$$
is an isomorphism, in which case $V^{\perp}=0$.

A  {\it lagrangian} of a nonsingular form $(V,\fb)$ is a f.g. free direct summand
$L \subset V$ such that $L^{\perp}=L$, or equivalently such that the sequence 
$$\xymatrix{0 \ar[r] & L \ar[r]^-{j} & V \ar[r]^-{j^*\fb} & L^* \ar[r] & 0}$$
is exact, with $j\colon L \to V$ the inclusion. 

\begin{propqed} 
For any nonsingular $\epsilon$-symmetric form $(V,\fb)$ over $R$
there is defined a lagrangian of $(V,\fb) \oplus (V,-\fb)$
$$\Delta~=~\{(x,x) \in V \oplus V\,\vert\, x\in V\}~.$$
For any $\alpha \in {\rm Aut}(V,\fb)$ the image of the diagonal lagrangian under the automorphism 
$$1 \oplus \alpha\colon (V,\fb) \oplus (V,-\fb) \longrightarrow (V,\fb) \oplus (V,-\fb)$$
is a lagrangian of $(V,\fb) \oplus (V,-\fb)$
\begin{equation*}
\grph(\alpha)~=~(1\oplus \alpha)(\Delta)~=~\{(x,\alpha(x))\,\vert\,x \in V\} \subseteq V \oplus V~.
\qedhere
\end{equation*}
\end{propqed}

\begin{eg} (i) The intersection form of a $2n$-dimensional manifold with boundary $(M,\partial M)$ is the $(-1)^n$-symmetric form
$(H^n(M,\partial M),\fb_M)$ over $\bR$
$$\fb_M\colon H^n(M,\partial M;\bR) \times H^n(M,\partial M;\bR) \longrightarrow \bR~;~(x,y) \longmapsto \langle x \cup y,[M] \rangle$$
with radical
$$H^n(M,\partial M;\bR)^{\perp}=\mathsf{ker}\left(H^n(M,\partial M;\bR) \longrightarrow H^n(M;\bR)\right).$$
If $M$ is closed then $(H^n(M;\bR),\fb_M)$ is nonsingular.\\
(ii) If $(N,\partial N)$ is a $(2n+1)$-dimensional manifold with boundary then
$$\mathsf{ker}\left(H^n(\partial N;\bR) \longrightarrow H^n(N;\bR)\right) \subseteq H^n(\partial N;\bR)$$
is a lagrangian of $(H^n(\partial N;\bR),\fb_{\partial N})$.
\end{eg}

%

\subsection{Signature}

\begin{defn}The {\it signature} of a symmetric form $(V,\fb)$ over $\bR$ is defined as usual by
$$\sigma(V,\fb)~=~{\rm dim}_\bR V_+ - {\rm dim}_\bR V_- \in \bZ$$
for any decomposition
$$(V,\fb)~=~(V_+,\fb_+) \oplus (V_-,\fb_-) \oplus (V_0,0)$$
with $(V_+,\fb_+)$ (resp. $(V_-,\fb_-)$) positive (resp. negative) definite.
\end{defn}

\begin{defn} The {\it signature} of a $4k$-dimensional manifold with boundary $(W,\partial W)$ is
$$\sigma(W) ~=~ \sigma(H^2(W,\partial W;\bR),\fb_W)\in\bZ~.$$
\end{defn}

We shall only be concerned with the case $k=1$.

\begin{propqed}\label{Novikov} {\rm (Novikov additivity of the signature)}\label{prop:novikovadditivity}\\
The signature of the union $W_1 \cup W_2$ of $4k$-dimensional manifolds with boundary along the whole boundary components is
\begin{equation*}
\sigma(W_1 \cup W_2)~=~\sigma(W_1)+\sigma(W_2)~.
\qedhere
\end{equation*}
\end{propqed}

The following construction is central to the computation of the signature of singular symmetric forms over $\bR$,
such as arise from 4-dimensional manifolds with boundary.

\begin{prop} {\rm  (Wall \cite{Wall:1969a})}\label{le:wall}
Let $(V,\fb)$ be a nonsingular symplectic form over $R$, and let
$j_1\colon L_1 \to V$, $j_2\colon L_2 \to V$, $j_3\colon L_3 \to V$ be the inclusions of three
lagrangians such that the $R$-module
$$\begin{array}{ll}
\Delta&=~ \Delta(L_1,L_2,L_3)~=~ \{(a,b,c)\in L_1\oplus L_2 \oplus L_3 \mid a+b+c=0\in V\}\\[1ex]
&=~\mathsf{ker}\left((j_1~j_2~j_3)\colon L_1 \oplus L_2 \oplus L_3 \longrightarrow V\right)
\end{array}$$
is f.g. free.\\
{\rm (i)} The form $(\Delta,\fa)$ over $R$ defined by
$$\begin{array}{l}
\fa~=~\begin{pmatrix} 0 & j_1^*\fb j_2 & 0 \\ 0 & 0 & 0 \\ 0 & 0 & 0 \end{pmatrix}\colon \Delta \times \Delta \longrightarrow R~;\\[1ex] 
\hskip100pt \left((a,b,c),(a',b',c')\right) \longmapsto \fb(b,a') = -\fb(c,a') = \fb(c,b')~, 
\end{array}$$
is symmetric, with radical
$$\Delta^{\perp}~=~\dfrac{(L_1 \cap L_2) \oplus (L_2 \cap L_3) \oplus (L_3 \cap L_1)}{L_1 \cap L_2 \cap L_3} \subseteq \Delta~.$$
{\rm (ii)}  The symmetric form
$$(\Delta',\fa')~=~\left(\mathsf{ker}\left((j_3^*\fb j_1~j_3^*\fb j_2)\colon L_1 \oplus L_2 \longrightarrow L_3^*\right),\begin{pmatrix} 0 &j_1^* \fb j_2\\
0 & 0 \end{pmatrix}\right)$$
is such that there is defined an isomorphism $f\colon(\Delta,\fa) \cong (\Delta',\fa')$ with
$$f\colon\Delta \longrightarrow\Delta'~=\{(a,b) \in L_1 \oplus L_2\,\vert\,a+b \in L_3\}~;~(a,b,c) \longmapsto (a,b)~.$$
{\rm (iii)} If $j_3^*\fb j_2\colon L_2 \to L^*_3$ is an isomorphism there is defined
an isomorphism of symmetric forms
$$\begin{pmatrix} 
1 \\ -(j_3^*\fb j_2)^{-1}(j_3^*\fb j_1)\end{pmatrix}\colon
\left(L_1,-(j_1^*\fb j_2)(j_3^*\fb j_2)^{-1}(j_3^*\fb j_1) \right)
\xymatrix{\ar[r]^-{\cong}&} (\Delta',\fa')~.
$$
\end{prop}
\begin{proof} (i) See Wall \cite{Wall:1969a}.\\
(ii)+(iii) Immediate from (i).
\end{proof}

\begin{rk} The construction appeared independently later in the work of Leray \cite{Leray:1981a}, clarifying earlier work of Maslov \cite{Maslov:1965a}.
\end{rk}

\begin{defn}\label{Def:Wall-Maslov} {\rm (Maslov \cite{Maslov:1965a}, Wall \cite{Wall:1969a}, Leray \cite{Leray:1981a})} 
 The {\it Wall--Maslov index} for any configuration of three lagrangians $L_1,L_2,L_3$ of a nonsingular symplectic form $(V,\fb)$ over $\bR$ is defined to be the signature
 $$\tau(L_1,L_2,L_3)~=~\sigma(\Delta(L_1,L_2,L_3),\fa) \in \bZ~$$
 with $\fa$ as in Proposition \ref{le:wall}.
%
\end{defn}

\begin{propqed} \label{Wall-non-add} {\rm  (Wall non-additivity of the signature \cite{Wall:1969a})}\label{prop:nonadditivity}
Let $(W,\partial W)$ be a 4-dimensional manifold with boundary which is a union of three codimension 0 submanifolds
with boundary $(W_i,\partial_+ W_i \sqcup \partial_- W_i)$, $i=1,2,3$,
$$(W,\partial W)~=~(W_1 \cup W_2 \cup W_3,\partial_-W_1\sqcup \partial_-W_2 \sqcup \partial_-W_3)$$
such that $W_1 \cap W_2$, $W_2 \cap W_3$ and $W_3 \cap W_1$ are three 3-manifolds with the same boundary surface
$$W_1 \cap W_2 \cap W_3~=~\partial (W_1 \cap W_2)~=~\partial (W_2 \cap W_3)~=~\partial (W_3 \cap W_1)~=~\Sigma~.$$
The signature of $(W,\partial W)$ is  
$$\sigma(W)~=~\sigma(W_1)+\sigma(W_2)+\sigma(W_3)-\sigma(\Delta(L_1,L_2,L_3)) \in \bZ$$
with  the three lagrangians of $(H^1(\Sigma;\bR),\fb_{\Sigma})$
\begin{align*}
L_1~=~\mathsf{ker}\left(H^1(\Sigma;\bR) \longrightarrow H^1(W_2 \cap W_3;\bR)\right),\\
L_2~=~\mathsf{ker}\left(H^1(\Sigma;\bR) \longrightarrow H^1(W_3 \cap W_1;\bR)\right),\\
L_3~=~\mathsf{ker}\left(H^1(\Sigma;\bR) \longrightarrow H^1(W_1 \cap W_2;\bR)\right). & \qedhere
\end{align*}


\end{propqed}      
\begin{figure}
\labellist
\small\hair 2pt
\pinlabel \small {\textcolor{red}{$\partial_{-}W_1$}} at 18 133
\pinlabel \small {\textcolor{red}{$\partial_{-}W_2$}} at 207 133
\pinlabel \small {\textcolor{red}{$\partial_{-}W_3$}} at 115 7
\pinlabel \small {$W_1$} at 67 111
\pinlabel \small {$W_2$} at 153 111
\pinlabel \small {$W_3$} at 117 40
\pinlabel \small {$\Sigma$} at 115 78
\endlabellist
\centering
\includegraphics[scale=0.75]{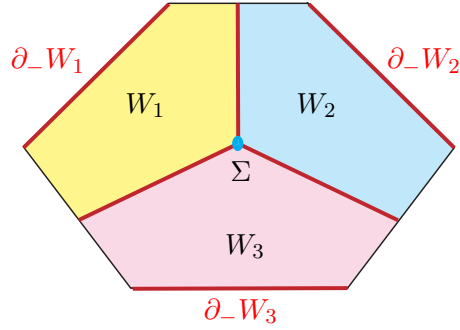}
\caption{The union $W_1 \cup W_2 \cup W_3$}
\label{fig:Maslov}
\end{figure}
\begin{rk}
The Novikov additivity of the signature is the special case of the Wall non-additivity of the signature when the glueing is  done along the whole boundary.
\end{rk}
\begin{defn}\label{twisted double}\label{thickening}
{\rm (i)} Let $N_1,N_2$ be two 3-manifolds with the same boundary
$$\partial N_1~=~\partial N_2~=~\Sigma_g~,$$
i.e. such that there are given embeddings $i_j\colon\Sigma_g \to N_j$ $(j=1,2)$ with $i_j(\Sigma_g)=\partial N_j$. The {\it twisted double} is the closed
3-dimensional manifold
$$D(N_1,N_2,i_1,i_2)~=~(N_1 \sqcup N_2)/
\{i_1(x) \sim i_2(x)\,\vert\, x \in \Sigma_g\}~.$$
{\rm (ii)} Let $N_1,N_2,N_3$ be three 3-manifolds with the same boundary
$$\partial N_1~=~\partial N_2~=~\partial N_2~=~\Sigma~,$$
i.e. such that there are given embeddings $i_j\colon\Sigma \to N_j$ $(j=1,2,3)$ with $i_j(\Sigma)=\partial N_j$. 
The {\it thickening} of the stratified set
$$N_1 \cup_{\Sigma} N_2 \cup_{\Sigma} N_3~=~(N_1 \sqcup N_2 \sqcup N_3)/
\{i_1(x) \sim i_2(x) \sim i_3(x)\,\vert\, x \in \Sigma\}$$
is the 4-dimensional manifold with boundary
$$\begin{array}{l}
\left(W(N_1,N_2,N_3,\Sigma),\partial W(N_1,N_2,N_3,\Sigma)\right)\\[1ex]
=~\left(D(N_1,N_2,i_1,i_2) \times I  \cup D(N_2,N_3,i_2,i_3) \times I \cup D(N_3,N_1,i_3,i_1) \times I,\right.\\[1ex]
\hskip150pt
\left. D(N_1,N_2,i_1,i_2)\sqcup D(N_2,N_3,i_2,i_3)\sqcup D(N_3,N_1,i_3,i_1)\right)~.
\end{array}$$

\begin{figure}[!ht]
\labellist
\small\hair 2pt
\pinlabel \small {$\Sigma$} at 290 173
\pinlabel \small {$N_1$} at 290 453
\pinlabel \small {$N_2$} at 520 59
\pinlabel \small {$N_3$} at 52 59
\pinlabel \small {$N_3 \cup N_1$} at 100 290
\pinlabel \small {\rotatebox{45}{$(N_3 \cup N_1) \times I$}} at 183 216
\pinlabel \small {$N_1 \cup N_2$} at 462 290
\pinlabel \small {\rotatebox{-45}{$(N_1 \cup N_2) \times I$}} at 375 216
\pinlabel \small {$N_2 \cup N_3$} at 290 36
\pinlabel \small {$(N_2 \cup N_3) \times I$} at 290 121
\endlabellist
\centering
\includegraphics[scale=0.35]{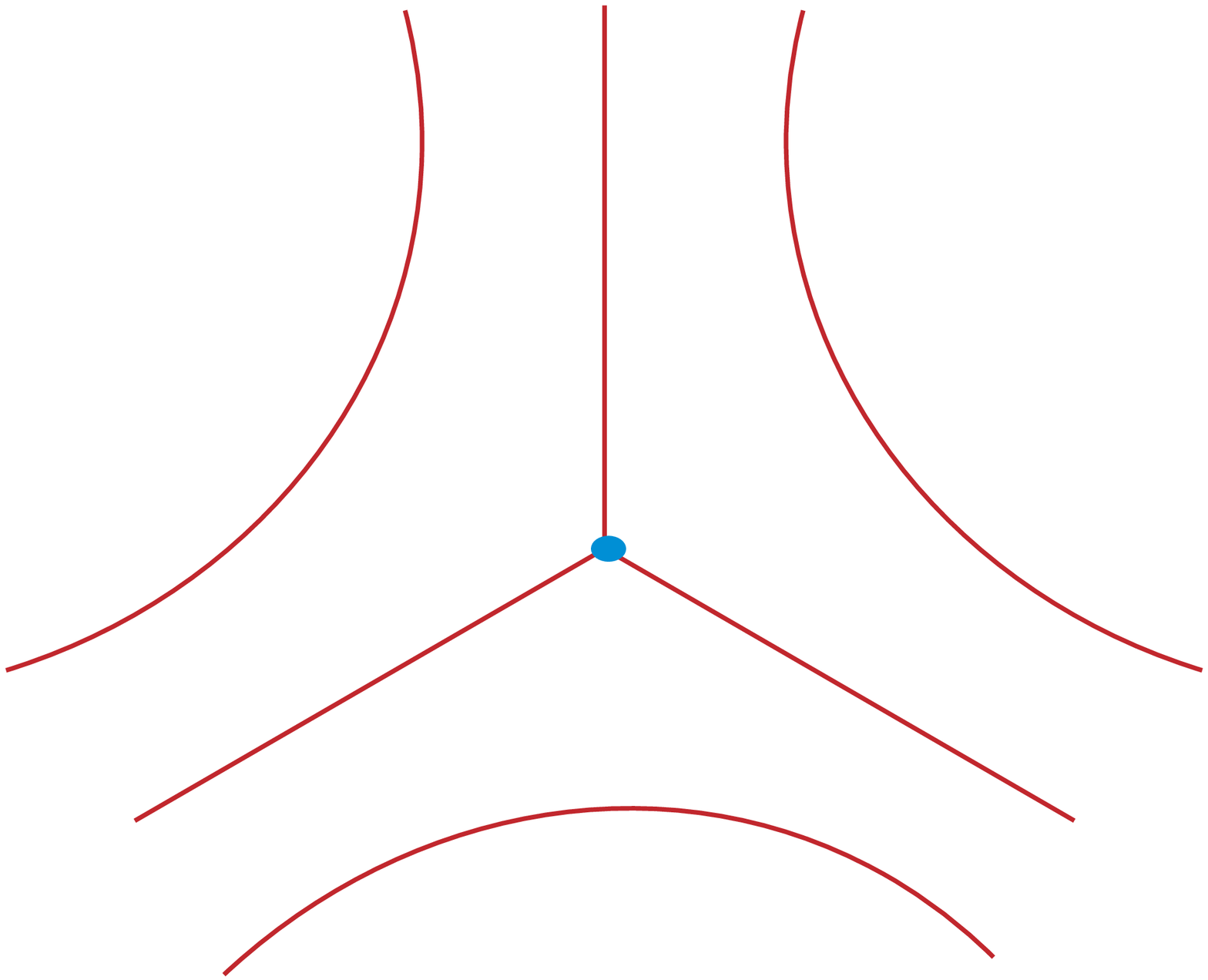}
\caption{ The thickening of $N_1 \cup_{\Sigma} N_2 \cup_{\Sigma} N_3$}
\label{fig:trisection}
\end{figure}
\end{defn}

\begin{prop} As in Definition \ref{thickening} (ii), let $N_1,N_2,N_3$ be  3-dimensional manifolds with the same boundary
$$\partial N_1~=~\partial N_2~=~\partial N_3~=~\Sigma~.$$
The nonsingular symplectic intersection form over $\bR$
$$(V,\fb)~=~\left(H^1(\Sigma;\bR),\fb_\Sigma\right)$$
has three lagrangians
$$L_i~=~\mathsf{Im}\left(H^1(N_i;\bR)\longrightarrow H^1(\Sigma;\bR)\right)~(i=1,2,3),$$ 
such that the signature of the thickening $W(N_1,N_2,N_3,\Sigma)$ is given by
$$\sigma\left(W(N_1,N_2,N_3,\Sigma)\right)~=~-\tau(L_1,L_2,L_3) \in \bZ~.$$
\end{prop}
\begin{proof}
This is a special case of Proposition \ref{prop:nonadditivity}.
\end{proof}

We shall use the group inclusion
$$S^1 \xymatrix{\ar[r]^-{\displaystyle{\cong}}&} \SO(2)\subset \Sp(2,\bR)~~;~~
e^{2\pi i a} \longmapsto \begin{pmatrix} \cos 2\pi a & -\sin 2\pi a\\
\sin 2\pi a & \cos 2\pi a \end{pmatrix}$$
as an identification.

\begin{eg}\label{eg:mu1}
Let $g=1$. The space ${\mathcal L}_1$ of the lagrangians of the nonsingular symplectic form over $\bR^2$
$$(V,\fb)~=~\left(\bR\oplus \bR, \begin{pmatrix} 0 & 1 \\ -1 & 0 \end{pmatrix}\right)$$
consists of all the 1-dimensional subspaces $L \subset \bR^2$, which are of the type 
$$L_a~=~\mathsf{Im}\left(\begin{pmatrix} \cos \pi a \\ \sin \pi a \end{pmatrix}\colon\bR \longrightarrow \bR \oplus \bR\right) \subset \bR \oplus \bR~(a \in [0,1))~,$$
with an injection
$${\mathcal L}_1 \longrightarrow \Sp(2,\bR)~~;~~L_a \longmapsto e^{2 \pi i a }=\begin{pmatrix}
\cos 2 \pi a & -\sin 2 \pi a \\
\sin 2 \pi a & \cos 2\pi a
\end{pmatrix}~.$$
For any $a_1,a_2,a_3 \in [0,1)$ the inclusions 
$$j_i~=~\begin{pmatrix} \cos \pi a_i \\ \sin \pi a_i \end{pmatrix}\colon\bR \longrightarrow \bR \oplus \bR~(i=1,2,3)$$
with images the lagrangians $j_i(\bR)=L_{a_i} \subset \bR \oplus \bR$ are such that
$$j_i^*\fb j_k~=~\sin \pi(a_k-a_i)\colon\bR \longrightarrow \bR~.$$
As in Lemma \ref{le:wall} (ii) there is defined an isomorphism of symmetric forms over $\bR$ 
$$\begin{array}{l}
f\colon\left(\Delta(L_{a_1},L_{a_2},L_{a_3}),\fa\right)\\[1ex]
=~\left(\mathsf{ker}\left(\begin{pmatrix} \cos \pi a_1 & \cos \pi a_2 & \cos\pi a_3 \\
\sin  \pi a_1 & \sin \pi a_2 & \sin \pi a_3 \end{pmatrix}\colon\bR\oplus \bR \oplus \bR \longrightarrow \bR\oplus \bR \right),
\begin{pmatrix} 0 & 1 & 0 \\ 0 & 0 & 0 \\ 0 & 0 & 0 \end{pmatrix}\right)\\[1ex]
\xymatrix{\ar[r]^-{\cong}&} (\Delta'(L_{a_1},L_{a_2},L_{a_3}),\fa')\\[1ex]
=~\left(\mathsf{ker}\left(\begin{pmatrix}
\sin \pi(a_1-a_3)& \sin\pi(a_2-a_3) \end{pmatrix}\colon\bR \oplus \bR\longrightarrow \bR\right),
\begin{pmatrix} 0 & \sin \pi(a_2-a_1) \\ 0 & 0 \end{pmatrix}\right)~. 
\end{array}$$
If $a_1,a_2,a_3$ are not all distinct then $\fb=0$, and 
$\tau(L_{a_1},L_{a_2},L_{a_3})=0$.
If $a_2 \neq a_3$ then Lemma \ref{le:wall} (iii) gives an isomorphism
$$\begin{pmatrix}
1 \\ -(j_3^*\fb j_2)^{-1}(j_3^*\fb j_1)\end{pmatrix}\colon
\left(\bR,-(j_1^*\fb j_2)(j_3^*\fb j_2)^{-1}(j_3^*\fb j_1)\right)
\xymatrix{\ar[r]^-{\cong}&} (\Delta',\fa')~.
$$
It follows that for any $a_1,a_2,a_3 \in [0,1)$ the Wall-Maslov index is
$$\begin{array}{ll}
\tau(L_{a_1},L_{a_2},L_{a_3})&=~
\sigma(\Delta (L_{a_1},L_{a_2},L_{a_3}),\fa)\\[1ex]
&=~\sigma(\Delta' (L_{a_1},L_{a_2},L_{a_3}),\fa')\\[1ex]
&=~\sigma(\bR,-\sin \pi(a_1-a_2) \sin\pi(a_2-a_3) \sin \pi(a_3-a_1))\\[1ex]
&=~-\sign(\sin \pi(a_1-a_2) \sin\pi(a_2-a_3) \sin \pi(a_3-a_1))~.
\end{array}$$
\end{eg}

\begin{lemma}\label{le:Maslov-invariant}
If $\alpha\in\Sp(2g,\bR)$ then $\tau(\alpha(L_1),\alpha(L_2),\alpha(L_3))= \tau(L_1,L_2,L_3)$.
\end{lemma}
\begin{proof}
This is clear from the definition.
\end{proof}

\begin{lemma}\label{le:Maslov-cocycle}
If $L_1$, $L_2$, $L_3$, $L_4$ are lagrangians in $\bR^{2g}$ then 
$$ \tau(L_1,L_2,L_3)-\tau(L_1,L_2,L_4)+\tau(L_1,L_3,L_4)-\tau(L_2,L_3,L_4)=0. $$
\end{lemma}
\begin{proof}See for example Py \cite[Th\'eor\`eme (2.2.1)]{Py:2005a}.
\end{proof}

Now suppose that the symplectic form over $\bR$
$$(\bR^{2g+2},\fb')~=~(\bR^{2g},\fb)\oplus \left(\bR \oplus \bR,\langle~,\rangle\right)$$ 
is formed from $(\bR^{2g},\fb)$ by adjoining new basis elements $v_{g+1}$ and $w_{g+1}$ with 
$\langle v_{g+1},v_{g+1}\rangle=0$, $\langle w_{g+1},w_{g+1}\rangle=0$ and $\langle v_{g+1},w_{g+1}\rangle=1$. Then given any lagrangian $L$ in $(\bR^{2g},\fb)$, 
we may form the lagrangian $\tilde L = L \oplus \bR v_{g+1}$ in $(\bR^{2g+2},\fb')$. 

\begin{lemma}\label{le:Maslov-restricts}
Setting $\tilde \Delta = \Delta(\tilde L_1,\tilde L_2,\tilde L_3)$, we have an isomorphism
$$(\tilde \Delta,\tilde\fa)~=~(\Delta,\fa)\oplus (\bR v_{g+1},0)~.$$
In particular, we have
$$ \tau(\tilde L_1,\tilde L_2,\tilde L_3) = \tau(L_1,L_2,L_3). $$
\end{lemma}
\begin{proof} The space $\tilde \Delta$ consists of vectors
$$(\tilde a,\tilde b,\tilde c)~=~ (a+\lambda v_{g+1},b+\mu v_{g+1},c+ \nu v_{g+1}) $$
with $(a,b,c)\in \Delta$ and $\lambda+\mu+\nu=0$, and the symmetric form $(\tilde \Delta,\tilde\fa)$ is given by
\begin{align*}
\tilde\fb((\tilde a,\tilde b,\tilde c),(\tilde a',\tilde b',\tilde c'))&=~
\tilde\fa(\tilde b,\tilde a')\\
&= ~\fa(b,a') ~=~\fb((a,b,c),(a',b',c'))~.
\qedhere
\end{align*}
\end{proof}

\begin{lemma}\label{le:graph} 
Let $(V,\fb)$ be a nonsingular symplectic form over $\bR$.\\
{\rm (i)} The graph of an automorphism $\alpha\colon(V,\fb) \to (V,\fb)$ is a lagrangian of
$(V \oplus V,\fb \oplus -\fb)$
$$\grph(\alpha)~=~\mathsf{Im}\left(\begin{pmatrix} 1 \\ \alpha \end{pmatrix}\colon
V \longrightarrow V \oplus V\right) \subset V \oplus V~.$$
{\rm (ii)} For any automorphisms $\alpha_i\colon(V,\fb) \to (V,\fb)$ $(i=1,2,3)$ let
$$j_i~=~\begin{pmatrix} 1 \\ \alpha_i \end{pmatrix}\colon (V,0) \longrightarrow (V \oplus V,\fb \oplus -\fb)$$
be the inclusions of the three graph lagrangians. For $i \neq k$
$$j_i^*(\fb \oplus -\fb)j_k~=~\fb-\alpha_i^* \fb \alpha_k~=~
\fb(1-\alpha_i^{-1}\alpha_k)\colon V \longrightarrow V^*~,$$
so that
$$\begin{array}{l}
(\Delta'(\grph(\alpha_1),\grph(\alpha_2),\grph(\alpha_3)),\fa')\\[1ex]
\hskip50pt
=~\left(\mathsf{ker}\left((1-\alpha_3^{-1}\alpha_1~1-\alpha_3^{-1}\alpha_2)\colon V \oplus V \longrightarrow V\right),\begin{pmatrix} 0 &\fb(1-\alpha_1^{-1} \alpha_2)\\
0 & 0 \end{pmatrix}\right)~.
\end{array}$$
If one of $1-\alpha_i^{-1} \alpha_k$ is $0$ then $\fb'=0$ and
$\tau(\grph(\alpha_1),\grph(\alpha_2),\grph(\alpha_3))=0$. If $1-\alpha^{-1}_3 \alpha_2\colon V \to V$ is an isomorphism there is defined
an isomorphism of symmetric forms 
$$\begin{array}{l}
\begin{pmatrix}
1 \\ -(\fb-\alpha_3^*\fb \alpha_2)^{-1}(\fb - \alpha_3^*\fb \alpha_1)\end{pmatrix}\colon
\left(V,-(\fb-\alpha_1^*\fb \alpha_2)(\fb-\alpha_3^*\fb \alpha_2)^{-1}(\fb-\alpha_3^*\fb \alpha_1)\right)\\[1ex]
\hskip100pt\xymatrix{\ar[r]^-{\cong}&} \left(\Delta'(\grph(\alpha_1),\grph(\alpha_2),\grph(\alpha_3)),\fb'\right)~,
\end{array}
$$
so that
$$\begin{array}{l}
\tau(\grph(\alpha_1),\grph(\alpha_2),\grph(\alpha_3))\\[1ex]
\hskip50pt=~
\tau(V,-(\fb-\alpha_1^*\fb \alpha_2)(\fb-\alpha_3^*\fb \alpha_2)^{-1}(\fb-\alpha_3^*\fb \alpha_1))\\[1ex]
\hskip50pt=~\tau(V,-\fb(1-\alpha_1^{-1}\alpha_2)(1-\alpha_3^{-1} \alpha_2)^{-1}(1-\alpha_3^{-1} \alpha_1))~.
\end{array}
$$
\end{lemma}
\begin{proof} (i) By construction.\\
(ii) Apply Lemma \ref{le:wall}.
\end{proof}

Note that the graph of any element 
$\alpha \in \Sp(2g,\bR)$ 
$$\grph(\alpha)~=~\{(x,\alpha(x))\,\vert\, x \in \bR^{2g}\} \subset \bR^{2g} \oplus \bR^{2g}$$ 
is a lagrangian of the symplectic form 
$$\left(\bR^{2g} \oplus \bR^{2g}, 
\left((x,y),(x',y')\right) \longmapsto  \langle x,x'\rangle - \langle
y,y'\rangle\right),$$
and that the signature of the double mapping torus (Section \ref{se:doublemappingtorus} below) is
$$ \sigma(T(\alpha,\beta))~=~-\tau(\grph(1),\grph(\alpha),\grph(\alpha\beta)) \in \bZ. $$

In what follows, when we write $\Sp(2g,\bR)^\delta$  we mean the same group $\Sp(2g,\bR)$, but with the discrete topology. For more details on the cohomology of Lie groups with discrete topology see Section \ref{se:lie}.

\begin{defn}\label{df:Meyercocycle}
The \emph{Meyer cocycle}\index{Meyer cocycle} $\tau_g\colon \Sp(2g,\bR)^\delta \times \Sp(2g,\bR)^\delta \to \bZ$ is defined by
$$ \tau_g(\alpha,\beta) = \tau(\grph(1),\grph(\alpha),\grph(\alpha\beta)). $$
\end{defn}

\begin{rk}
It follows from Lemma \ref{le:Maslov-cocycle} that $\tau_g$ is a cocycle on $\Sp(2g,\bR)^\delta$.
By Lemma \ref{le:Maslov-invariant} we have 
$$ \tau_g(\alpha,\beta)=\tau(\grph(\alpha^{-1}),\grph(1),\grph(\beta)). $$
So the subspace $\Delta$ consists of triples
$$ \left((x,\alpha^{-1}(x)),(-x-y,-x-y),(y,\beta(y))\right) $$
with $x,y \in \bR^{2g}$ such that 
$\alpha^{-1}(x)+\beta(y)=x+y$. Thus $\Delta$ is isomorphic to the
space 
$$ \{(x,y)\in\bR^{2g} \oplus \bR^{2g} \mid (\alpha^{-1}-1)x +(\beta-1)y =0\}, $$
with symmetric form
$$\fb\left((x,y),(x',y')\right)~=~\left\langle x+y,(1-\beta)y'\right\rangle~.$$
This is the definition found in Meyer \cite{Meyer:1973a}, except that the original cocycle was
on $\Sp(2g,\bZ)$.
\end{rk}

\begin{eg} \label{eg:tau1}
Let $(V,\fb)=\left(\bR \oplus \bR,\begin{pmatrix}  0 & 1 \\ -1 & 0 \end{pmatrix}\right)$, so that identifying $V=\bC$ we have $\fb=-i=e^{-\pi i /2}$. For $a \in [0,1)$ let
$$\begin{array}{l}
\alpha~=~e^{2\pi i a}~=~\begin{pmatrix} \cos 2 \pi a &  -\sin 2 \pi a \\
 \sin 2 \pi a &   \cos 2 \pi a  \end{pmatrix}   \in \Aut(V,\fb)=\Sp(2,\bR)~, \end{array} $$
so that
 $$ 1-\alpha~=~(2 \sin \pi a) e^{\pi i(a -1/2)}~
 =~(2 \sin \pi a) \begin{pmatrix}  \sin \pi a &
 \cos \pi a \\
 -\cos \pi a & \sin \pi a \end{pmatrix}~.$$
 (i) For any $a_1,a_2,a_3 \in [0,1)$ let $\alpha_j=e^{2\pi i a_j}$ $(j=1,2,3)$.
 By Lemma \ref{le:graph} (ii) the Wall-Maslov index of the graph lagrangians
 $\grph(\alpha_j)$ in $(V \oplus V,\fb \oplus -\fb)$ is
  $$\begin{array}{l}
\tau(\grph(\alpha_1),\grph(\alpha_2),\grph(\alpha_3))\\[1ex]
 =~\sigma(V,-\fb(1-\alpha^{-1}_1\alpha_2))(1-\alpha_3^{-1}\alpha_2)^{-1}(1-\alpha_3^{-1}\alpha_1))\\[1ex]
 =~\sigma(V,-e^{-\pi i /2}((2 \sin \pi (a_2-a_1)) e^{\pi i(a_2-a_1 -1/2)}) 
 ((2 \sin \pi (a_2-a_3) ) e^{\pi i(a_2-a_3 -1/2)})^{-1}\\[1ex]
 \hskip150pt ((2 \sin \pi (a_1-a_3)) e^{\pi i(a_1-a_3 -1/2)}))\\[1ex]
 =~ \sigma(V,\dfrac{2\sin \pi (a_2-a_1) \sin \pi(a_1-a_3)}{\sin \pi (a_2-a_3)})\\[1ex]
 =~2 \sign( \sin \pi (a_1-a_2) \sin \pi (a_2-a_3) \sin \pi (a_3-a_1)) \in \{0,2,-2\}~.
 \end{array}$$ 
  (ii)  By (i) the evaluation of the Meyer cocycle on  $\alpha=e^{2\pi i a}$, $\beta=e^{2\pi i b} \in \Sp(2,\bR)$ is
 $$\begin{array}{ll}
\tau_1(\alpha,\beta)& =~\tau(\grph(1),\grph(\alpha),\grph(\alpha\beta))\\[1ex]
&=~2 \sign( \sin \pi a \sin \pi b \sin \pi (a+b)) \in \{0,2,-2\}~.
 \end{array}$$ 
\end{eg}

We now consider another cocycle $\tau'_g$ on $\Sp(2g,\bR)^\delta$, again defined 
in terms of the Maslov index. We shall see in Theorem \ref{th:tautau'} below that it is cohomologous to the Meyer cocycle $\tau_g$, but not equal to it. (For $g=1$ see Example \ref{eg:tauprime1}).

\begin{defn}\label{df:Maslovcocycle}
The \emph{Maslov cocycle}\index{Maslov cocycle}
$\tau'_g\colon \Sp(2g,\bR)^\delta \times \Sp(2g,\bR)^\delta \to \bZ$
is defined as follows. Choose a lagrangian subspace $L \subseteq
\bR^{2g}$, and set
$$ \tau'_g(\alpha,\beta) = \tau(L,\alpha(L),\alpha\beta(L)). $$
It follows from Lemma \ref{le:Maslov-invariant} that this is 
independent of the choice of $L$, and that
$$ \tau'_g(\alpha,\beta) = \tau(\alpha^{-1}(L),L,\beta(L)). $$
\end{defn}

\begin{eg}\label{eg:tauprime1}
Using the inclusion
$$[0,1) \xymatrix{\ar[r]^-{\cong}&} \SO(2)\subset\Sp(2,\bR)~~;~~a \longmapsto e^{2\pi i a}=
\begin{pmatrix} \cos 2\pi a & -\sin 2\pi a \\
\sin 2\pi a & \cos 2\pi a \end{pmatrix}$$
as an identification we have that for $g=1$ Example \ref{eg:tau1} gives the Meyer cocycle on $[0,1)$ to be
$$\tau_1\colon[0,1) \times [0,1) \longrightarrow \bZ~~;~~
(a,b) \longmapsto -2\sign(\sin \pi a \sin \pi b \sin \pi (a+b))~.$$
The Maslov cocycle is given by
$$\tau'_1\colon[0,1] \times [0,1] \longrightarrow \bZ~~;~~
(a,b) \longmapsto \tau(L_0,L_{2\pi a},L_{2\pi (a+b)})=-\sign(\sin 2\pi a \sin 2\pi b \sin (2\pi (a+b)))~.$$
The Dedekind $((~))$-function is defined by
$$((~))\colon\bR \longrightarrow (-1/2,1/2) ~~;~~ x \longmapsto ((x))~=~
\begin{cases} \{x\}-1/2&{\rm if}~x \in \bR\backslash \bZ,\\
0&{\rm if}~x \in \bZ, \end{cases}$$
with $\{x\} \in [0,1)$ the fractional part of $x \in \bR$, and is such that
$$\begin{array}{l}
2((2x))-4((x))~=~\sign(\sin 2\pi x)~,\\[1ex]
2(\,((x))+((y))-((x+y))\,)~=~-\sign(\sin \pi x \sin \pi y \sin \pi(x+y))~.
\end{array}$$
Thus
$$\begin{array}{l}
\tau_1(a,b)~=~4(\,((a))+((b))-((a+b))\,)~,\\[1ex]
\tau'_1(a,b)~=~2(\,((2a))+((2b))-((2a+2b))\,)~,\\[1ex]
\tau_1(a,b)-\tau'_1(a,b)~=~-\sign(\,\sin 2\pi a)-\sign(\sin 2\pi b)+\sign(\sin 2\pi(a+b)) \in \bZ
\end{array}
$$
and 
$$[\tau_1]~=~[\tau'_1] \in H^2(\Sp(2,\bR)^\delta;\bZ)~.$$
(See also Example \ref{eg:MeyerS1}.)
\end{eg}
\begin{rk}\label{gilmer-masbaum}
Gilmer and Masbaum recall a result of Walker in \cite[Theorem 8.10]{Gilmer/Masbaum:2013a} that states in their notations that $\left[\tau'_g\right]=-\left[\tau_g\right]$ for every $g\geq 1$. This is because they have the convention that $\tau_g(\alpha, \beta)$ is the signature of the surface bundle over a pair of pants defined by $\alpha, \beta$ \cite[p. 1087]{Gilmer/Masbaum:2013a}, while Meyer \cite[p. 243]{Meyer:1973a} and this paper (Definition \ref{df:Meyercocycle} and just above) take the opposite convention, that $\tau_g(\alpha, \beta)$ is minus the signature of the surface bundle over a pair of pants defined by $\alpha, \beta$.
\end{rk}

\begin{prop}\label{pr:cocycles-restrict}
Embed $\bR^{2g}$ in $\bR^{2g+2}$ by adjoining new basis elements
$v_{g+1}$ and $w_{g+1}$, and consider the corresponding embedding
$\Sp(2g,\bR)^\delta \to \Sp(2g+2,\bR)^\delta$. 
\begin{enumerate}
\item[\rm (i)] The Maslov cocycle $\tau'_{g+1}$ on $\Sp(2g+2,\bR)^\delta$ restricts to
the Maslov cocycle $\tau'_g$ on $\Sp(2g,\bR)^\delta$.
\item[\rm (ii)] The Meyer cocycle $\tau_{g+1}$ on $\Sp(2g+2,\bR)^\delta$ restricts to
the Meyer cocycle $\tau_g$ on $\Sp(2g,\bR)^\delta$.
\end{enumerate}
\end{prop}
\begin{proof}
(i) This follows immediately from Lemma \ref{le:Maslov-restricts}.

(ii) For $\alpha\in\Sp(2g,\bR)^\delta$, the graph of $\alpha$ considered as 
an element of $\Sp(2g+2,\bR)^\delta$ is the direct sum of the graph of $\alpha$
considered as an element of $\Sp(2g,\bR)^\delta$ and $\grph(1)\subset
\bR^2\oplus \bR^2$. So apply Lemma \ref{le:Maslov-restricts} twice.
\end{proof}


\section{Surface bundles and the mapping torus constructions} \label{se:geometry}

In Section \ref{se:surfacebundles} we recall the classification of oriented
surface bundles $\Sigma_g \to E \to B$ with the fibre $\Sigma_g$ the standard closed surface of genus $g$ and the
base $B$ an oriented manifold with boundary (which may be empty).  
We shall be particularly concerned with the construction of surface bundles in four specific cases of the base $B$, in each of which $E$ is obtained from the monodromy morphism $\pi_1(B) \to \Gamma_g$ to the mapping class group of $\Sigma_g$ by a geometric mapping torus construction: 
\begin{itemize}
\item[(i)]  for a circle (\ref{se:mappingtorus}), 
\item[(ii)]  for a pair of pants (\ref{se:doublemappingtorus}), 
\item[(iii)]  for a torus with one boundary component (\ref{se:commutatormappingtorus}), 
\item[(iv)] for a surface  (\ref{se:multiplecommutatormappingtorus}). 
\end{itemize}

Let $\Omega_4$ denote the oriented cobordism group in dimension $4$. In Section \ref{se:geometriccocycle} we construct a geometric cocycle $\tau \in H^2(\Gamma_g;\Omega_4)$ for the oriented cobordism class (= signature) of the total space $E$ of a surface bundle
$\Sigma_g \to E\to B=\Sigma_h$ over a surface.

\subsection{Classification of surface bundles}\label{se:surfacebundles}

Let $\Homeo(\Sigma_g)$ be the topological group of self-homeomorphisms $\alpha:\Sigma_g \to \Sigma_g$,
and let $\Homeo^+(\Sigma_g) \subset \Homeo(\Sigma_g)$\index{Homeo@$\Homeo^+(\Sigma_g)$}
be the subgroup of the orientation-preserving self-homeomorphisms. A surface bundle $\Sigma_g \to E \to B$
is oriented if the manifolds $B,F$ are oriented and  the structure group of the bundle is 
$ \Homeo^+(\Sigma_g)$, so that $E$ is also an oriented manifold.
We shall only be considering oriented surface bundles.

Let $E\Homeo^+(\Sigma_g)$ be a contractible space with a free $\Homeo^+(\Sigma_g)$-action, so that the surface bundle over the classifying space $B\Homeo^+(\Sigma_g)=E\Homeo^+(\Sigma_g)/\Homeo^+(\Sigma_g)$ 
$$\Sigma_g \longrightarrow E\Homeo^+(\Sigma_g)\times_{\Homeo^+(\Sigma_g)}\Sigma_g \longrightarrow 
B\Homeo^+(\Sigma_g)$$
is universal. 
We shall only consider $\Homeo^+(\Sigma_g)$ for $g\geqslant 2$, when the connected components are contractible.
The set of connected components is the mapping class group $\Gamma_g=\pi_0\Homeo^+(\Sigma_g)$,
the discrete group of isotopy classes of orientation preserving homeomorphisms $\alpha\colon\Sigma_g \to \Sigma_g$, and
the forgetful map $B\Homeo^+(\Sigma_g)\to B\Gamma_g$ is a homotopy equivalence.

\begin{propqed}
\label{pr:surfacebundles}
$(${\rm Farb and Margalit \cite[pp. 154--155]{Farb/Margalit:2012a}}$)$ For $g \geq 2$, every surface bundle $\Sigma_g \to E \to B$ is isomorphic to the pullback of the universal surface  
bundle along a map $B \to B\Gamma_g$, with the monodromy defining a bijective correspondence
\[
\{\hbox{isomorphism classes of surface bundles $\Sigma_g \to E\to B$}\}/\approx\]
\[
 [B,B\Homeo^+(\Sigma_g)]~=~ [B,B\Gamma_g]= \{\hbox{homotopy classes of maps $\chi\colon B \to B\Gamma_g$}\}~.
\qedhere\]
\end{propqed}

We shall be mainly concerned with surface bundles $\Sigma_g \to E \to B$ when $B$ is a connected $n$-dimensional manifold with boundary, so that $E$ is an $(n+2)$-dimensional manifold with boundary. The forgetful map $$[B,B\Homeo^+(\Sigma_g)] \to \{\textnormal{\textit{conjugacy classes of }} {\chi \in \Hom}(\pi_1(B),\Gamma_g)\}$$
is a bijection, so a surface bundle $\Sigma_g \to E \to B$ is determined by
the monodromy group morphism $\chi\colon\pi_1(B) \to \Gamma_g$.

\begin{eg}
For the circle $B=S^1$ a surface bundle $\Sigma_g \to E \to S^1$ is classified by the monodromy map $\alpha \in [S^1,B\Gamma_g]$, with $E$ isomorphic to the mapping torus $T(\alpha)$ (Section \ref{se:mappingtorus} below).
\end{eg}

For $h,k \geqslant 0$ let $\Sigma_{h,k}$ be the connected surface with $k$ boundary components
$$(\Sigma_{h,k},\partial \Sigma_{h,k})~=~
\left({\rm cl.}(\Sigma_h\backslash \mathop{\sqcup}\limits_k D^2),
\mathop{\sqcup}\limits_kS^1\right)$$
with Euler characteristic $\chi(\Sigma_{h,k})=2-2h-k$. 
A surface bundle 
$$\Sigma_g \longrightarrow (E,\partial E) \longrightarrow (\Sigma_{h,k},\partial \Sigma_{h,k})$$
is classified by the monodromy group morphism
$$\begin{array}{l}
\chi\colon\pi_1(\Sigma_{h,k})~=~\left\langle x_1,y_1,\dots,x_h,y_h,z_1,\dots,z_k\,\vert\,
[x_1,y_1]\dots [x_h,y_h]=z_1\dots z_k\right\rangle \longrightarrow \Gamma_g~;\\[1ex]
\hskip150pt
x_i \longmapsto \alpha_i~,~y_i \longmapsto \beta_i~,~z_j \longmapsto \gamma_j~.
\end{array}$$

\begin{eg} For the pair of pants, $P=\Sigma_{0,3}$ a surface bundle 
$\Sigma_g \to (E,\partial E) \to (P,\partial P)$ 
is classified by the monodromy morphism 
$$\chi\colon\pi_1(P)~=~\left\langle x,y\right\rangle \longrightarrow \Gamma_g~;~x \longmapsto \alpha~,~y \longmapsto \beta$$
with $E$ isomorphic to the double mapping torus $T(\alpha,\beta)$ (Section \ref{se:doublemappingtorus} below) and $\partial E=T(\alpha) \sqcup T(\beta)\sqcup T(\alpha\beta)$.
\end{eg}

\begin{eg} For the torus with one boundary component, $Q=\Sigma_{1,1}$ a surface bundle 
$\Sigma_g \to (E,\partial E) \to (Q,\partial Q)$ is classified by the monodromy morphism
$$\chi\colon\pi_1(Q)~=~\left\langle x,y,z\,\vert\,[x,y]=z\right\rangle \longrightarrow \Gamma_g~;~x \longmapsto
\alpha~,~y \longmapsto \beta~,~z \longmapsto \gamma$$
with $\gamma=[\alpha,\beta]$, and $E$ isomorphic to the commutator mapping torus $S(\alpha,\beta)$ (Section \ref{se:commutatormappingtorus} below)
and $\partial Q=T(\gamma)$.
\end{eg}

\begin{eg} For the surface with one boundary component, $B=\Sigma_{h,1}$ a surface bundle
$\Sigma_g \to (E,\partial E) \to (\Sigma_{h,1},S^1)$
is classified by the monodromy morphism
$$\begin{array}{l}
\chi\colon\pi_1(\Sigma_{h,1})~=~\left\langle x_1,y_1,\dots,x_h,y_h,z\,\vert\,[x_1,y_1] \dots [x_h,y_h]=z\right\rangle \longrightarrow \Gamma_g~;\\[1ex]
\hskip150pt
x_i \longmapsto \alpha_i~,~y_i \longmapsto \beta_i~,~z \longmapsto \gamma
\end{array}$$
with $[\alpha_1,\beta_1]\dots [\alpha_h,\beta_h]=\gamma \in \Gamma_g$, $E$  isomorphic to the multiple commutator mapping torus 
$S(\alpha_1,\beta_1,\dots,\alpha_h,\beta_h)$  (Section \ref{se:multiplecommutatormappingtorus} below), and $\partial E=T(\gamma)$.
\end{eg}

\subsection{Surface bundles over a circle}\label{se:mappingtorus}

We will now explain in more detail the geometric construction of surface bundles over $S^1$.

\begin{defn}
The \textit{mapping torus} of $\alpha \in \Homeo^+(\Sigma_g)$ is the closed 
3-dimensional manifold 
$$T(\alpha)~=~ \Sigma_g \times [0,1]/\{(x,0) \sim (\alpha(x),1)\,\vert\,x \in \Sigma_g\}~,$$
a surface bundle over $S^1$
$$\Sigma_g \longrightarrow T(\alpha) \longrightarrow S^1~=~[0,1]/\{0 \sim 1\}$$
with monodromy $[\alpha] \in\Gamma_g$.
\end{defn}


\noindent By Proposition \ref{pr:surfacebundles} the function
$$\begin{array}{l}
[S^1,B\Gamma_g] \longrightarrow \{\hbox{isomorphism classes of surface bundles $\Sigma_g \to E \to S^1$}\}~;\\[1ex]
\hskip150pt \alpha \longmapsto (\Sigma_g \to T(\alpha) \to S^1)
\end{array}$$
is a bijection.  A surface bundle $\Sigma_g \to T(\alpha) \to S^1$ extends to
a surface bundle
$$\Sigma_g \longrightarrow (\delta T(\alpha),T(\alpha)) \longrightarrow (D^2,S^1)$$
if and only if $\alpha=1 \in \Gamma_g$, in which case
 $(\delta T(\alpha),T(\alpha))=\Sigma_g \times (D^2,S^1)$ up to isomorphism.

\begin{lemma} \label{le:homeomorphism} Given $\alpha,\beta,\gamma \in \Homeo^+(\Sigma_g)$, the following mapping tori are homeomorphic
\begin{itemize}
\item[(i)] $T(\alpha) ~\cong~ T(\alpha^{-1}),$
\item[(ii)]  $T(\alpha)~\cong~T(\gamma \alpha \gamma^{-1}),$
\item[(iii)]$T(\alpha\beta)~ \cong~ T(\beta\alpha).$
\end{itemize}
\end{lemma}
\begin{proof}
\begin{itemize}
\item[(i)] The map $T(\alpha) \to T(\alpha^{-1}) ; (x, t) \mapsto (x, 1- t)$ is a homeomorphism.
\item[(ii)] The map $T(\alpha) \to T(\gamma \alpha \gamma^{-1}) ; (x, t) \mapsto (\gamma(x), t)$ is a homeomorphism. 
\item[(iii)] The map $T(\alpha \beta) \to T(\beta \alpha) ; (x, t) \mapsto (\beta(x), t)$ is a homeomorphism. Note that (iii) is an immediate consequence of (ii). \qedhere\
\end{itemize}
\end{proof}

\subsection{Surface bundles over a pair of pants}\label{se:doublemappingtorus}

The pair of pants is the oriented surface defined by the 2-sphere with three boundary components
$$\begin{array}{ll}
(P,\partial P)&=~(\Sigma_{0,3},\partial \Sigma_{0,3})\\[1ex]
&=~({\rm cl.}(S^2 \backslash (D^2 \sqcup D^2 \sqcup D^2)),S^1 \sqcup S^1 \sqcup S^1)
\end{array}$$
with $\chi(P)=-1$. The pair of pants $P$ is homotopy equivalent to the figure $8$, $S^1 \vee S^1$, so that
the three inclusions $S^1 \subset \partial P \to P$ induce morphisms 
$$\pi_1(S^1)~=~\bZ \longrightarrow \pi_1(P)~=~\bZ * \bZ~=~\left\langle x,y \right\rangle ~;~ 1_1 \longmapsto x~,~1_2 \longmapsto y~,1_3 \longmapsto xy~.$$


\begin{defn} For any $\alpha,\beta \in \Homeo^+(\Sigma_g)$ let $N_1,N_2,N_3$ be the
three null-cobordisms of $\Sigma_g \times \{0,1\}$ defined by
$$\begin{array}{l}
i_1\colon\Sigma_g \times \{0,1\} \longrightarrow N_1~=~\Sigma_g \times I~;~(x,0) \longmapsto (x,0)~,~(x,1) \longmapsto (x,1)~,\\[1ex]
i_2\colon\Sigma_g \times \{0,1\} \longrightarrow N_2~=~\Sigma_g \times I~;~(x,0) \longmapsto (x,0)~,~(x,1) \longmapsto (\alpha(x),1)~,\\[1ex]
i_3\colon\Sigma_g \times \{0,1\} \longrightarrow N_3~=~\Sigma_g \times I~;~(x,0) \longmapsto (x,0)~,~(x,1) \longmapsto (\beta(x),1).
\end{array}$$ 
The {\it double mapping torus} of $\alpha,\beta\in\Homeo^+(\Sigma_g)$ is the thickening 
$$\begin{array}{ll}
\left(T(\alpha,\beta),\partial T(\alpha,\beta)\right)&=~
\left(W(N_1,N_2,N_3,\Sigma_g),\partial W(N_1,N_2,N_3,\Sigma_g)\right)\\[1ex]
&=~\left(T(\alpha) \times I \cup T(\beta) \times I\cup
T(\alpha\beta) \times I,T(\alpha) \sqcup T(\beta) \sqcup T(\alpha\beta)\right)
\end{array}
$$
which is a surface bundle over the pair of pants
$$\Sigma_g \longrightarrow \left(T(\alpha,\beta),\partial T(\alpha,\beta)\right) \longrightarrow (P,\partial P).$$

\begin{figure}
\labellist
\small\hair 2pt
\pinlabel \small {$T(\alpha)$} at 30 234
\pinlabel \small {$T(\beta)$} at 30 64
\pinlabel \small {$T(\alpha \beta)$} at 394 139
\pinlabel \small {$T(\alpha)\times I$} at 161 194
\pinlabel \small {$T(\beta) \times I$} at 161 88
\pinlabel \small {\rotatebox{45}{$T(\alpha \beta) \times I$}} at 251 139
\pinlabel \tiny {\rotatebox{70}{$\Sigma_g \times I$}} at 210 209
\pinlabel \tiny {\rotatebox{-70}{$\Sigma_g \times I$}} at 208 87
\pinlabel \tiny {$\Sigma_g \times I$} at 145 150
\pinlabel \tiny {$\alpha \cup I$} at 174 129
\pinlabel \tiny {\rotatebox{70}{$1 \cup1$}} at 192 161
\pinlabel \small {$\Sigma_g \cup \Sigma_g$} at 352 246
\endlabellist
\centering
\includegraphics[scale=0.70]{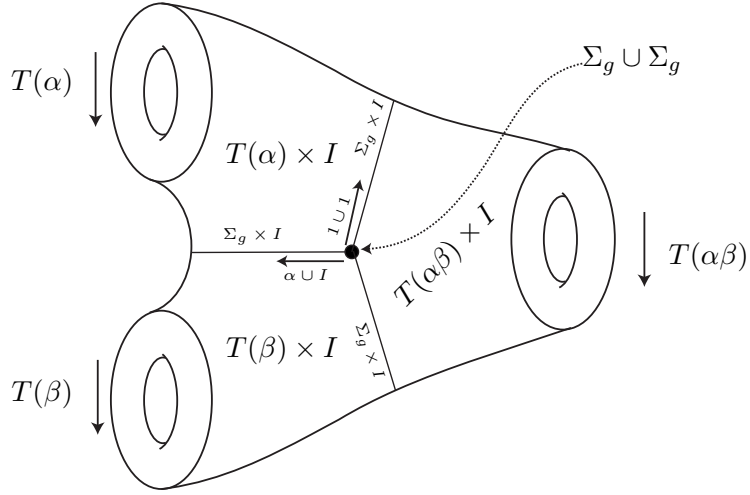}
\caption{The double mapping torus $T(\alpha,\beta)$ as a thickening}
\label{fig:double-mappingtorus-trisection}
\end{figure}

\raggedbottom
\end{defn}

\begin{rk} The double mapping torus can also be expressed as the union of three copies of
$\Sigma_g \times I \times I$, as follows:
$$\begin{array}{l}
T(\alpha, \beta)~=\\[1ex]
(\Sigma_g \times I) \times \{[0, 1/2],1 \}  \cup_{\beta \sqcup Id\colon (\Sigma_g \sqcup \Sigma_g) \times \{[0, 1/2], 1 \} \to(\Sigma_g \sqcup \Sigma_g) \times \{[0, 1/2], 0 \} } (\Sigma_g \times I) \times \{[0, 1/2],0 \} ~\cup \\[1ex]
 (\Sigma_g \times I) \times \{[1/2, 1],1 \}  \cup_{\alpha \beta \sqcup Id\colon (\Sigma_g \sqcup \Sigma_g) \times \{[1/2, 1], 1 \} \to(\Sigma_g \sqcup \Sigma_g) \times \{[1/2, 1], -1 \} } (\Sigma_g \times I) \times \{[1/2, 1],-1 \} ~ \cup \\[1ex]
 (\Sigma_g \times I) \times \{[1/2, 1],0 \}  \cup_{ \alpha \sqcup Id\colon (\Sigma_g \sqcup \Sigma_g) \times \{[1/2, 1], 0 \} \to(\Sigma_g \sqcup \Sigma_g) \times \{[0, 1/2], -1 \} } (\Sigma_g \times I) \times \{[0, 1/2],-1 \}.
\end{array}$$
See Figure \ref{fig:double-mappingtorus-construction} for the construction. The figure also shows how to glue the three pieces together to obtain the desired mapping tori for each of the boundaries.
\begin{figure}
\centering
\includegraphics[scale=0.8]{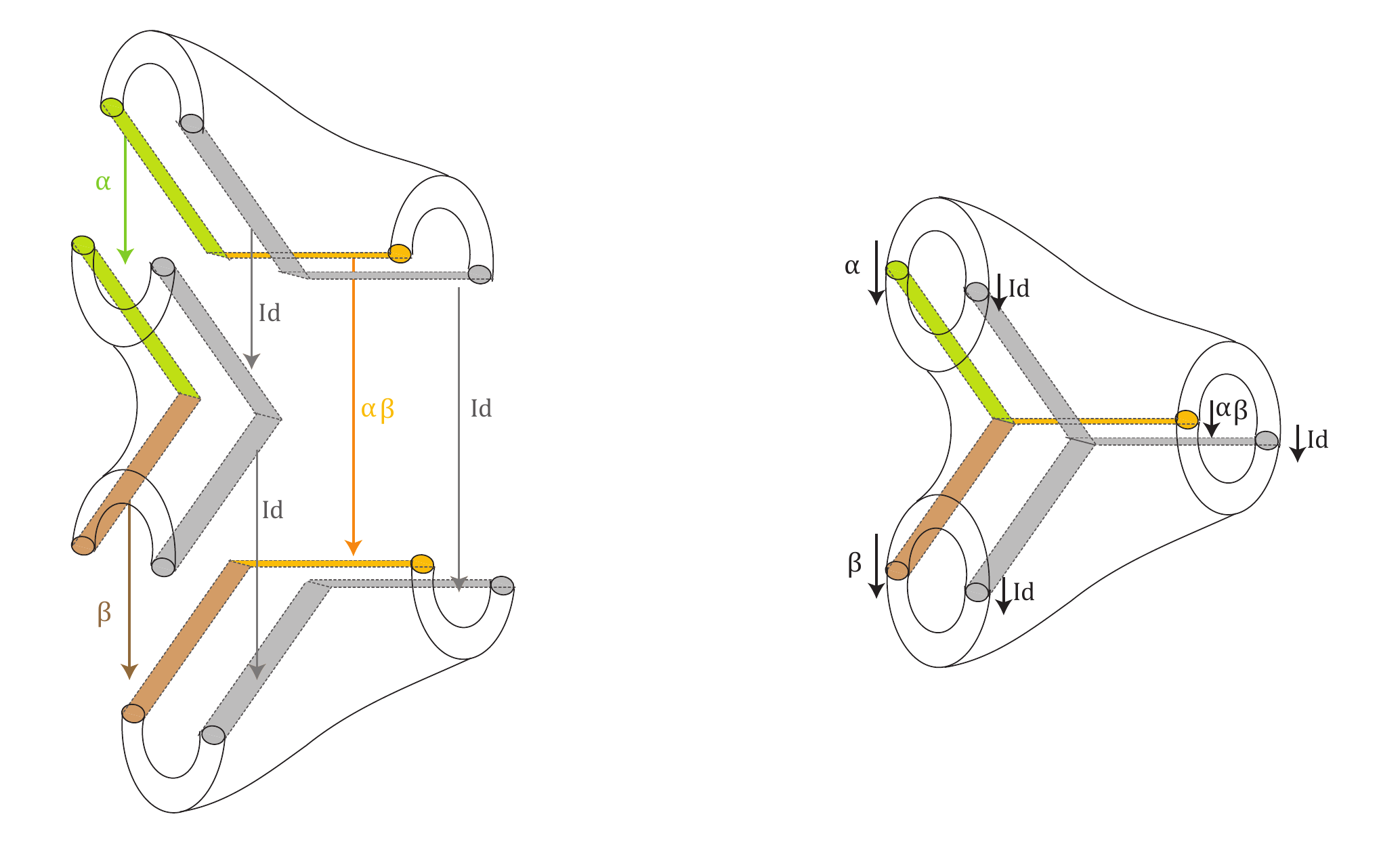}
\caption{The construction of $T(\alpha,\beta)$}
\label{fig:double-mappingtorus-construction}
\end{figure}
\raggedbottom
\end{rk}
\raggedbottom

\noindent By Proposition \ref{pr:surfacebundles} the function
$$\begin{array}{l}
[P,B\Gamma_g]~=~[S^1 \vee S^1,B\Gamma_g] \\[1ex]
 \longrightarrow \{\hbox{isomorphism classes of surface bundles $\Sigma_g \to E \to P$}\}~;\\[1ex]
\hskip150pt (\alpha,\beta) \longmapsto (\Sigma_g \to T(\alpha,\beta) \to P)
\end{array}$$
is a bijection.

\subsection{Surface bundles over a torus with one boundary component}\label{se:commutatormappingtorus}
The torus with one boundary component is the 2-dimensional manifold with boundary
$(Q,\partial Q)~=~(\Sigma_{1,1},S^1)$
%
%
%
with $\chi(Q)=-1$. This manifold $Q$ is homotopy equivalent to the figure eight $S^1 \vee S^1$,
with the inclusion $\partial Q \to Q$ inducing the morphism 
$$\pi_1(\partial Q)~=~\bZ \longrightarrow \pi_1(Q)~=~\bZ*\bZ~=~\left\langle x,y \right\rangle ~;~ 1 \longmapsto [x,y]~.$$

\begin{defn} The {\it commutator double torus} of $\alpha,\beta \in  \Homeo^+(\Sigma_g)$ is the 4-dimensional manifold with boundary  
$$(S(\alpha,\beta),\partial S(\alpha,\beta))~=~(T(\alpha) \times I\cup_{T(\alpha) \sqcup
T(\beta \alpha^{-1} \beta^{-1})}T(\alpha,\beta\alpha^{-1}\beta^{-1}),
T([\alpha,\beta]))~,$$
which is a surface bundle over the torus with one boundary component
$$\Sigma_g \longrightarrow(S(\alpha,\beta),\partial S(\alpha,\beta)) \longrightarrow (Q,\partial Q)~.$$
\end{defn}

\begin{figure}
\labellist
\small\hair 2pt
\pinlabel \tiny {$T(\alpha) \times I$} at 143 121
\pinlabel \tiny {$T(\alpha)$} at 276 189
\pinlabel \tiny {$T(\beta \alpha^{-1} \beta^{-1})$} at 245 56
\pinlabel \tiny {$T([\alpha, \beta])$} at 521 124
\pinlabel \tiny {$T(\alpha, \beta \alpha^{-1} \beta^{-1})$} at 531 221
\endlabellist
\centering
\includegraphics[scale=0.45]{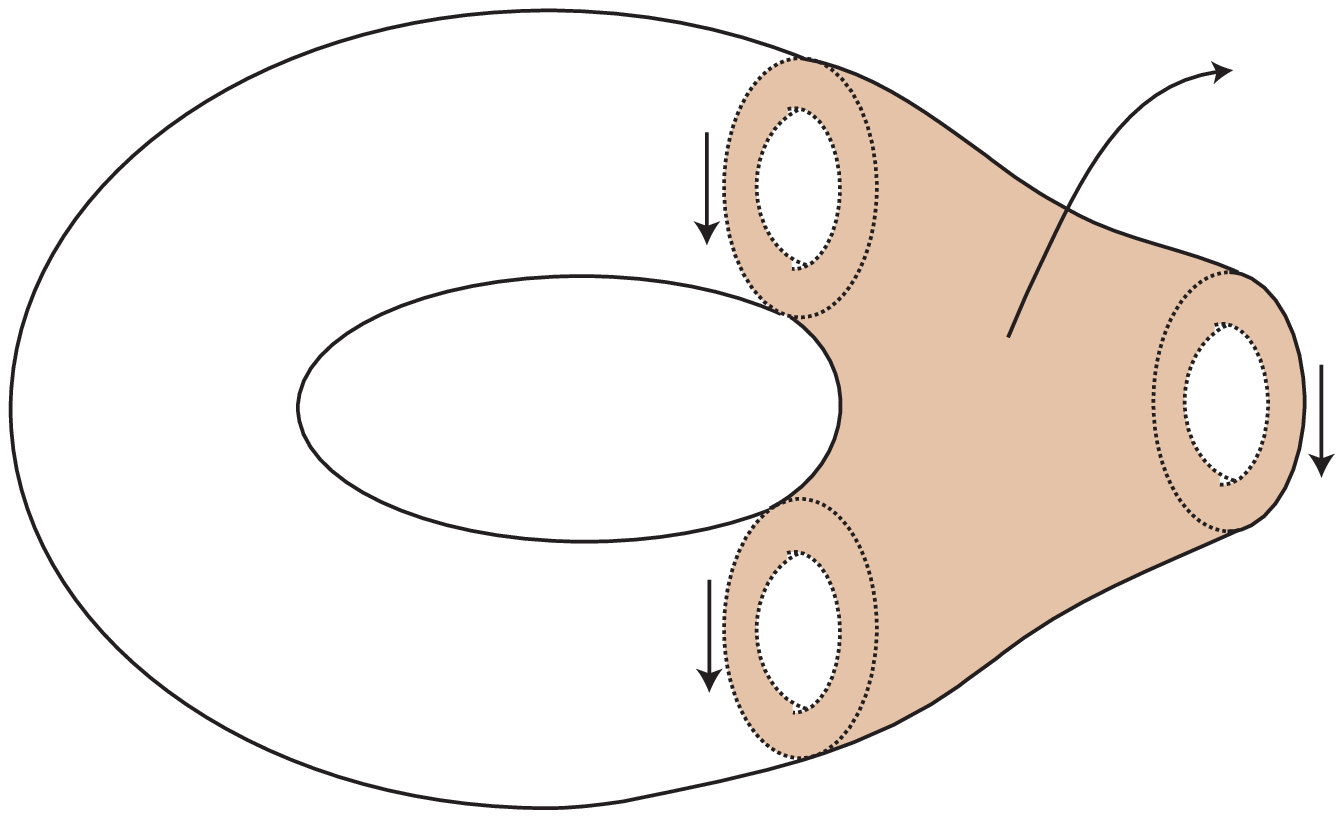}
\caption{The construction of $S(\alpha,\beta)$}
\label{fig:see-monster}
\end{figure}

\noindent By Proposition \ref{pr:surfacebundles} the function
$$\begin{array}{l}
[Q,B\Gamma_g]~=~[S^1 \vee S^1,B\Gamma_g] \\[1ex]
\longrightarrow \{\hbox{isomorphism classes of surface bundles $\Sigma_g \to E \to Q$}\}~;\\[1ex]
\hskip150pt (\alpha,\beta) \longmapsto \left(\Sigma_g \to S(\alpha,\beta) \to Q\right)
\end{array}$$
is a bijection.

\subsection{Surface bundles over a surface}\label{se:multiplecommutatormappingtorus}

Meyer \cite[Satz III. 8.1]{Meyer:1972a} used a decomposition of $\Sigma_{h,k}$ 
along $3h+k-1$ Jordan curves to express the signature $\sigma(E) \in \bZ$ of a surface bundle
$$\Sigma_g \longrightarrow E \longrightarrow \Sigma_{h,k}$$
in terms of the monodromy $\chi\colon\pi_1(\Sigma_{h,k}) \to \Gamma_g$.
In the special case $k=1$ we shall now obtain such an expression for $\sigma(E)$
using $2h-1$ Jordan curves, with $\Sigma_{h,1}$ a union of $h-1$ pairs of pants and $h$ tori with one boundary component, which gives a more direct relationship between the algebra and the topology.

\begin{defn} \label{tau-geo}
The {\it multiple commutator mapping torus} of $\alpha_1,\beta_1,\alpha_2,\beta_2,\dots,\alpha_h,\beta_h \in  \Homeo^+(\Sigma_g)$  with commutators
$\gamma_i=[\alpha_i,\beta_i]$ is the 4-manifold
$$\begin{array}{l}
S(\alpha_1,\beta_1,\alpha_2,\beta_2,\dots,\alpha_h,\beta_h) ~=\\[1ex]
S(\alpha_1,\beta_1) \cup S(\alpha_2,\beta_2) \cup \dots \cup S(\alpha_h,\beta_h)\cup T(\gamma_1,\gamma_2) \cup T(\gamma_1\gamma_2,\gamma_3) \cup \dots \cup T(\gamma_1\gamma_2 \dots \gamma_{h-1},\gamma_h)
 \end{array} $$
 with boundary
 $$\partial S(\alpha_1,\beta_1,\alpha_2,\beta_2,\dots,\alpha_h,\beta_h)~=~
 T(\gamma_1\gamma_2 \dots \gamma_h)$$
 which is a surface bundle
 $$\Sigma_g \longrightarrow \left(S(\alpha_1,\beta_1,\alpha_2,\beta_2,\dots,\alpha_h,\beta_h) ,T(\gamma_1\gamma_2 \dots \gamma_h)\right) \longrightarrow
 (\Sigma_{h,1},S^1)~.$$
See Figure \ref{fig:wee-monster}.

\begin{figure}
\labellist
\small\hair 2pt
\pinlabel \tiny{$S(\alpha_1, \beta_1)$} at 62 154
\pinlabel \tiny {$S(\alpha_2, \beta_2)$} at 145 9
\pinlabel \tiny {$S(\alpha_3, \beta_3)$} at 234 9
\pinlabel \tiny {$S(\alpha_h, \beta_h)$} at 411 9
\pinlabel \tiny {$T(\gamma_1, \gamma_2)$} at 146 170
\pinlabel \tiny {$T(\gamma_1 \gamma_2, \gamma_3)$} at 234 170
\pinlabel \tiny {$T(\gamma_1 \gamma_2 \dots \gamma_{h-1}, \gamma_h)$} at 490 151
\pinlabel \tiny {$T(\gamma_1)$} at 110 242
\pinlabel \tiny {$T(\gamma_1 \gamma_2)$} at 192 242
\pinlabel \tiny {$T(\gamma_1 \gamma_2 \gamma_3)$} at 278 242
\pinlabel \tiny {$T(\gamma_1 \gamma_2 \dots \gamma_h)$} at 455 242
\pinlabel \tiny {$T(\gamma_2)$} at 146 109
\pinlabel \tiny {$T(\gamma_3)$} at 235 112
\pinlabel \tiny {$T(\gamma_h)$} at 407 118
\endlabellist
\centering
\includegraphics[scale=0.67]{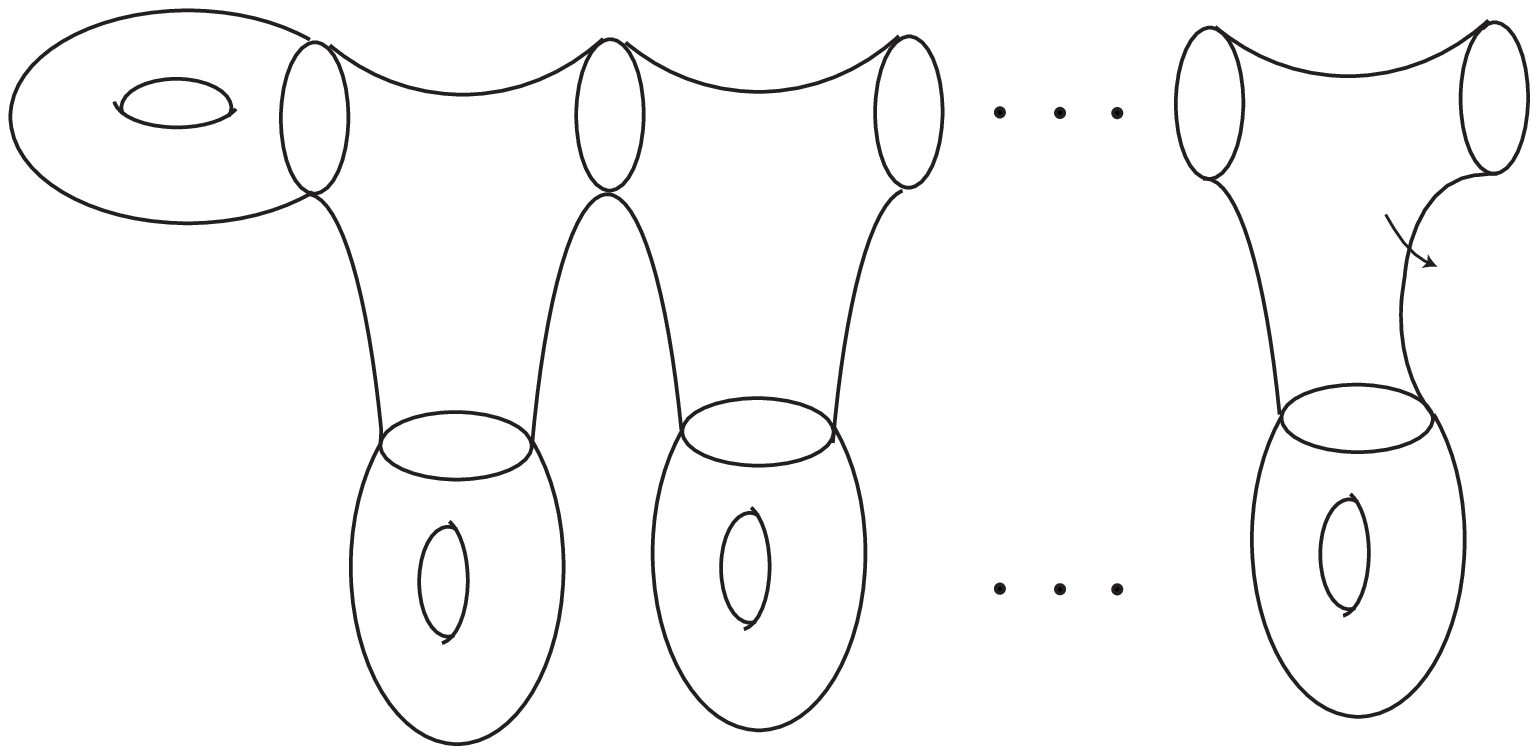}
\caption{The construction of $S(\alpha_1,\beta_1,\alpha_2,\beta_2,\dots,\alpha_h,\beta_h)$}
\label{fig:wee-monster}
\end{figure}

\end{defn}

\begin{prop}
For any expression of $\gamma \in  \Homeo^+(\Sigma_g)$ as a product of commutators 
$$\gamma~=~[\alpha_1,\beta_1][\alpha_2,\beta_2] \dots [\alpha_h,\beta_h] \in  \Homeo^+(\Sigma_g)$$
there is defined a surface bundle over the surface $\Sigma_{h,1}$
$$\Sigma_g \longrightarrow \left(\delta T(\gamma),T(\gamma)\right) \longrightarrow (\Sigma_{h,1},S^1)$$
with
$$\delta T(\gamma)~=~S(\alpha_1,\beta_1,\alpha_2,\beta_2,\dots,\alpha_h,\beta_h)~.$$
The surface bundle over $\Sigma_{h,1}$ extends to a surface bundle over the closed surface $\Sigma_h$
$$\Sigma_g \longrightarrow \delta T(\gamma) \cup_{T(\gamma)} \Sigma_g \times D^2 \longrightarrow \Sigma_{h,1} \cup_{S^1} D^2~=~\Sigma_h$$
if and only if $[\gamma]= 1 \in \Gamma_g$. See Figure \ref{fig:wee-monster2}.
\hfill\qed
\end{prop}

\begin{figure}
\labellist
\small\hair 2pt
\pinlabel \tiny{$S(\alpha_1, \beta_1)$} at 62 154
\pinlabel \tiny {$S(\alpha_2, \beta_2)$} at 145 9
\pinlabel \tiny {$S(\alpha_3, \beta_3)$} at 234 9
\pinlabel \tiny {$S(\alpha_h, \beta_h)$} at 411 9
\pinlabel \tiny {$T(\gamma_1, \gamma_2)$} at 146 170
\pinlabel \tiny {$T(\gamma_1 \gamma_2, \gamma_3)$} at 234 170
\pinlabel \tiny {$T(\gamma_1 \gamma_2 \dots \gamma_{h-1}, \gamma_h)$} at 490 151
\pinlabel \tiny {$\Sigma_g \times D^2$} at 485 213
\pinlabel \tiny {$T(\gamma_1)$} at 110 242
\pinlabel \tiny {$T(\gamma_1 \gamma_2)$} at 192 242
\pinlabel \tiny {$T(\gamma_1 \gamma_2 \gamma_3)$} at 278 242
\pinlabel \tiny {$T(\gamma_1 \gamma_2 \dots \gamma_h = 1)$} at 420 252
\pinlabel \tiny {$T(\gamma_2)$} at 146 109
\pinlabel \tiny {$T(\gamma_3)$} at 235 112
\pinlabel \tiny {$T(\gamma_h)$} at 407 118
\endlabellist
\centering
\includegraphics[scale=0.67]{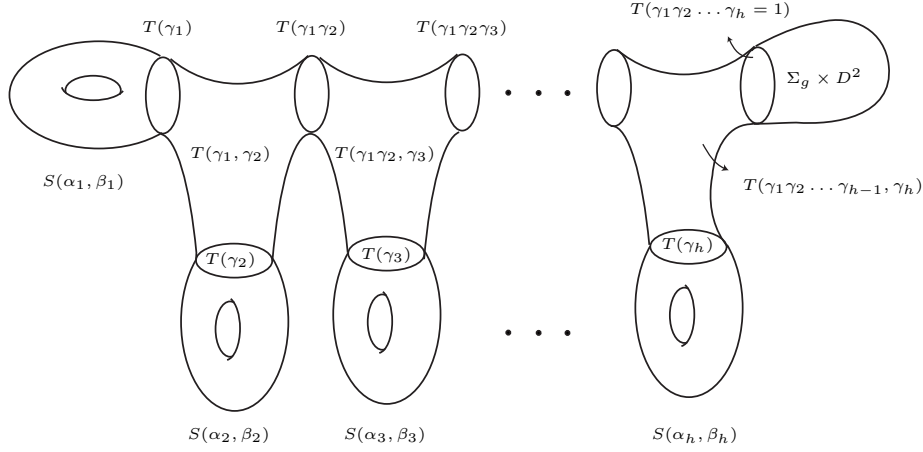}
\caption{The construction of $\delta T(\gamma) \cup_{T(\gamma)} \Sigma_g \times D^2$}
\label{fig:wee-monster2}
\end{figure}

Recall that from Figure \ref{fig:wee-monster} onwards $\gamma_i = [\alpha_i, \beta_i]$. By abuse of notation, in the computation below we will also write $\alpha_i, \gamma_i$ for their image in $\Gamma_g$ under the canonical morphism $\Homeo^+(\Sigma_g)\rightarrow \Gamma_g$.

To compute the signature of the surface bundle $E = \delta T(\gamma) \cup_{T(\gamma)} \Sigma_g \times D^2$ we use the canonical projection $\rho\colon \Gamma_g \rightarrow \Sp(2g, \bZ)$ and Novikov additivity of the signature, together with the fact that $\sigma(\Sigma_g \times D^2) =0$  and the decomposition of  each $S(\alpha_i, \beta_i)$ as the union of $T(\alpha_i) \times I$  and $T(\alpha, \beta \alpha^{-1}\beta^{-1})$.
\begin{align*} 
\sigma(E) & = \sigma \left(\delta T(\gamma) \cup_{T(\gamma)} \Sigma_g \times D^2\right) \\
            & = \sigma\left(\delta T(\gamma)\right) + \sigma\left(\Sigma_g \times D^2\right) \\
           & = \sigma\left(S(\alpha_1,\beta_1) \cup S(\alpha_2,\beta_2) \cup \dots \cup S(\alpha_h,\beta_h)\cup T(\gamma_1,\gamma_2)  \cup \dots \cup T(\gamma_1\gamma_2 \dots \gamma_{h-1},\gamma_h)\right) +0 \\      
           & = \sigma\left(\cup_{i=1}^h S(\alpha_i, \beta_i)\right) +   \sigma\left(T(\gamma_1,\gamma_2) \right)  + \dots + \sigma \left(T(\gamma_1\gamma_2 \dots \gamma_{h-1},\gamma_h)\right) \\
           & = \sigma\left(\cup_{i=1}^h (T(\alpha_i) \times I \cup T(\alpha_i, \beta_i \alpha^{-1}_i \beta^{-1}_i)\right) +   \sigma\left(T(\gamma_1,\gamma_2)\right)  + \dots +\sigma\left(T(\gamma_1\gamma_2 \dots \gamma_{h-1},\gamma_h)\right) \\
           & = -\sum_{i=1}^h \tau_g\left(\rho(\gamma_i), \rho(\alpha_i)\right) -  \sum_{i=1}^{h-1} \tau_g\left(\rho(\widetilde{\gamma}_{i}), \rho(\gamma_{i+1})\right), ~~ \textnormal{ where }  \widetilde{\gamma}_{i} = \gamma_1 \gamma_2 \dots \gamma_i.          
\end{align*}
Note that this expression coincides, up to differences in notation, with the expression in \cite[Satz 1, (14*)]{Meyer:1973a}.

\subsection{Geometric cocycle}\label{se:geometriccocycle}

Let $\Omega_n$ be the $n$-dimensional oriented cobordism group. It is well-known that 
$$\Omega_0~=~\bZ~,~\Omega_1~=~\Omega_2~=~\Omega_3~=~0$$
and that the signature defines an isomorphism
$$\sigma\colon\Omega_4 \longrightarrow \bZ~;~ M \longmapsto \sigma(M).$$

\begin{prop}\label{cobordism cocycle}
{\rm (i)} The mapping class group $\Gamma_g$ fits into a 
geometric central extension
$$1\longrightarrow\Omega_{4} \longrightarrow \widetilde{\Gamma}_g \longrightarrow \Gamma_g\longrightarrow 1$$
with
%
$$\begin{array}{l}
\widetilde{\Gamma}_g~=~\{(M=\textnormal{ $4$-dimensional manifold},\alpha \in \Gamma_g)\,\vert\,\partial M=T(\alpha)\}/\sim~,\\[1ex]
(M,\alpha) \sim (M',\alpha')~\textnormal{if $\alpha=\alpha'\in \Gamma_g$ and
$M\cup_{T(\alpha)}-M'=0 \in\Omega_{4}$}~,\\[1ex]
1~=~(\Sigma_g \times D^2,1:\Sigma_g \to \Sigma_g)~,\\[1ex]
-(M,\alpha)~=~(-M,\alpha^{-1})~{\rm via}~-T(\alpha^{-1}) \cong T(\alpha)~;~[x,t] \mapsto [\alpha(x),1-t]~,\\[1ex] 
(M_1,\alpha)(M_2,\beta)=((M_1 \sqcup M_2)\cup_{T(\alpha) \sqcup T(\beta)}-T(\alpha,\beta),\alpha\beta)~,\\[1ex]
\widetilde{\Gamma}_g \longrightarrow \Gamma_g~;~(M,\alpha) \longmapsto \alpha~,\\[1ex]
\Omega_{4} \longrightarrow\widetilde{\Gamma}_g~;~N \longmapsto
(N \sqcup \Sigma_g \times D^2,1:\Sigma_g \to \Sigma_g)~.
\end{array}$$

{\rm (ii)} A section $\Gamma_g \to \widetilde{\Gamma}_g; \alpha \mapsto \delta T(\alpha)$ determines a cocycle
$$\tau^{geo}\colon\Gamma_g \times \Gamma_g \longrightarrow \Omega_{4}~;~
(\alpha,\beta) \longmapsto T(\alpha,\beta)\cup \delta T(\alpha) \cup \delta T(\beta) \cup \overline{\delta T(\alpha\beta)}$$
for the class $[\tau^{geo}] \in H^2(\Gamma_g;\Omega_{4})$.
\end{prop}

\begin{proof}
{\rm (i)} Note that since since  $\Omega_{3}=0$, and the mapping tori $T(\alpha)$, $T(\beta)$ and $T(\alpha \beta)$ are $3$-dimensional closed manifolds, then the nullbordisms $\delta T(\alpha)$, $\delta T(\beta)$ and $\delta T(\alpha \beta)$ always exist.

{\rm (ii)} By definition, for $G$ a group and $A$ an abelian group, a \emph{cocycle}\index{cocycle} is a function $\tau:G \times G \to A$ such that
$$\tau(x,y)+\tau(xy,z)=\tau(y,z)+\tau(x,yz) \in A~(x,y,z \in G)~.$$

In order to see that $\tau^{geo}$ described above is indeed a cocycle, we start by noting that there is a $4$-dimensional manifold with boundary the disjoint union of mapping tori $T(\alpha) \sqcup T(\beta) \sqcup T(\alpha \beta \gamma) \sqcup T(\gamma)$.
This manifold can be described either as the union of the double mapping tori $T(\alpha,\beta) \cup T(\alpha\beta,\gamma)$ or as $T(\beta,\gamma) \cup T(\alpha,\beta\gamma)$, as depicted geometrically in Figure \ref{fig:cocycle-condition}.

\begin{figure}
\labellist
\small\hair 2pt
\pinlabel \tiny {$\alpha$} at 94 133
\pinlabel \tiny {$\beta$} at 94 33
\pinlabel \tiny {$\alpha \beta \gamma$} at 244 133
\pinlabel \tiny {$\gamma$} at 250 33
\pinlabel \tiny {$\alpha$} at 365 133
\pinlabel \tiny {$\beta$} at 365 33
\pinlabel \tiny {$\alpha \beta \gamma$} at 516 133
\pinlabel \tiny {$\gamma$} at 521 33
\pinlabel \tiny {$\alpha \beta$} at 168 83
\pinlabel \tiny {$\beta \gamma$} at 441 83
\pinlabel \tiny {$T(\alpha \beta)$} at 132 99
\pinlabel \tiny {$T(\alpha \beta, \gamma)$} at 208 99
\pinlabel \tiny {$T(\alpha, \beta \gamma)$} at 445 127
\pinlabel \tiny {$T(\beta, \gamma)$} at 445 46
\endlabellist
\centering
\includegraphics[scale=0.60]{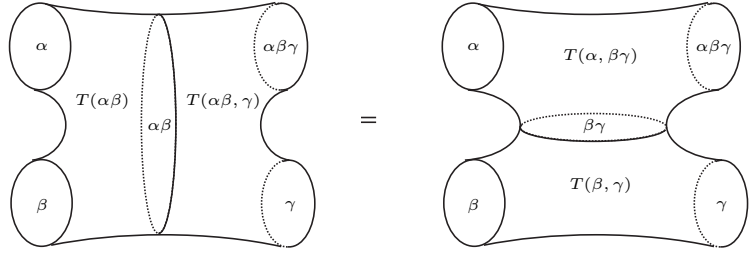}
\caption{$T(\alpha,\beta) \cup T(\alpha\beta,\gamma) = T(\beta,\gamma) \cup T(\alpha,\beta\gamma)$}
\label{fig:cocycle-condition}
\end{figure}

Note also that 
\begin{equation}\label{sigma-geococycle}
\sigma\left(T(\alpha,\beta)\right) + \sigma\left(T(\alpha\beta,\gamma)\right) = \sigma\left(T(\beta,\gamma) \right) + \sigma\left(T(\alpha,\beta\gamma)\right).
\end{equation}

Now in order for $\tau^{geo}$ to be a cocycle, the following identity has to be satisfied,
\begin{equation}\label{taugeo} \tau^{geo}(\alpha, \beta)+\tau^{geo}(\alpha \beta,\gamma)=\tau^{geo}(\beta,\gamma)+\tau^{geo}(\alpha,\beta \gamma).
\end{equation}
In the context of Equation \eqref{taugeo}, addition should be interpreted as disjoint union of manifolds and the equal sign means that the two sides belong to the same bordism class in $\Omega_4$. Since there is an isomorphism $\sigma \colon \Omega_4 \to \bZ$, to check that these manifolds belong to the same bordism class we will only need to check that they have the same signature. That is,
\begin{align*}
\sigma\left(T(\alpha,\beta)\cup \delta T(\alpha) \cup \delta T(\beta) \cup \overline{\delta T(\alpha \beta)}\right)  + \sigma\left(T(\alpha \beta,\gamma)\cup \delta T(\alpha \beta) \cup \delta T(\gamma) \cup \overline{\delta T(\alpha \beta \gamma)}\right) \\
= \sigma\left(T(\beta,\gamma)\cup \delta T(\beta) \cup \delta T(\gamma) \cup \overline{\delta T(\beta \gamma)}\right)  + \sigma\left(T(\alpha, \beta\gamma)\cup \delta T(\alpha ) \cup \delta T(\beta \gamma) \cup \overline{\delta T(\alpha \beta \gamma)}\right).
\end{align*}

Using Novikov additivity of the signature and Equation \eqref{sigma-geococycle}, we see that this identity holds, and hence $\tau^{geo}$ is a cocycle.
\end{proof}

\section{Cohomology of Lie groups made discrete}\label{se:lie}

For any discrete group $G$ the group cohomology $H^2(G;\bZ)$ is isomorphic 
to the degree two cohomology $H^2(BG;\bZ)$ of the classifying space $BG$.
It classifies central group extensions\index{central extension}
$$ 1 \longrightarrow \bZ \longrightarrow \tilde G \longrightarrow G \longrightarrow 1. $$
A cocycle $\tau:G \times G \to \bZ$ determines a central group extension
$$1 \longrightarrow \bZ \longrightarrow \bZ\times_{\tau} G \longrightarrow G \longrightarrow 1$$
with
$$\bZ \times_\tau G=\{(m,a)\,\vert\,m \in \bZ,a \in G\}~,~
(m,a)(n,b)=(m+n+\tau(a,b),ab)~.$$

A (finite dimensional) Lie group $G$ has a classifying space\index{classifying space} 
$BG$ as a topological group. The singular cohomology $H^2(BG;\bZ)$ classifies 
covering groups\index{covering group}
$$ 1 \longrightarrow \bZ \longrightarrow \tilde G \longrightarrow G \longrightarrow 1 $$
where $\tilde G$ is again a Lie group.

We write $G^\delta$ 
for the same Lie group, but with the discrete topology.
Following Milnor \cite{Milnor:1983a} 
(see also Thurston \cite{Thurston:1974a}) we write $\bar G$ for
the homotopy fibre of $G^\delta \to G$. The goal of this section
is to recap some of what is known about the natural map 
$H^2(BG;\bZ)\to H^2(BG^\delta;\bZ)$.

\begin{eg}
Let $G=\bR$, the additive group of the real numbers. Since $G$ is 
contractible we have $\bar G = G^\delta$. \\
\indent Now $\bR$ is a $\bQ$-vector space of dimension equal to the 
cardinality of the continuum.
Since $\bQ$ is a filtered colimit of rank one free groups, 
we have $H_1(B\bQ^\delta)\cong \bQ$ and $H_i(B\bQ^\delta)=0$ 
for $i>1$. The K\"unneth theorem then gives 
$H_i(B\bR^\delta)\cong\Lambda^i_\bQ(\bR)$, the $i$th exterior
power of the reals over the rationals. This is a rational 
vector space of dimension equal to the continuum. 
By the universal coefficient theorem, 
$$ H^i(B\bR^\delta;\bZ) \cong \Ext(\Lambda^{i-1}_\bQ(\bR),\bZ), $$ 
which is again a large rational vector space (note that $\Ext(\bQ,\bZ)$ 
is already an uncountable dimensional rational vector space).
\end{eg}
\begin{eg}
Let $G=\UU(1)\cong S^1$. The universal cover of $G$ is $\bR$, so
we have a pullback square
$$ \xymatrix{\bR^\delta \ar[r] \ar[d] & \bR \ar[d] \\
G^\delta \ar[r] & G.} $$
Since $\bR$ is contractible, this implies that $\bar G=\bR^\delta$.
Thinking of $\UU(1)$ as a $K(\bZ,1)$, we have a fibre sequence
$$ K(\bZ,1)\longrightarrow K(\bR^\delta,1) \longrightarrow K(G^\delta,1) \longrightarrow K(\bZ,2). $$
We have $H_1(BG^\delta) = \bR/\bZ$, $H_1(B\bR^\delta)=\bR^\delta$.
The Gysin sequence of the fibration $S^1 \to K(\bR^\delta,1) \to K(G^\delta,1)$ is
\[ \dots  \longrightarrow H_1(BG^\delta) \longrightarrow H_2(B\bR^\delta) \longrightarrow H_2(BG^\delta)
\longrightarrow H_0(BG^\delta) \longrightarrow H_1(B\bR^\delta) \longrightarrow H_1(BG^\delta) \longrightarrow 0 \]
and therefore takes the form
\[ \dots \longrightarrow \bR/\bZ \longrightarrow  \Lambda^2_\bQ(\bR) \longrightarrow H_2(BG^\delta) \longrightarrow \bZ \longrightarrow \bR^\delta \longrightarrow G^\delta \longrightarrow 0. \]
Since there are no non-trivial homomorphisms from $\bR/\bZ$ to a
rational vector space, it follows that
\[ H_2(BG^\delta) \cong \Lambda^2_\bQ(\bR). \]
In cohomology, we have $H^1(BG^\delta;\bZ)=0$ and
\[ H^2(BG^\delta;\bZ) \cong \Ext(\bR/\bZ,\bZ) \cong \bZ \oplus \Ext(\bR,\bZ), \]
a direct sum of $\bZ$ and a large rational vector space.
\end{eg}

More generally, it is shown in Sah and Wagoner \cite[p. 623]{Sah/Wagoner:1977a} that if $G$ is a simply connected Lie group such that
the simply connected composition factors of $G$ are either $\bR$
or isomorphic to universal covering groups of Chevalley groups\index{Chevalley groups}  
over $\bR$ or $\bC$ then the integral homology $H_2(BG^\delta)$
is a $\bQ$-vector space of dimension equal to the continuum.

\begin{eg}
Let us examine the real Chevalley group $\Sp(2g,\bR)$. In this case, 
the universal cover is given by 
$$ 1 \longrightarrow \bZ \longrightarrow \widetilde{\Sp(2g,\bR)} \longrightarrow \Sp(2g,\bR) \longrightarrow 1. $$
Thus
$$ H^2(B\Sp(2g,\bR)^\delta;\bZ) \cong H^2(B\Sp(2g,\bR);\bZ) 
\oplus H^2(B\widetilde{\Sp(2g,\bR)^\delta};\bZ) $$
is a direct sum of $\bZ$ and a $\bQ$-vector space of dimension equal to
the continuum.
\end{eg}

In light of the size of $H^2(BG^\delta;\bZ)$, we clearly need to
restrict the kind of cocycles we should consider, in order to identify
the central extensions which are also covering groups;~ namely those
that are in the image of $H^2(BG;\bZ)\to H^2(BG^\delta;\bZ)$.
It is clearly no use trying to restrict to continuous cocycles;~ for example
if $G$ is connected then the only continuous cocycles are the constant ones.
So what should we try?
This problem was solved by Mackey \cite{Mackey:1957a}, as follows.

Recall that given a topological space, the Borel sets\index{Borel set} are the smallest
collection of subsets containing the open sets, and closed under
complementation and arbitrary unions. A map is a Borel map if the
inverse image of every Borel set is a Borel set.
Theorem~7.1 of Mackey's paper \cite{Mackey:1957a}
implies that under reasonably general conditions, given just the Borel sets and a group structure on $G$ consisting of Borel maps,
there is a unique structure of locally compact topological group on $G$
for which these are the Borel sets and group structure. The proof
depends on Weil's converse to Haar's theorem on measures, 
described in Appendix 1 of Weil \cite{Weil:1940a}. In particular, the
topology can be obtained by taking for the neighbourhood system of the identity
element in $G$ the family of all sets of the form $A^{-1}A$, where $A$ is a Borel
set of positive measure.

The consequence of Mackey's theorem which we wish to describe is as follows.
Given a cocycle $G\times G \to \bZ$ which is also a Borel map then
consider the covering group
$$ 1 \longrightarrow \bZ \longrightarrow \tilde G \longrightarrow G \longrightarrow 1. $$ 
The group $\tilde G$ inherits a Borel structure, as the underlying set is $G \times \bZ$.
This is compatible with the group structure, and so by Mackey's theorem
$\tilde G$ is a topological group in a unique way. Uniqueness of the
structure of topological group on $G$ implies that the topological group $\tilde G/\bZ$ is
isomorphic to $G$, and therefore the map $\tilde G \to G$ is continuous. 
It follows that $\tilde G$ is a Lie group.

\begin{defn} The {\it Borel cohomology} $H^2_B(G;\bZ)$ is the abelian group of Borel cocycles on $G^\delta$ modulo coboundaries of Borel cochains $G^\delta\to \bZ$.
\end{defn}

Milnor \cite[Theorem 1]{Milnor:1983a} proved that for any Lie group $G$ with finitely many connected components the natural 
map $H^2(BG;\bZ)\to H^2(BG^\delta;\bZ)$ is injective.

\begin{theorem}\label{th:Borel}
For any Lie group $G$ with finitely many connected components, the image of the injection $H^2(BG;\bZ)\to H^2(BG^\delta;\bZ)$ is
equal to the image of $H^2_B(G;\bZ) \to H^2(G^\delta;\bZ)$.
\end{theorem}
\begin{proof} The proof is immediate by the commutativity of the following diagram:
$$ \xymatrix{H^2(BG;\bZ) \ar[r] \ar[d]^\cong & H^2(BG^\delta;\bZ) \ar[d]^\cong \\
H^2_B(G;\bZ) \ar[r] & H^2(G^\delta;\bZ).} $$
This can be found in \cite[below Theorem 2.1]{ChatterjiMislindeCornulierPittet:2013a}. It follows from work of Wigner \cite{wigner}, and can be best understood by combining \cite[Theorem 10]{moore} (see also \cite[Theorem 38]{ChatterjiMislinPittetSaloff-Coste}), and \cite[p. 1518]{baaj-skandalis-vaes}.
\end{proof}


\section{Cocycles on $S^{1\delta}$}\label{se:Cocycles on S1}

We are concerned in this section with certain Borel cocycles $\tau\colon S^{1\delta}\times S^{1\delta} \to \bZ$
on the (discrete) circle group ${S^{1\delta}}$, and the covering groups they define. We shall think  of $S^{1\delta}$ as the unit interval $[0,1]$ with the
endpoints identified, using the group isomorphism
$$[0,1]/(0 \sim 1)\xymatrix{\ar[r]^-{\cong}&} S^{1\delta}~~;~~a \longmapsto e^{2\pi ia}$$
as an identification. 

The ``standard cocycle''\index{standard cocycle} is given by
$$ \tau \colon[0,1) \times [0,1) \longrightarrow \bZ~~;~~(a,b) \longmapsto  \tau(a,b) = \begin{cases} 0 & \text{if } 0 \leqslant a+b < 1, \\ 1 & \text{if } 1 \leqslant a+b < 2. \end{cases}  $$
This defines the extension
$$ 1 \longrightarrow \bZ \longrightarrow \bR^\delta \longrightarrow S^{1\delta} \longrightarrow 1. $$
We shall say that a $1$-cochain $f$ is \emph{nice} if it is piecewise constant. In other words, 
$S^1$ is divided into a finite number of intervals with open or closed ends, on 
each of which $f$ is constant. A coboundary $\tau$ is \emph{nice}
if it is the coboundary of a nice $1$-cochain, and a $2$-cocycle is
\emph{nice}\index{nice cocycle} if it is the sum of a nice 
coboundary and an integer multiple $m$ of the standard cocycle. Every nice $2$-cocycle defines a topological covering 
group of $S^1$, and by Theorem \ref{th:Borel} and Proposition \ref{pr:UgSp2gR} below, the quotient of the nice cocycles by the
nice coboundaries gives $H^2_B(S^1;\bZ) \cong \bZ$, with generator the class $[\tau]$ of the standard cocycle. 
The integer $m$, which we call the \emph{covering number}\index{covering number} of the cocycle $\tau$, is thus well defined. If $m=0$ then the
covering is trivial:
$$ 1 \longrightarrow \bZ \longrightarrow S^{1\delta} \times \bZ \longrightarrow S^{1\delta} \longrightarrow 1, $$
while if $m\ne 0$, the covering takes the form
$$ 1 \longrightarrow \bZ \longrightarrow \bR^\delta \times_{m\tau} \bZ/m \longrightarrow S^{1\delta}\longrightarrow 1, $$
with the group structure on $\bR^\delta\times_{m\tau}\bZ/m=\bR^\delta\times\bZ/m$ given by
$$(\bR^\delta\times_{m\tau}\bZ/m) \times (\bR^\delta\times_{m\tau}\bZ/m) \longrightarrow\bR^\delta\times_{m\tau}\bZ/m;((a,x),(b,y)) \longmapsto (a+b+m\tau(x,y),xy)~.$$

We can draw a picture of a nice cocycle $\tau$ in an obvious way, by dividing 
$[0,1]\times [0,1]$ into subregions where $\tau$ is constant.

\begin{eg} The standard cocycle gives the following picture:
{\small
\begin{center}
\begin{picture}(80,80)
\put(10,10){\line(1,0){60}}
\put(10,70){\line(1,0){60}}
\put(10,10){\line(0,1){60}}
\put(70,10){\line(0,1){60}}
\put(10,70){\line(1,-1){60}}
\put(25,25){$0$}
\put(50,50){$1$}
\end{picture}
\end{center}
}
\noindent
This is a Borel cocycle on $S^{1\delta}$ because it is piecewise constant on Borel subsets.
\end{eg}

\begin{theorem}\label{th:S1cocycle}
The picture of a nice cocycle has only horizontal, vertical and leading diagonal boundary lines. The covering number
is independent of the values on the horizontal and vertical lines, i.e. can be
computed from the values on diagonal lines only. For each diagonal
line segment, we compute 
\[
d=(a-b)\times r,
\]
where 
\begin{itemize}
\item $a=\mbox{value of cocycle above the line segment}$,
\item $b=\mbox{value of cocycle below the line segment}$,
\item $r=\mbox{ratio of the length of the line segment to the length of the main diagonal}$.
\end{itemize}
Adding these quantities $d$ over all diagonals gives the covering number.
\end{theorem}
\begin{proof}
The quantity described in the theorem is additive on nice cocycles, so we
only need to check the theorem on the standard cocycle and on coboundaries.
For the standard cocycle, there is just one diagonal line segment, and the difference
in the values from below to above the line segment is one, so the theorem is true in
this case. For a nice coboundary $\delta f$, the diagonal lines happen where $f(a+b)$ 
changes in value. For each point in the unit interval where $f$ changes value, the
total normalised length of the one or two diagonal line segments representing the change
in value of $f(a+b)$ is equal to one. So we are adding the changes in value of $f$ as
$f$ goes once round $S^1$. The total is therefore zero.
\end{proof}

Notice that in this theorem, the values of $\tau$ on the boundaries of regions
are irrelevant to the computation of the covering number.

\begin{eg}\label{eg:two}
For any $p,q \in \bZ$ define the nice cochain
$$f(p,q)\colon[0,1) \longrightarrow \bZ~~;~~a \longmapsto \begin{cases} p&{\rm if}~0 \leqslant a < 1/2,\\
 q&{\rm if}~1/2 \leqslant a <1\phantom{,}
 \end{cases}$$
 with nice coboundary
 $$ \delta f(p,q)\colon[0,1) \times [0,1) \longrightarrow \bZ~~;~~
 (a,b) \longmapsto  f(a)+f(b)-f(a+b)~.$$
  For any $m \in \bZ$ the picture of the nice cocycle $m(\text{standard})+\delta f(p,q)$ has 8 regions
 $$\begin{tikzpicture}[scale=1.5]
\draw (0,0) -- (4,0) -- (4,4) -- (0,4) -- (0,0);
\draw (0,4) -- (4,0);
\draw (2,0) -- (2,4);
\draw (0,2) -- (4,2);
\draw (2,4) -- (4,2);
\draw (0,2) -- (2,0);
\draw (1,0.5) node{$p$};
\draw (1,1.5) node{$2p-q$};
\draw (1,2.5) node{$p$};
\draw (1,3.5) node{$m+q$};
\draw (3,0.5) node{$p$};
\draw (3,1.5) node{$m+q$};
\draw (2.7,2.5) node{$m-p+2q$};
\draw (3,3.5) node{$m+q$};
\end{tikzpicture}$$
with covering number
$$\textstyle{\frac{1}{2}}((2p-q)-p)+(m+q-p)+\textstyle{\frac{1}{2}}((m+q)-(m-p+2q))~=~ m~.$$ 
\end{eg}

\begin{eg}\label{eg:MeyerS1}
We illustrate the theorem with the two cases that will be of interest to us 
in understanding the Meyer cocycle $\tau_1$ of Definition \ref{df:Meyercocycle} for $g=1$
(Example \ref{eg:tau1})
$$\begin{array}{l}
\tau_1\colon [0,1) \times [0,1) \longrightarrow \bZ~;\\
(a,b) \longmapsto -2\sign\left(\sin(\pi a)\sin(\pi b)\sin(\pi(a+b))\right)~=~\begin{cases} -2&{\rm if}~
0< a+b<1,\\
\phantom{-}0&{\rm if}~a, b~{\rm or }~a+b=0, 1,\\
\phantom{-}2&{\rm if}~1<a+b<2
\end{cases}
\end{array}$$
and the Maslov cocycle $\tau'_1$ of Definition \ref{df:Maslovcocycle} for $g=1$ (Example \ref{eg:tauprime1})
$$\tau'_1\colon[0,1) \times [0,1) \longrightarrow \bZ~~;~~(a,b) \longmapsto
-\sign\left(\sin(2\pi a)\sin(2\pi b)\sin(2\pi(a+b))\right)~.$$
The diagrams for these are as follows:

{\small
\begin{center}
\begin{picture}(80,80)
\put(10,10){\line(1,0){60}}
\put(10,70){\line(1,0){60}}
\put(10,10){\line(0,1){60}}
\put(70,10){\line(0,1){60}}
\put(10,70){\line(1,-1){60}}
\put(22,24){$-2$}
\put(50,50){$2$}
\end{picture}
\qquad\qquad
\begin{picture}(80,80)
\put(10,10){\line(1,0){60}}
\put(10,40){\line(1,0){60}}
\put(10,70){\line(1,0){60}}
\put(10,10){\line(0,1){60}}
\put(40,10){\line(0,1){60}}
\put(70,10){\line(0,1){60}}
\put(10,40){\line(1,-1){30}}
\put(10,70){\line(1,-1){60}}
\put(40,70){\line(1,-1){30}}
\put(13,14){$-1$}
\put(27.5,27.5){$1$}
\put(43,44){$-1$}
\put(57.5,57.5){$1$}
\put(43,14){$-1$}
\put(57.5,27.5){$1$}
\put(13,44){$-1$}
\put(27.5,57.5){$1$}
\end{picture}
\end{center}
}

\noindent In the terminology of Example \ref{eg:two}
$$\begin{array}{l}
\tau_1~=~4({\rm standard})+\delta f(-2,-2)~,\\[1ex]
\tau'_1~=~4({\rm standard})+\delta f(-1,-3)~,
\end{array}$$
so that the covering number is equal to $4$ in both cases.  Now
$$f(1,-1)(a)~=~\sign (\sin 2 \pi a)~,$$
so that the cocycles differ by the coboundary
$$\begin{array}{ll}
\tau_1(a,b)-\tau'_1(a,b)&=~-\delta f(1,-1)(a,b)\\[1ex]
&=~-\sign(\sin 2\pi a)-\sign(\sin 2\pi b)+\sign(\sin 2\pi(a+b))~.
\end{array}$$
and
$$[\tau_1]~=~[\tau'_1]~=~4 \in H^2_b(S^{1};\bZ)~=~\bZ~.$$
(See also Example \ref{eg:tauprime1} and Remark \ref{gilmer-masbaum}.)
\end{eg}


\section{The symplectic group $\Sp(2g,\bR)$}\label{se:Sp2gR}

In this section, we examine the homotopy type and cohomology of
$\Sp(2g,\bR)$, and use this to compare the cocycles described in
Sections \ref{se:forms} and \ref{se:Cocycles on S1}.


\begin{theorem}\label{th:max-compact}
Let $G$ be a Lie group with finitely many connected components, and let $K$
be a maximal compact subgroup.\index{maximal compact subgroup} 
Then as a topological space,
$G$ is homeomorphic to a Cartesian product $K \times \bR^d$, where $d$ is the
codimension of $K$ in $G$. In particular, the inclusion of
$K$ in $G$ is a homotopy equivalence.
\end{theorem}
\begin{proof}
See Theorem 3.1 in Chapter XV of Hochschild \cite{Hochschild:1965a}.
\end{proof}

In the case of $\Sp(2g,\bR)$, a connected Lie group of dimension $g(2g+1)$, 
the maximal compact subgroup is
the unitary group\index{unitary group} $\UU(g)$ of dimension $g^2$. Thus as a topological space,
we have $\Sp(2g,\bR) \cong \UU(g) \times \bR^{g(g+1)}$.

\begin{prop}\label{pr:UgSp2gR}
The inclusion $\UU(g) \to \Sp(2g,\bR)$ of topological groups is a homotopy equivalence.
Thus we have $\pi_1(\Sp(2g,\bR)) \cong \bZ$ and 
$$ H^*(B\Sp(2g,\bR);\bZ)\cong \bZ[c_1,\dots,c_g], $$
where $c_i\in H^{2i}(B\Sp(2g,\bR);\bZ)=H^{2i}(B\UU(g);\bZ)$ denotes the $i$-th Chern class.
\end{prop}
\begin{proof}
This follows from Theorem \ref{th:max-compact} and 
well known properties of $\UU(g)$.
\end{proof}

As $\pi_1(\Sp(2g,\bR))\cong \bZ$, the universal cover is a group
$\widetilde{\Sp(2g,\bR)}$\index{Sp2gR@$\widetilde{\Sp(2g,\bR)}$}
sitting in a central extension
$$ 1 \longrightarrow \bZ \longrightarrow \widetilde{\Sp(2g,\bR)} \longrightarrow \Sp(2g,\bR) \longrightarrow 1. $$

Provided $g\geqslant 4$, 
pulling back the universal cover of $\Sp(2g,\bR)$ to the perfect group $\Sp(2g,\bZ)$
gives the universal central extension
$$ 1 \longrightarrow \bZ \longrightarrow \widetilde{\Sp(2g,\bZ)} \longrightarrow \Sp(2g,\bZ) \longrightarrow 1. $$
For more about this group, see Section \ref{Signatures of surface bundles mod $N$ for $N>8$}.

\begin{cor}
The coset space $\Sp(2g,\bR)/\Sp(2g-2,\bR)$ is homeomorphic to
a Cartesian product $S^{2g-1} \times \bR^{2g}$.\qed
\end{cor}

\begin{cor}\label{co:pi1Sp}
The inclusions 
$$ S^1=\UU(1)\longrightarrow \Sp(2,\bR) \longrightarrow \Sp(4,\bR) \longrightarrow \cdots \longrightarrow \Sp(2g,\bR) $$
induce isomorphisms on $\pi_1$.
\end{cor}
\begin{proof}
By the long homotopy exact sequence of a fibration.
\end{proof}

\begin{cor}\label{co:H_b^2Sp}
The Borel cohomology $H_B^2(\Sp(2g,\bR);\bZ)$ is isomorphic to $\bZ$,
with generator $c_1$. The restriction map 
$$ H_B^2(\Sp(2g,\bR);\bZ)\longrightarrow H_B^2(\Sp(2g-2,\bR);\bZ) $$ 
is an isomorphism.
\end{cor}
\begin{proof}
Proposition \ref{pr:UgSp2gR} contains the information that $H_B^2(\Sp(2g,\bR);\bZ)\cong\bZ$. The claim on the restriction map follows from the Serre spectral sequence of a fibration and the fact that $\Sp(2g,\bR)/\Sp(2g-2,\bR)$ is $(2g-2)$-connected for $g\geq 2$.
\end{proof}

\begin{theorem}\label{th:tautau'}
The cohomology class in $H^2_B(\Sp(2g,\bR);\bZ) \cong H^2(B\Sp(2g,\bR);\bZ)\cong \bZ$ defined by the
Meyer cocycle $\tau_g$ and the Maslov cocycle $\tau'_g$ are both equal to $4c_1$.
\end{theorem}
\begin{proof}
Remember from Theorem \ref{th:Borel} that $H_B^2(\Sp(2g, \bR); \bZ)\cong H^2(B\Sp(2g,\bR);\bZ)$,  which is $\bZ$ by Corollary \ref{co:H_b^2Sp}, and that it injects into $H^2(\Sp(2g, \bR)^\delta; \bZ)$.
It follows from Proposition \ref{pr:cocycles-restrict} and Corollaries \ref{co:pi1Sp}, \ref{co:H_b^2Sp}
that it suffices to prove the second statement for the subgroup 
$S^1=\UU(1)\subseteq \Sp(2,\bR)$.
In this case, the Meyer cocycle is represented by the first diagram
in Example \ref{eg:MeyerS1}, while the Maslov cocycle is represented by the second diagram.
According to Theorem \ref{th:S1cocycle}, these both have covering number $4$.
The standard cocycle on $S^{1\delta}$ is Borel and has covering number $1$ (see beginning of Section \ref{se:Cocycles on S1}). By Corollary \ref{co:H_b^2Sp} it is thus the restriction to $S^{1\delta}$ of a cocycle on $\Sp(2,\bR)$ representing $c_1$.
\end{proof}


\section{Signature of surface bundles modulo an integer}\label{se:proofs}

\subsection{Signatures of surface bundles mod $N$ for $N\leq 8$}\label{N<=8}
In \cite{BensonCampagnoloRanickiRovi:2017a} we showed the existence of a cohomology class in the second cohomology group of a finite quotient $\fH$ of $\Sp(2g, \bZ)$ that computes the mod 2 reduction of signature/4 for a surface bundle over a surface.

We recall this theorem in detail hereafter. It chases the cohomology class from the finite group
$\fH=\Sp(2g,\bZ)/\fK$ to Meyer's class on $\Sp(2g,\bZ)$.
We include details of low genus cases, which makes the theorem a
little hard to follow, so afterwards we summarise the situation
in a diagram valid for
$g\geqslant 4$.

As in \cite{BensonCampagnoloRanickiRovi:2017a} we write $\Gamma(2g,N)\unlhd\Sp(2g,\bZ)$ for the \emph{principal congruence subgroup}
consisting of symplectic matrices which are congruent to the identity
modulo $N$.
\begin{defn} \label{def:fK and fY}
For $g\geqslant 1$ we write $\fK$ (resp. $\fY$) for the subgroup of $\Sp(2g,\bZ)$ (resp. $\Sp(2g,\bZ/4)$) consisting of
matrices
\[ \begin{pmatrix} I+2\mra & 2\mrb \\ 2\mrc & I+2\mrd \end{pmatrix} \in
  \Sp(2g,\bZ)\, {\rm (}\mbox{resp. } \Sp(2g,\bZ/4){\rm )} \]
satisfying:
\begin{enumerate}
\item[\rm (i)] The vectors of diagonal entries $\Diag(\mrb)$ and $\Diag(\mrc)$ are even, and
\item[\rm (ii)] the trace $\Tr(a)$ is even (equivalently, $\Tr(d)$ is even).
\end{enumerate}
\end{defn}


\begin{lemma}\label{i: K, H}
\begin{enumerate}[\rm (i)]
\item Let $g \geq 2$. The group $\fK$  (resp. $\fY$) is a normal subgroup of $\Sp(2g,\bZ)$ (resp. $\Sp(2g,\bZ/4)$). We write
$\fH$ for the quotient $\Sp(2g,\bZ)/\fK\cong\Sp(2g,\bZ/4)/\fY$.
\item  The quotient $\Gamma(2g,2)/\fK \leq \fH$ is an elementary abelian
  $2$-group $(\bZ/2)^{2g+1}$.

\end{enumerate}
\end{lemma}

\begin{proof}
\begin{enumerate}[\rm (i)]
\item See \cite[Theorem 2.2(i)]{BensonCampagnoloRanickiRovi:2017a}.

\item See \cite[Theorem 2.2(ii)]{BensonCampagnoloRanickiRovi:2017a}.
\end{enumerate} 
\vspace*{-1.3\baselineskip}
\end{proof}

\begin{lemma}\label{braid-lemma}  \begin{enumerate}[\rm(i)] 
\item  \label{iii: commutative braids}  There are commutative braids of extensions
\index{sp2gZ2@$\fsp(2g,\bbZ{2})$}
\medskip
\[  \braid{\Gamma(2g,4)}{\Sp(2g,\bZ)}{\fH}{\fK}{\Sp(2g,\bbZ{4})}{\fY}\qquad
  \braid{\fY}{\Sp(2g,\bbZ{4})}{\Sp(2g,\bbZ{2})}{\fsp(2g,\bbZ{2})}{\fH}{(\bbZ{2})^{2g+1}}\]
  
\item \label{iv: orders} 
The extension
\[ 1 \longrightarrow (\bZ/2)^{2g+1} \longrightarrow\fH \longrightarrow \Sp(2g,\bbZ{2}) \longrightarrow 1 \]
does not split. 

\item The finite groups in the braids have orders
\[ \begin{array}{l}
\vert \Sp(2g,\bbZ{2}) \vert=2^{g^2}\prod\limits^g_{i=1}(2^{2i}-1)~,~\vert \Sp(2g,\bbZ{4}) \vert=2^{g(3g+1)}\prod\limits^g_{i=1}(2^{2i}-1)~,\\
\vert \fH \vert=2^{(g+1)^2}\prod\limits^g_{i=1}(2^{2i}-1)~,~
\vert \fY\vert =2^{(2g+1)(g-1)}~,~\vert \fsp(2g,\bbZ{2}) \vert = 2^{g(2g+1)}~.
\end{array} \]
\end{enumerate}
  \end{lemma}

  \begin{proof}
  \begin{enumerate}[\rm(i)]
  \item  The left diagram holds by definition of the groups in presence.

The upper left sequence of the right diagram is already in the left one. The upper right one is to be found in \cite[after Theorem 5.1]{BensonCampagnoloRanickiRovi:2017a}. Note that this shows that $\fsp(2g,\bbZ{2})\cong \Gamma(2g,2)/\Gamma(2g,4)$. This observation proves the exactness of the lower left sequence as well. The lower right one holds by the third isomorphism theorem for groups. The right triangle commutes: it is a composition of quotient maps. The left one commutes as well: it is a composition of inclusions, by definition of $\fY$ and using $\fsp(2g,\bbZ{2})\cong \Gamma(2g,2)/\Gamma(2g,4)$. The square commutes:
\[
\xymatrix
{
\Gamma(2g,2)/\Gamma(2g,4)\ar@{^{(}->}[r]\ar[d]_{/(\fK/\Gamma(2g, 4))}&\Sp(2g,\bZ)/(\Gamma(2g, 4)\ar[d]^{/(\fK/\Gamma(2g, 4))}\\
\Gamma(2g,2)/\fK\ar@{^{(}->}[r]&\Sp(2g,\bZ)/\fK.
}
\]

\item See \cite[Theorem 2.2(iii)]{BensonCampagnoloRanickiRovi:2017a}.

\item For the order of $\Sp(2g,\bZ/2)$, see \cite[2.3]{atlas}. For the order of $\fH$, combine the order of $\Sp(2g,\bZ/2)$ with the short exact sequence in the diagram on the right in Part (\ref{iii: commutative braids}). The order of $\fsp(2g, \bZ/2)$ is given in \cite[Theorem 5.1]{BensonCampagnoloRanickiRovi:2017a}. Using again the diagrams of \eqref{iii: commutative braids}, we obtain successively $|\fY|$ and $|\Sp(2g, \bZ/4)|$.
\end{enumerate}
  \end{proof}

\begin{prop}
\label{v: E, Z, H} The group $\fH$ has a double cover\index{H@$\tilde\fH$, double cover of $\fH$}
\[ 1 \longrightarrow \bbZ{2} \longrightarrow \tilde\fH \longrightarrow \fH \longrightarrow 1 \]
which is a non-split extension of $\Sp(2g,\bbZ{2})$ by an almost
extraspecial group 
$E$\index{E@$E$, almost extraspecial group}\index{almost extraspecial group} 
of order $2^{2g+2}$; namely $E$
is a central product of $\bbZ{4}$ with an
extraspecial group\index{extraspecial group} of order $2^{2g+1}$.
Furthermore, the conjugation action of $\tilde\fH$
on $E$ induces an isomorphism between $\tilde\fH/Z(\tilde\fH) \cong \fH/Z(\fH)$
and the index two subgroup of
$\Aut(E)$ centralising $Z(\tilde\fH)=Z(E)$. The groups fit into a
commutative braid of extensions\medskip
\[ \braid{\bbZ{2}}{\tilde\fH}{\Sp(2g,\bbZ{2})}{E}{\fH}{(\bbZ{2})^{2g+1}} \]
\end{prop}

\begin{proof}
The existence of the non-trivial double cover of $\fH$ is implied by \cite[Theorem 6.12, Corollary 6.14]{BensonCampagnoloRanickiRovi:2017a} for $g\geq 3$ and by the cohomology group tables in \cite{BensonCampagnoloRanickiRovi:2017a} for $g=2$.

By definition, an almost extraspecial group of order $2^{2g+2}$ is a central product of $\bZ/4$ with an extraspecial group of order $2^{2g+1}$: take the direct product of $\bZ/4$ with an extraspecial group of order $2^{2g+1}$ and quotient out the diagonally embedded central subgroup of order two. The center of the resulting almost extraspecial group is a copy of $\bZ/4$. 

Griess \cite[Theorem 5 (b)]{Griess:1973a} (see also \cite[Section 3]{BensonCampagnoloRanickiRovi:2017a}) shows that the group $\tilde\fH$ is an extension of $\Sp(2g, \bZ/2)$ by an almost extraspecial group of order $2^{2g+2}$, that we denote by $E$, and that $\tilde\fH/Z(\tilde\fH)$ is isomorphic to the subgroup of index $2$ in $\Aut(E)$ that centralizes the center of $E$. Moreover $Z(\tilde\fH)=Z(E)$ (caveat: the statement in Griess's theorem that says $Z(2^{2g+1})=Z(\tilde\fH)$ is wrong, but his proof shows the correct assertion).

The isomorphism $\tilde\fH/Z(\tilde\fH)\stackrel{\cong}{\rightarrow}\fH/Z(\fH)$ follows from the fact that the preimage of $Z(\fH)$ is exactly $Z(\tilde\fH)$: indeed $|Z(\tilde\fH)|=4$, and $Z(\tilde\fH)$ maps to $Z(\fH)$ with kernel $\bZ/2$.

To see that $|Z(\fH)|$ is (at least) $2$, proceed as follows: the action by conjugation of $\Sp(2g, \bZ/2)$ on the kernel of the extension $(\bZ/2)^{2g+1}$ is trivial for exactly all elements of $\Sp(2g, \bZ/2)$ that lift to an element of $Z(\fH)$. Recall that the action is given through the isomorphism $\Sp(2g, \bZ/2)\cong \Orth(2g+1, \bZ/2)$ \cite[before Remark 5.3]{BensonCampagnoloRanickiRovi:2017a}. The elements that act trivially are therefore the ones lying in the radical of the bilinear form. This is one dimensional, isomorphic to $\bZ/2$.

The commutative braid of extensions is implied by \cite[Theorem 2.2 (ix)]{BensonCampagnoloRanickiRovi:2017a}. 

Commutativity of the diagram implies that $\tilde{\fH}$ is non-split as an extension of $\Sp(2g, \bbZ{2})$: the existence of a section of $\tilde\fH\rightarrow \Sp(2g, \bZ/2)$ would imply the existence of a section of $\fH\rightarrow \Sp(2g, \bZ/2)$, and this would contradict Lemma  \ref{braid-lemma}     (\ref{iv: orders}).

If $g=2$, then there is more than one isomorphism class of groups $\fH$, and thus $\tilde{\fH}$, fitting in this diagram. The particular group $\fH$ described here is one candidate. This can be checked using the representation of the double cover $\tilde{\fH}$ given in \cite{Benson:theta}.
\end{proof}

\begin{lemma}\
\label{ix: reduction mod 8} We have
\begin{align*}
H^2(\Sp(2g,\bZ);\bZ)&\cong
\begin{cases}
\bZ \oplus \bZ/2 & \mbox{if }g=2, \\
\bZ & \mbox{if }g\geqslant 3, \end{cases} \\
H^2(\Sp(2g,\bZ);\bZ/8)&\cong\begin{cases}
\bZ/8 \oplus \bZ/2 \oplus \bZ/2 & \mbox{if }g=2, \\
\bZ/8 \oplus \bZ/2 & \mbox{if }g=3, \\
\bZ/8 & \mbox{if }g\geqslant 4. \end{cases}
\end{align*}
The image of reduction modulo eight is
$\begin{cases}
\bZ/8 \oplus \bZ/2 & \mbox{if }g=2, \\
\bZ/8 & \mbox{if }g\geqslant 3.\end{cases}$
\end{lemma}
\begin{proof}

See Lemma \ref{Lemma:co-homology computations} \eqref{H^2 Sp Z Z}, \eqref{H^2 Sp Z Z/8}, \eqref{H^2 Sp image r8}.

\end{proof}

\begin{theorem}\label{th:main}
\begin{enumerate} [\rm(i)]
\item \label{vi: inflations} The cohomology class in $H^2(\fH;\bbZ{2})$ determined by
the central extension $\tilde\fH$ inflates to a non-zero class 
$[\overline{\tau_g}] \in H^2(\Sp(2g,\bbZ{4});\bbZ{2})$,\index{tau@$\tau_4 \in H^2(\Sp(2g,\bbZ{4});\bbZ{2})$} 
to a non-zero class in 
$H^2(\Sp(2g,\bZ);\bbZ{2})$, and also to a non-zero class in $H^2(\mcg{g};\bbZ{2})$ for $g\geq 3$.
If $\fX$\index{X@$\fX$, double cover of $\Sp(2g,\bbZ{4})$} 
denotes the double cover of $\Sp(2g,\bbZ{4})$ corresponding to $[\overline{\tau_g}]$, we have:

\[ \braidtwo{\bbZ{2}}{\tilde\fH}{\mathfrak{X}}{\fH}{\fY}{\Sp(2g,\bbZ{4})} \]

\bigskip

\item\label{vii: H, Sp(2g, Z/2^n)} For $n\geq 2$, the inflation map induces an isomorphism
\[ H^2(\fH;\bbZ{2}) \cong H^2(\Sp(2g,\bbZ{2^n});\bbZ{2})
\cong \begin{cases}\bbZ{2} \oplus \bbZ{2}
  \oplus \bbZ{2} & \mbox{if }g=2, \\ \bbZ{2} \oplus \bbZ{2} & \mbox{if }g=3, \\
\bbZ{2} & \mbox{if }g\geqslant 4.\end{cases} \]
\item\label{viii: Z/2 Z/8} For $g \geqslant 3$ the inclusion of $\bZ/2$ in $\bZ/8$ sending $1$ to $4$ induces an isomorphism
\[ H^2(\Sp(2g,\bZ/4); \bZ/2) \longrightarrow H^2(\Sp(2g,\bZ/4);\bZ/8). \]
\item\label{x: inflation Z/4 Z} For $g \geqslant 3$ the inflation map
\[ H^2(\Sp(2g,\bZ/4); A) \longrightarrow H^2(\Sp(2g,\bZ); A) \]
is injective for any abelian group of coefficients $A$ with trivial action.
\item\label{xi: universal covers} Pulling back the universal cover of $\Sp(2g,\bR)$ to $\Sp(2g,\bZ)$
gives a group $\widetilde{\Sp(2g,\bZ)}$:
\[ \xymatrix@R=6mm{1 \ar[r] & \bZ \ar[r] \ar@{=}[d] &
\widetilde{\Sp(2g,\bZ)} \ar[r] \ar[d] & \Sp(2g,\bZ) \ar[r] \ar[d] & 1\\
1 \ar[r] & \bZ \ar[r] & \widetilde{\Sp(2g,\bR)} \ar[r] & \Sp(2g,\bR)
\ar[r] & 1.} \]
This is the universal central extension for $g\geqslant 4$.

For $g=3$, $\Sp(2g,\bbZ{2})$ has an exceptional double cover,
and the pullback of $\widetilde{\Sp(2g,\bZ)}$
and this
double cover over $\Sp(2g, \bZ/2)$ gives the universal central extension of $\Sp(2g,\bZ)$,
with kernel $\bZ \times \bbZ{2}$.

For $g=2$, $\Sp(2g,\bZ)$ is not perfect, so it does not have a
universal central extension. But the cohomology class
defined by the pullback is a generator for a $\bZ$ summand of
$H^2(\Sp(4,\bZ);\bZ)\cong \bZ\oplus \bZ/2$. 
\item\label{xii: generators versus Meyer} The Meyer cocycle
 is four times the generator of $H^2(\Sp(2g,\bZ);\bZ)\cong \bZ$ for
 $g\geqslant 3$, and four times the generator of the $\bZ$ summand for
 $g=2$.
\item\label{xiii: diagram chase} Following the non-zero cohomology class in $H^2(\fH;\bZ/2)$
corresponding to the central extension given in Proposition \ref{v: E, Z, H} 
through the maps described in \eqref{vi: inflations}--\eqref{x: inflation Z/4 Z} in this theorem and in Lemma \ref{ix: reduction mod 8}
\[ H^2(\fH;\bZ/2) \longrightarrow H^2(\Sp(2g,\bZ/4);\bZ/2) \longrightarrow H^2(\Sp(2g,\bZ/4);\bZ/8) \longrightarrow H^2(\Sp(2g,\bZ);\bZ/8) \]
gives the image of the Meyer cohomology class under reduction modulo eight
\[ H^2(\Sp(2g,\bZ);\bZ) \longrightarrow H^2(\Sp(2g,\bZ);\bZ/8). \]
Following it through
\[ H^2(\fH;\bZ/2) \longrightarrow H^2(\Sp(2g,\bZ/4);\bZ/2)\longrightarrow H^2(\Sp(2g,\bZ);\bZ/2) \]
gives the image of one quarter of the Meyer cohomology class under
reduction modulo two.
\end{enumerate}
\end{theorem}
\begin{proof}
\begin{enumerate}[\rm (i)]
\item \cite[Corollary 6.14 and Corollary 6.5]{BensonCampagnoloRanickiRovi:2017a} imply the inflation of the class corresponding to $\tilde{\fH}$ to a non-zero class in $H^2(\Sp(2g, \bbZ{4}); \bbZ{2})$ and in $H^2(\Sp(2g, \bZ); \bbZ{2})$ for $g\geq 3$.
For $g=2$, the inflation of the class of $\tilde{\fH}$ to non-zero classes in $H^2(\Sp(4, \bbZ{4}); \bbZ{2})$ and in $H^2(\Sp(4, \bZ); \bbZ{2})$ can be taken from the second statement of \eqref{xiii: diagram chase}.

The further inflation to a non-zero class in $H^2(\Gamma_g; \bZ/2)$ follows from the isomorphism $H^2(\Sp(2g, \bZ); \bZ/2)\cong H^2(\Gamma_g; \bZ/2)$ for $g\geq 3$ (see Lemma \ref{Lemma:co-homology computations} \eqref{H^2 map Gamma_g Sp Z/2}, which also shows that the statement does not hold for $g=2$).

The braid of extensions describes the groups in presence.

\item See \cite[Corollary 6.14]{BensonCampagnoloRanickiRovi:2017a} for the isomorphism when $g\geq 3$ and \cite[Theorem 6.12]{BensonCampagnoloRanickiRovi:2017a} for the value of $H^2(\fH; \bZ/2)$ for $g=3$. 
Let us recall the argument to obtain the case $g=2$.

The five-term homology exact sequence associated with 
$$1\longrightarrow \Gamma(2g, 2^n)\longrightarrow \Sp(2g,\bZ)\longrightarrow \Sp(2g,\bZ/2^n)\longrightarrow 1$$
is
\begin{multline*}
H_2(\Sp(2g,\bZ); \bZ)\longrightarrow H_2(\Sp(2g, \bZ/2^n); \bZ)\longrightarrow H_0(\Sp(2g, \bZ/2^n); H_1(\Gamma(2g, 2^n)))\\
\longrightarrow H_1(\Sp(2g,\bZ); \bZ)\longrightarrow H_1(\Sp(2g, \bZ/2^n); \bZ)\longrightarrow 0.
\end{multline*}
By \cite[Lemma 6.2]{BensonCampagnoloRanickiRovi:2017a}, we have $ H_0(\Sp(2g, \bZ/2^n); H_1(\Gamma(2g, 2^n)))=0$ for $g\geq 2, n\geq 1$. 
So for every $g\geq 2, n\geq 1$, the map $H_2(\Sp(2g,\bZ); \bZ)\rightarrow H_2(\Sp(2g, \bZ/2^n); \bZ)$ is surjective.

The five-term homology exact sequence associated with 
$$1\longrightarrow\fY\longrightarrow\Sp(2g,\bZ/4)\longrightarrow\fH\longrightarrow1$$
is
\begin{multline*}
H_2(\Sp(2g,\bZ/4); \bZ)\longrightarrow H_2(\fH; \bZ)\longrightarrow H_0(\fH; H_1(\fY))\\
\longrightarrow H_1(\Sp(2g,\bZ/4); \bZ)\longrightarrow H_1(\fH; \bZ)\longrightarrow 0.
\end{multline*}
By \cite[Lemma 5.5 (i) and proof of Proposition 6.6]{BensonCampagnoloRanickiRovi:2017a} we have $ H_0(\fH; H_1(\fY))=0$ for $g\geq 1$. The map $H_2(\Sp(2g,\bZ/4); \bZ)\rightarrow H_2(\fH; \bZ)$ is thus surjective for $g\geq 1$. Together with the above this implies that $H_2(\Sp(2g,\bZ/2^n); \bZ)\rightarrow H_2(\fH; \bZ)$ is surjective for all $g, n\geq2$.
Consider the sequence of induced maps
$$\mathsf{Hom}(H_2(\Sp(2g,\bZ); \bZ), \bZ/2)\longleftarrow \mathsf{Hom}(H_2(\Sp(2g, \bZ/2^n); \bZ), \bZ/2)\longleftarrow \mathsf{Hom}(H_2(\fH; \bZ), \bZ/2).$$
By the above both are injective for $g, n\geq 2$. 
Insert the values of $H_2(\Sp(2g,\bZ); \bZ)$ (Lemma \ref{Lemma:co-homology computations} \eqref{H_2 Sp Z Z}) and $H_2(\fH; \bZ)$ \cite[Tables]{BensonCampagnoloRanickiRovi:2017a}. We see that both maps must be isomorphisms for $g, n\geq 2$. By the universal coefficient theorem this finishes the proof for $g\geq 3$, as each of $\Sp(2g,\bZ), \Sp(2g, \bZ/2^n)$, and $\fH$ are perfect.

For $g=2$, we need to also consider the $\mathsf{Ext}$-terms coming from $H_1(\Sp(2g,\bZ);\bZ)$, $H_1(\Sp(2g, \bZ/2^n);\bZ)$, and $H_1(\fH; \bZ)$. By the five-term exact sequences above, the maps $H_1(\Sp(2g,\bZ); \bZ)\rightarrow H_1(\Sp(2g, \bZ/2^n)$  for every $n\geq 1$ and $H_1(\Sp(2g,\bZ/4); \bZ)\rightarrow H_1(\fH; \bZ)$ are surjective. This implies that $$H_1(\Sp(2g, \bZ/2^n)\longrightarrow H_1(\fH; \bZ)$$ is surjective for all $n\geq 2$ as well. Insert the values of $H_1(\Sp(2g,\bZ); \bZ)$ (Lemma \ref{Lemma:co-homology computations} \eqref{H_1 Sp}) and $H_1(\fH; \bZ)$ \cite[Tables]{BensonCampagnoloRanickiRovi:2017a}. The composition 
$$\bZ/2\cong H_1(\Sp(2g,\bZ);\bZ)\longrightarrow H_1(\Sp(2g, \bZ/2^n);\bZ)\longrightarrow H_1(\fH; \bZ)\cong\bZ/2$$
is surjective, thus also injective. Then each of the two maps is an isomorphism, and thus by functoriality we obtain isomorphisms also for 
$$\mathsf{Ext}^1_{\bZ}(H_1(\fH; \bZ), \bZ/2)\longrightarrow \mathsf{Ext}^1_{\bZ}(H_1(\Sp(2g, \bZ/2^n);\bZ), \bZ/2)\longrightarrow \mathsf{Ext}^1_{\bZ}(H_1(\Sp(2g,\bZ);\bZ), \bZ/2)$$
for every $n\geq 2$. This together with the preceding results on the $\mathsf{Hom}$-term finishes the proof for $g=2$.
\item (See also \cite[Chapter 6]{CampagnoloThesis}.) We consider the following short exact sequence:
$$
\xymatrix
{
0\ar[r]&\mathbb{Z}/2\ar[r]^{i} & \mathbb{Z}/8\ar[r]^{p} &\mathbb{Z}/4\ar[r]&0,
}
$$ 
where $i$ is the map sending $1$ to $4$.
The induced long exact sequence in cohomology for $\Sp(2g, \mathbb{Z}/4)$ is:
\vspace{-3pt}
$$\begin{tikzpicture}[descr/.style={fill=white,inner sep=1pt}]
        \matrix (m) [
            matrix of math nodes,
            row sep=1.5em,
            column sep=2.5em,
            text height=1.5ex, text depth=0.25ex
        ]
       {  & ... & H^1(\Sp(2g, \mathbb{Z}/4); \mathbb{Z}/4) \\
             H^2(\Sp(2g, \mathbb{Z}/4); \mathbb{Z}/2) & H^2(\Sp(2g, \mathbb{Z}/4);\mathbb{Z}/8) & H^2(\Sp(2g, \mathbb{Z}/4);\mathbb{Z}/4) \\
            H^3(\Sp(2g, \mathbb{Z}/4);\mathbb{Z}/2) &... & \\
            &     \mbox{}    &                 &      \mbox{}    \\
        };
 \path[overlay,->, font=\scriptsize,>=latex]
        (m-1-2) edge node[descr,yshift=0.3ex] {$p^*$}(m-1-3)
     
            (m-1-3) edge[out=355,in=175] node[descr,yshift=0.3ex] {$\partial^1$} (m-2-1)
        (m-2-1) edge node[descr,yshift=0.3ex] {$i^*$}(m-2-2)
        (m-2-2) edge node[descr,yshift=0.3ex] {$p^*$} (m-2-3)
        (m-2-3) edge[out=355,in=175] node[descr,yshift=0.3ex] {$\partial^2$} (m-3-1)
        (m-3-1) edge node[descr,yshift=0.3ex] {$i^*$} (m-3-2);
\end{tikzpicture}$$
Let $g\geq 3$. Then $H^1(\Sp(2g, \mathbb{Z}/4); \mathbb{Z}/4)=0$ by perfection of the group. Using the cohomology computations of Part (\ref{vii: H, Sp(2g, Z/2^n)}) and Lemma \ref{Lemma:co-homology computations} \eqref{H^2 Sp Z/4 Z/8}, we thus have an isomorphism 
\[i^* \colon H^2(\Sp(2g, \mathbb{Z}/4); \mathbb{Z}/2)\stackrel{\cong}{\longrightarrow} H^2(\Sp(2g, \mathbb{Z}/4);\mathbb{Z}/8).\]
\medskip

\item By \cite[Corollary 6.5]{BensonCampagnoloRanickiRovi:2017a}.
\medskip

\item See \cite[Section 1] {BensonCampagnoloRanickiRovi:2017a} for $g\geq 3$.
For $g=2$, consider the following commutative diagram:
$$
\xymatrix
{
 & \bZ\ar[r]& \widetilde{\Sp(8, \bZ)}\ar[r]& \Sp(8, \bZ)\ar[dr] & \\
1\ar[ur]\ar[r] & \bZ\ar[r]\ar[u]_=& \widetilde{\Sp(4, \bZ)}\ar[r]\ar@{^{(}->}[u]& \Sp(4, \bZ)\ar@{^{(}->}[u]\ar[r] &1
}
$$
The inclusion of the symplectic group $\Sp(4, \bZ)\subset\Sp(8, \bZ)$ induces a morphism in cohomology $$
\xymatrix
{
H^2(\Sp(8, \bZ); \bZ)\cong\bZ\ar[r]& H^2(\Sp(4, \bZ); \bZ)\cong\bZ\oplus\bZ/2
}
$$ that maps $[\tau_4]=4\mapsto [\tau_2]=(4, 0)$ by Proposition \ref{pr:cocycles-restrict} (ii). For the generators we obtain $1\mapsto (1, 0)$ (see the proof of \eqref{xii: generators versus Meyer} below). This proves the statement for $g=2$.

\medskip

\item 

See \cite[Satz 2]{Meyer:1973a} for $g\geq 3$, as the map $p^*\colon H^2(\Sp(2g, \bZ); \bZ)\rightarrow H^2(\Gamma_g; \bZ)$ is an isomorphism (Lemma \ref{Lemma:co-homology computations} \eqref{H^2 map Gamma_g Sp}).

For $g=2$, consider the commutative diagram
$$
\xymatrix
{
\bZ\cong H^2_B(\Sp(6, \bR); \bZ)\ar[r]^{\cong}\ar[d]^{\cong}_r&H^2(\Sp(6, \bZ); \bZ)\cong\bZ\ar[d]^{r_\bZ}\\
\bZ\cong H^2_B(\Sp(4, \bR); \bZ)\ar@{^{(}->}[r]&H^2(\Sp(4, \bZ); \bZ)\cong\bZ\oplus\bZ/2.
}
$$
The map $r$ is an isomorphism by Corollary \ref{co:H_b^2Sp}, while the horizontal isomorphism is given by \cite[Proposition 3.1, p. 258]{barge-ghys}, as the groups in presence are cyclic (see also the proof of Lemma \ref{Lemma:co-homology computations} \eqref{H_2 map Gamma_g Sp}). The injection at the bottom comes from the same argument: the first Chern class $c_1$ is sent to $\frac{1}{4}[\tau_3]$ in the top line, and to either $(\frac{1}{4}[\tau_2], 0)$ or $(\frac{1}{4}[\tau_2], 1)\in\bZ\oplus\bZ/2$ in the bottom line. But Proposition \ref{pr:cocycles-restrict} (ii) says that the map $r$ sends $\tau_3$ to $\tau_2$ even as cocycles; a fortiori this is true for $r_\bZ$ on the restrictions of $\tau_3$ and $\tau_2$. So composing the arrrows shows that $c_1$ goes to $(\frac{1}{4}[\tau_2], 0)$, just as claimed.
\item For $g\geq 4$, the first statement is obtained as follows: denote by $[\eta]\in H^2(\fH; \bbZ{2})$ the generator corresponding to the extension $\tilde{\fH}$. The first map is an isomorphism by \eqref{vii: H, Sp(2g, Z/2^n)}, the second  by \eqref{viii: Z/2 Z/8}, the third is injective by \eqref{x: inflation Z/4 Z}, so that, using the computations in \eqref{vii: H, Sp(2g, Z/2^n)} and Lemma \ref{ix: reduction mod 8} the image of $[\eta]$ in $H^2(\Sp(2g, \bZ); \bbZ{8})$ is $4\in \bbZ{8}$. By Lemma \ref{ix: reduction mod 8} again, this is exactly the image of $[\tau_g]=4\in H^2(\Sp(2g, \bZ); \bZ)$ in $H^2(\Sp(2g, \bZ); \bbZ{8})$.

For $g=3$, consider the following commutative diagram:

{\footnotesize
$$
\xymatrix
{ 
\bZ/2\oplus\bZ/2\ar[d]^{\cong}& \bZ/2\oplus\bZ/2\ar[d]^{\cong}& \bZ/2\oplus\bZ/2\ar[d]^{\cong}& \bZ/8\oplus\bZ/2\ar[d]^{\cong}& \bZ\ar[d]^{\cong}\\
H^2(\fH; \bZ/2)\ar[r] & H^2(\Sp(6, \bZ/4); \bZ/2)\ar[r]& H^2(\Sp(6, \bZ/4); \bZ/8)\ar[r]& H^2(\Sp(6, \bZ); \bZ/8)&H^2(\Sp(6, \bZ); \bZ)\ar[l]\\
H^2(\fH; \bZ/2)\ar[r]_-{\cong}\ar[u]_f & H^2(\Sp(8, \bZ/4); \bZ/2)\ar[r]_{\cong}\ar[u]_h& H^2(\Sp(8, \bZ/4); \bZ/8)\ar@{^{(}->}[r]\ar[u]_k& H^2(\Sp(8, \bZ); \bZ/8)\ar[u]_\ell&H^2(\Sp(8, \bZ); \bZ)\ar[l]\ar[u]_{\cong}\\
\bZ/2\ar[u]_{\cong}& \bZ/2\ar[u]_{\cong}& \bZ/2\ar[u]_{\cong}& \bZ/8\ar[u]_{\cong}& \bZ\ar[u]_{\cong}
}
$$ }
The class $[\tau_4]\in H^2(\Sp(8, \bZ); \bZ)$ maps to $[\tau_3]\in H^2(\Sp(6, \bZ); \bZ)$ by Proposition \ref{pr:cocycles-restrict} (ii) and to $4\in H^2(\Sp(8, \bZ); \bZ/8)\cong\bZ/8$ by Lemma \ref{ix: reduction mod 8}. By  Lemma \ref{ix: reduction mod 8} again, $[\tau_3]$ maps to $(4, 0)\in\bZ/8\oplus\bZ/2$. The map $\ell$ is thus injective. Then $k$ has to be injective too, and so does $h$, and then $f$. Consequently, the non-zero class $f([\eta])$ is mapped under composition of the maps in the first line of the diagram to $$\ell(4)=(4, 0)\in H^2(\Sp(6, \bZ); \bZ/8).$$ As we remarked, this is the image of $[\tau_3]$ in $H^2(\Sp(6, \bZ); \bZ/8)$. That proves the first statement for $g=3$.

For $g=2$, we reason in the same way. We have the following commutative diagram:
{ \footnotesize
$$
\xymatrix
{
& \bZ/2\oplus\bZ/2\oplus\bZ/2\ar[d]^{\cong}& \bZ/2\oplus\bZ/2\oplus\bZ/2\ar[d]^{\cong}& \bZ/8\oplus\bZ/2\oplus\bZ/2\ar[d]^{\cong}& \bZ\oplus\bZ/2\ar[d]^{\cong}\\
H^2(\fH; \bZ/2)\ar[r] & H^2(\Sp(4, \bZ/4); \bZ/2)\ar[r]& H^2(\Sp(4, \bZ/4); \bZ/8)\ar[r]& H^2(\Sp(4, \bZ); \bZ/8)&H^2(\Sp(4, \bZ); \bZ)\ar[l]\\
H^2(\fH; \bZ/2)\ar[r]_-{\cong}\ar[u]_f & H^2(\Sp(8, \bZ/4); \bZ/2)\ar[r]_{\cong}\ar[u]_h& H^2(\Sp(8, \bZ/4); \bZ/8)\ar@{^{(}->}[r]\ar[u]_k& H^2(\Sp(8, \bZ); \bZ/8)\ar[u]_\ell&H^2(\Sp(8, \bZ); \bZ)\ar[l]\ar@{^{(}->}[u]\\
\bZ/2\ar[u]_{\cong}& \bZ/2\ar[u]_{\cong}& \bZ/2\ar[u]_{\cong}& \bZ/8\ar[u]_{\cong}& \bZ\ar[u]_{\cong}
}
$$ }
The class $[\tau_4]\in H^2(\Sp(8, \bZ); \bZ)$ maps to $[\tau_2]\in H^2(\Sp(4, \bZ); \bZ)$ by Proposition \ref{pr:cocycles-restrict} (ii) and to $4\in H^2(\Sp(8, \bZ); \bZ/8)\cong\bZ/8$ by Lemma \ref{ix: reduction mod 8}. By  Lemma \ref{ix: reduction mod 8} again, $[\tau_2]$ maps to $(4, 0, 0)\in\bZ/8\oplus\bZ/2\oplus\bZ/2$. The map $\ell$ is thus injective. Then $k$ has to be injective too, and so does $h$, and then $f$. Consequently, the non-zero class $f([\eta])$ is mapped under composition of the maps in the first line of the diagram to $\ell(4)=(4, 0, 0)\in H^2(\Sp(4, \bZ); \bZ/8)$. As we remarked, this is the image of $[\tau_2]$ in $H^2(\Sp(4, \bZ); \bZ/8)$. That proves the first statement for $g=2$.

The second statement follows from \cite[Theorem 2.2 (viii)]{BensonCampagnoloRanickiRovi:2017a} for $g\geq 4$.

For $g=3$, consider the following commutative diagram:
$$
\xymatrix
{ 
\bZ/2\oplus\bZ/2\ar[d]^{\cong}& \bZ/2\oplus\bZ/2\ar[d]^{\cong}& \bZ/2\oplus\bZ/2\ar[d]^{\cong}& \bZ\ar[d]^{\cong}\\
H^2(\fH; \bZ/2)\ar[r] & H^2(\Sp(6, \bZ/4); \bZ/2)\ar[r]& H^2(\Sp(6, \bZ); \bZ/2)&H^2(\Sp(6, \bZ); \bZ)\ar[l]\\
H^2(\fH; \bZ/2)\ar[r]_-{\cong}\ar[u]_f & H^2(\Sp(8, \bZ/4); \bZ/2)\ar[r]_{\cong}\ar[u]_h& H^2(\Sp(8, \bZ); \bZ/2)\ar[u]_\ell&H^2(\Sp(8, \bZ); \bZ)\ar[l]\ar[u]_{\cong}\\
\bZ/2\ar[u]_{\cong}& \bZ/2\ar[u]_{\cong}&  \bZ/2\ar[u]_{\cong}& \bZ\ar[u]_{\cong}
}
$$
The class $1\in H^2(\Sp(8, \bZ); \bZ)$ maps to $1\in H^2(\Sp(6, \bZ); \bZ)$ by Proposition \ref{pr:cocycles-restrict} (ii) and to $1\in H^2(\Sp(8, \bZ); \bZ/2)\cong\bZ/2$ by Lemma \ref{Lemma:co-homology computations} \eqref{H^2 Sp image r2}. By Lemma \ref{Lemma:co-homology computations} \eqref{H^2 Sp image r2} again, $1$ maps to $(1, 0)\in H^2(\Sp(6, \bZ); \bZ/2)\cong\bZ/2\oplus\bZ/2$. The map $\ell$ is thus injective. Then $h$ has to be injective too, and so does $f$. Consequently, the non-zero class $f([\eta])$ is mapped under composition of the maps in the first line of the diagram to $\ell(1)=(1, 0)\in H^2(\Sp(6, \bZ); \bZ/2)$. As we remarked, this is the image of $1\in H^2(\Sp(6, \bZ); \bZ)$ in $H^2(\Sp(6, \bZ); \bZ/2)$. That proves the second statement for $g=3$.

For $g=2$, we reason in the same way. We have the following commutative diagram:
$$
\xymatrix
{ 
& \bZ/2\oplus\bZ/2\oplus\bZ/2\ar[d]^{\cong}& \bZ/2\oplus\bZ/2\oplus\bZ/2\ar[d]^{\cong}& \bZ\oplus\bZ/2\ar[d]^{\cong}\\
H^2(\fH; \bZ/2)\ar[r] & H^2(\Sp(4, \bZ/4); \bZ/2)\ar[r]& H^2(\Sp(4, \bZ); \bZ/2)&H^2(\Sp(4, \bZ); \bZ)\ar[l]\\
H^2(\fH; \bZ/2)\ar[r]_-{\cong}\ar[u]_f & H^2(\Sp(8, \bZ/4); \bZ/2)\ar[r]_{\cong}\ar[u]_h& H^2(\Sp(8, \bZ); \bZ/2)\ar[u]_\ell&H^2(\Sp(8, \bZ); \bZ)\ar[l]\ar@{^{(}->}[u]\\
\bZ/2\ar[u]_{\cong}& \bZ/2\ar[u]_{\cong}&  \bZ/2\ar[u]_{\cong}& \bZ\ar[u]_{\cong}
}
$$
The class $1\in H^2(\Sp(8, \bZ); \bZ)$ maps to $(1, 0)\in H^2(\Sp(4, \bZ); \bZ)$ by Proposition \ref{pr:cocycles-restrict} (ii) (see also the proof of \eqref{xii: generators versus Meyer} above) and to $1\in H^2(\Sp(8, \bZ); \bZ/2)\cong\bZ/2$ by Lemma \ref{Lemma:co-homology computations} \eqref{H^2 Sp image r2}. By Lemma \ref{Lemma:co-homology computations} \eqref{H^2 Sp image r2} again, $(1, 0)$ maps to $(1, 0, 0)\in H^2(\Sp(4, \bZ); \bZ/2)\cong\bZ/2\oplus\bZ/2\oplus\bZ/2$. The map $\ell$ is thus injective. Then $h$ has to be injective too, and so does $f$. Consequently, the non-zero class $f([\eta])$ is mapped under composition of the maps in the first line of the diagram to $\ell(1)=(1, 0, 0)\in H^2(\Sp(4, \bZ); \bZ/2)$. As we remarked, this is the image of $(1, 0)\in H^2(\Sp(4, \bZ); \bZ)$ in $H^2(\Sp(4, \bZ); \bZ/2)$. That proves the second statement for $g=2$. \qedhere
\end{enumerate}
\end{proof}


The following diagram, valid for $g\geqslant 4$, may help clarify the theorem.
\[ \xymatrix@=8mm{
\bZ/2=H^2(\fH;\bZ/2) \ar[r]^\cong \ar[d]^\cong &
H^2(\fH;\bZ/8)=\bZ/2 \ar[d]^\cong \\
\bZ/2=H^2(\Sp(2g,\bZ/4);\bZ/2) \ar[r]^\cong \ar[d]^\cong &
H^2(\Sp(2g,\bZ/4);\bZ/8)=\bZ/2 \ar[d]^4 \\
\bZ/2=H^2(\Sp(2g,\bZ);\bZ/2) \ar[r]^4 &
H^2(\Sp(2g,\bZ);\bZ/8)=\bZ/8 \\
\bZ=H^2(\Sp(2g,\bZ);\bZ) \ar[r]^4 \ar@{->>}[u] &
H^2(\Sp(2g,\bZ);\bZ)=\bZ \ar@{->>}[u]} \]
This illustrates the connection between Meyer's cohomology class on the bottom
right and the central extension $\tilde \fH$ of $\fH$ regarded as a cohomology class on the top left.

\subsection{Signatures of surface bundles mod $N$ for $N>8$}\label{Signatures of surface bundles mod $N$ for $N>8$}

We just explained the existence of a cohomology class in the second cohomology group of the finite quotient $\fH$ of $\Sp(2g, \bZ)$ that computes the mod 2 reduction of signature/4 for a surface bundle over a surface. Here we explain in detail why this information on the signature is the best we can hope for if we restrict ourselves to cohomology classes on finite quotients of the symplectic group.

Precisely, we want to show the following:
\begin{theorem}\label{cas Z/m}
Let $g\geqslant 4$ and $K\leq \Sp(2g, \bZ)$ be a normal subgroup of finite index. 
If $c \in H^2(\Sp(2g,\bZ)/K;\bZ/N)$ is a cohomology class such that for the monodromy $\bar\chi\colon\pi_1(\Sigma_h)\rightarrow\Sp(2g,\bZ)$ of any surface bundle $\Sigma_g\rightarrow E\rightarrow\Sigma_h$ we have $$\sigma(E) = -\langle\bar\chi^*p^*(c),[\Sigma_h]\rangle \in\bZ/N,$$
then $N=2, 4,$ or $8$. If $N=2$ or $4$, the value $\langle\bar\chi^*p^*(c),[\Sigma_h]\rangle \in\bZ/N$ is always $0$. Here $p\colon\Sp(2g,\bZ)\rightarrow\Sp(2g,\bZ)/K$ denotes the quotient map.
\end{theorem}
This will be a direct consequence of a more general result (see also Remark \ref{Z/m}):
\begin{theorem}\label{cas general}
Let $g\geqslant 4$ and $K\leq \Sp(2g, \bZ)$ be a normal subgroup of finite index, and $A$ be a commutative ring. If $c \in H^2(\Sp(2g,\bZ)/K; A)$ is a cohomology class such that for the monodromy $\bar\chi\colon\pi_1(\Sigma_h)\rightarrow\Sp(2g,\bZ)$ of any surface bundle $\Sigma_g\rightarrow E\rightarrow\Sigma_h$ we have $$\sigma(E) = -\langle\bar\chi^*p^*(c),[\Sigma_h]\rangle \in A,$$
then $\langle\bar\chi^*p^*(c),[\Sigma_h]\rangle$ is the image of $4\in\bZ$ in $A$ and lies in the $2$-torsion of $A$. Here $p\colon\Sp(2g,\bZ)\rightarrow\Sp(2g,\bZ)/K$ denotes the quotient map.
\end{theorem}
Morally, this theorem says that any cohomology class on a finite quotient of the symplectic group yields at most a $2$-valued information on the signature of a surface bundle over a surface (as it lies in the $2$-torsion of the arbitrary coefficient ring $A$).

The above results rely on the non-residual finiteness of $\widetilde{\Sp(2g, \bZ)}$  proved by Deligne \cite{Deligne:1978a}. We recall the setup.

Let $\widetilde{\Sp(2g, \bZ)}$ denote the central extension obtained by pulling back the universal cover of $\Sp(2g, \mathbb{R})$ to $\Sp(2g, \bZ)$:
$$\xymatrix{
1\ar[r]&\bZ\ar[r]\ar@{=}[d]&\widetilde{\Sp(2g, \bZ)}\ar[r]_u\ar[d]&\Sp(2g, \bZ)\ar[r]\ar[d]&1\phantom{.}\\
1\ar[r]&\bZ\ar[r]&\widetilde{\Sp(2g, \mathbb{R})}\ar[r]&\Sp(2g, \mathbb{R})\ar[r]&1.
}$$
For $g\geqslant 4$, the group $\widetilde{\Sp(2g, \bZ)}$ is the universal central extension of the perfect group $\Sp(2g, \bZ)$.

Let $[\tau_g]\in H^2(\Sp(2g, \bZ); \bZ)$ denote the class of the Meyer signature cocycle. Recall that $[\tau_g]=4\in\bZ\cong H^2(\Sp(2g, \bZ); \bZ)$ for $g\geqslant 3$.
\begin{theorem}[Deligne \cite{Deligne:1978a}]\label{deligne}
For $g\geqslant 2$, the group $\widetilde{\Sp(2g, \bZ)}$ is not residually finite: every subgroup of finite index contains the subgroup $2\bZ$.
\end{theorem}
\begin{proof}[Proof of Theorem \ref{cas general}]
Let $Q=\Sp(2g,\bZ)/K$ be a finite quotient of $\Sp(2g, \bZ)$. Denote by $p\colon\Sp(2g,\bZ)\rightarrow Q$ the projection map. We have a natural isomorphism 
$$\begin{array}{ccc}
\widetilde{\Sp(2g, \bZ)}/\widetilde{K}&\longrightarrow &\Sp(2g, \bZ)/K\\
\widetilde{g}\widetilde{K}&\longmapsto &u(\tilde{g})K,
\end{array}$$
where $\widetilde{K}=u^{-1}(K)$ denotes the preimage of $K$ in $\widetilde{\Sp(2g, \bZ)}$. So $[\widetilde{\Sp(2g, \bZ)}:\widetilde{K}]<\infty$. By Theorem \ref{deligne}, the subgroup $\widetilde{K}$ then contains $2\bZ$.

Note that as $\widetilde{\Sp(2g, \bZ)}$ is the universal central extension of $\Sp(2g, \bZ)$ for $g\geqslant 4$, the central subgroup of the extension is isomorphic to the second homology group:
$$\bZ\cong H_2(\Sp(2g, \bZ); \bZ),\, \forall g\geqslant 4.$$

{\bfseries Claim 1:} The map induced by $p$ in homology annihilates the subgroup $2\bZ$:
$$2\bZ\leq \mathsf{ker}\left(p_*\colon H_2(\Sp(2g, \bZ); \bZ)\longrightarrow H_2(Q; \bZ)\right).$$

{\bfseries Proof of Claim 1:} As a quotient of the perfect group $\Sp(2g, \bZ)$, the group $Q$ is perfect. Hence it possesses a universal central extension $\widetilde{Q}$ satisfying:
$$\xymatrix{
1\ar[r]&H_2(Q; \bZ)\ar[r]&\widetilde{Q}\ar[r]_\pi&Q\ar[r]&1.
}$$
Denote by $E$ the pullback of $\widetilde{Q}$ and $\Sp(2g, \bZ)$ fitting in the following comutative diagram:
$$\xymatrix{
1\ar[r]&H_2(Q; \bZ)\times\{\mathrm{id}\}\ar[r]\ar[d]_{p_1}& E\ar[d]_{p_1}\ar[r]_-{p_2}& \Sp(2g, \bZ)\ar[r]\ar[d]_p&1\phantom{.}\\
1\ar[r]&H_2(Q; \bZ)\ar[r]&\widetilde{Q}\ar[r]_\pi&Q\ar[r]&1.
}$$
The maps $p_1$ and $p_2$ denote the projections to the first and the second factor respectively, and $$E=\left\{(\tilde{q}, g)\in\widetilde{Q}\times\Sp(2g, \bZ)\,\vline\, \pi(\tilde{q})=p(g)\right\}.$$
They are both surjective, as $\pi$ and $p$ are. Moreover one checks directly that $E$ is a central extension of $\Sp(2g, \bZ)$. Consequently by the properties of the universal central extension, there exists a unique map $\phi\colon \widetilde{\Sp(2g, \bZ)}\rightarrow E$ such that $p_2\circ\phi=u$. We now have the commutative diagram
$$\xymatrix{
1\ar[r]&H_2(\Sp(2g, \bZ); \bZ)\ar[r]\ar[d]_{\phi}& \widetilde{\Sp(2g, \bZ)}\ar[d]_{\phi}\ar[dr]^-{u}&&\\
1\ar[r]&H_2(Q; \bZ)\times\{\mathrm{id}\}\ar[r]\ar[d]_{p_1}& E\ar[d]_{p_1}\ar[r]_-{p_2}& \Sp(2g, \bZ)\ar[r]\ar[d]_p&1\phantom{.}\\
1\ar[r]&H_2(Q; \bZ)\ar[r]&\widetilde{Q}\ar[r]_\pi&Q\ar[r]&1.
}$$
By functoriality (and general knowledge on the second homology of perfect groups) the leftmost composition $p_1\circ\phi$ is precisely the induced map in homology $p_*\colon H_2(\Sp(2g, \bZ); \bZ)\rightarrow H_2(Q; \bZ)$.
Now we have to show that the kernel of this map contains $2 H_2(\Sp(2g, \bZ); \bZ)\cong 2\bZ$.

Note that 
$$\mathsf{ker}\left(p_*\colon H_2(\Sp(2g, \bZ); \bZ)\longrightarrow H_2(Q; \bZ)\right)=\mathsf{ker}\left(p_1\circ\phi\right)\cap  H_2(\Sp(2g, \bZ); \bZ).$$

To conclude, it is thus enough to show that $\mathsf{ker}\left(p_1\circ\phi\right)$ contains $2\bZ$. Thus by Theorem \ref{deligne} it is enough to show that $[\widetilde{\Sp(2g, \bZ)}:\mathsf{ker}\left(p_1\circ\phi\right)]<\infty$.

To this purpose, observe that, 
 as $Q$ is finite, the modules in its bar resolution are finitely generated, so that $H_2(Q; \bZ)$ is finitely generated. Also $H^2(Q; \bZ)$ is annihilated by $|Q|$ \cite[III, 10.2]{brown}, so it is torsion.  As $Q$ is perfect, this together with the universal coefficient theorem yields that $H_2(Q; \bZ)$ is torsion (see the diagram below with $A=\bZ$). Consequently, $H_2(Q;\bZ)$ is finite.

Therefore the exension $\widetilde{Q}$ of the finite groups $Q$ and $H_2(Q, \mathbb{Z})$ is finite as well. This directly implies that $[\widetilde{\Sp(2g, \bZ)}:\mathsf{ker}\left(p_1\circ\phi\right)]<\infty$ and proves Claim 1. 

As $\Sp(2g, \bZ)$ and $Q$ are perfect, the universal coefficient theorem for cohomology provides the following commutative diagram:
$$\xymatrix{
\mathsf{Hom}(H_2(Q; \bZ), A)\ar[rr]^-{\mathsf{Hom}(p_*)}\ar@{}[d]|*[@]{\cong}& & \mathsf{Hom}(H_2(\Sp(2g, \bZ); \bZ), A)\ar@{}[d]|*[@]{\cong}\\
H^2(Q; A)\ar[rr]^-{p^*}&& H^2(\Sp(2g, \bZ); A)
}$$
for any principal ideal domain $A$.

{\bfseries Claim 2:} $\mathsf{Im}(p^*)\subset T_2(H^2(\Sp(2g, \bZ); A))$, where $T_2$ denotes the $2$-torsion.

{\bfseries Proof of Claim 2:} As indicated by the commutative diagram, the map $p^*$ is defined as follows:
$$ H^2(Q; A)\cong \mathsf{Hom}(H_2(Q; \bZ), A)\ni\alpha\longmapsto \alpha\circ p_*\in \mathsf{Hom}(H_2(\Sp(2g, \bZ); \bZ), A)\cong H^2(\Sp(2g, \bZ); A).$$
Then let $g\geqslant 4$ and $n\in H_2(\Sp(2g, \bZ); \bZ)\cong\bZ$. If $2\mid n$, Claim $1$ implies $p_*(n)=0$. So $\mathsf{Hom}(p_*)(\alpha)$ is determined by $\mathsf{Hom}(p_*)(\alpha)(1)$, and we have
$$2\mathsf{Hom}(p_*)(\alpha)(1)=\mathsf{Hom}(p_*)(\alpha)(2)=0.$$

This proves Claim $2$.

Note that $H^2(\Sp(2g, \bZ); A)\cong\mathsf{Hom}(H_2(\Sp(2g, \bZ); \bZ), A)\cong \mathsf{Hom}(\bZ, A)\cong A$ if $g\geqslant 4$.

Now comes the final part of the proof. Let $[\tau_g]_A$ denote the image of $[\tau_g]$ in $H^2(\Sp(2g, \bZ); A)$. Suppose that there exists a class $c\in H^2(Q; A)$ such that $$p^*(c)=[\tau_g]_A.$$ By Claim $2$ we have $2p^*(c)=0$. On the other hand, $[\tau_g]_A$ is the image of $[\tau_g]=4\in \bZ\cong H^2(\Sp(2g, \bZ); \bZ)$ in $A$. So we have $$p^*(c)=4_A\in T_2(A).$$

Hence this class $c$ can provide at most the $2$-valued information $4_A$ on the signature: either it is $0$ in $A$ or not.
\end{proof}

\begin{rk}\label{Z/m}
In particular, setting $A=\bZ/N$, we find the following:
$4_N\in T_2(\bZ/N)$, and
$$T_2(\bZ/N)=
\begin{cases}
		\{0\}& \mbox{if }N\mbox{ is odd},\\
		\left\{0, \textstyle{\frac{N}{2}}\right\}& \mbox{if }N\mbox{ is even}.
		\end{cases}$$
Hence no information is obtained in this way for odd $N$. For even $N$, suppose $\textstyle{\frac{N}{2}}=4_N$ so that 
\begin{eqnarray}
\textstyle{\frac{N}{2}}=4\pmod{N}.
\end{eqnarray}
Therefore $N=8$.
\end{rk}
\begin{rk}
If $g=3$, the conclusions of Theorems \ref{cas Z/m} and \ref{cas general} still hold, with the same proof, as the only difference between $g=3$ and $g\geqslant 4$ is an additional $2$-torsion summand in the second homology of the symplectic group with integer coefficients (Lemma \ref{Lemma:co-homology computations} \eqref{H_2 Sp Z Z}) as well as a $2$-torsion factor in the universal extension of $\Sp(6,\bZ)$ (Theorem \ref{th:main} (v)).
\end{rk}

\appendix
\section{Cohomology computations}\label{app: co-homology computations}
This appendix presents computations of homology and cohomology groups often used in the main text.							
\begin{lemma}\label{Lemma:co-homology computations}						
\begin{enumerate}[\rm (i)] 
\item\label{mapping class group} The mapping class group $\Gamma_g$ of the closed oriented surface of genus $g$ has
$$H_1(\Gamma_g; \bZ)=\begin{cases}
													\bZ/10 & \mbox{if }  g=2,\\
													\{0\} & \mbox{if }  g\geq 3.
															\end{cases}$$
Indeed, the mapping class groups are perfect $\forall g\geq 3$.
							$$H_2(\Gamma_g; \bZ)=\begin{cases}
													\bZ/2 & \mbox{if }  g=2,\\
													\bZ\oplus\bZ/2 & \mbox{if }  g=3,\\
													\bZ & \mbox{if }  g\geq4.
															\end{cases}$$
															
$$H^2(\Gamma_g; \bZ)=\begin{cases}						\bZ/10 & \mbox{if }  g=2,\\
													\bZ & \mbox{if }  g\geq 3.
															\end{cases}$$
															$$H^2(\Gamma_g; \bZ/2)=\begin{cases}
													\bZ/2\oplus\bZ/2& \mbox{if } g=2, 3,\\
													\bZ/2 & \mbox{if }  g\geq4.
															\end{cases}$$
															
\item\label{H_1 Sp} $$H_1(\Sp(2g, \bZ); \bZ)=\begin{cases}		\bZ/2 & \mbox{if }  g=2,\\
													\{0\} & \mbox{if }  g\geq 3.
															\end{cases}$$
															$$H_1(\Sp(2g, \bZ/4); \bZ)=\begin{cases}
													\bZ/2 & \mbox{if }  g=2,\\
													\{0\} & \mbox{if }  g\geq 3.
															\end{cases}$$
In fact, the symplectic group $\Sp(2g, \bZ)$ and thus all of its quotients, as quotients of the mapping class group, are perfect $\forall g\geq 3$.
\item\label{H_2 Sp Z Z} $$H_2(\Sp(2g, \bZ); \bZ)=\begin{cases}        \bZ\oplus\bZ/2 & \mbox{if }  g=2, 3,\\
													\bZ & \mbox{if }  g\geq 4.
															\end{cases}$$
															
\item\label{H_2 Sp Z/4 Z} $$H_2(\Sp(2g, \bZ/4); \bZ)=\begin{cases}	\bZ/2 \oplus\bZ/2&\mbox{if } g=2, 3,\\
													\bZ/2 & \mbox{if }  g\geq 4.
															\end{cases}$$
															
\item\label{H^2 Sp Z Z} $$H^2(\Sp(2g, \bZ); \bZ)=\begin{cases}  	\bZ\oplus\bZ/2&\mbox{if }g=2,\\
													\bZ & \mbox{if }  g\geq 3.
															\end{cases}$$
															
\item\label{H^2 Sp Z Z/2} $$H^2(\Sp(2g, \bZ); \bZ/2)=\begin{cases}   \bZ/2\oplus\bZ/2\oplus\bZ/2 & \mbox{if }  g=2,\\
													\bZ/2\oplus\bZ/2 & \mbox{if }  g=3,\\
													\bZ/2& \mbox{if }  g\geq 4.
															\end{cases}$$
\item\label{H^2 Sp Z Z/8} $$H^2(\Sp(2g, \bZ); \bZ/8)=\begin{cases}	\bZ/8\oplus\bZ/2\oplus\bZ/2 & \mbox{if } g=2,\\
													\bZ/8\oplus\bZ/2  & \mbox{if }  g=3,\\
													\bZ/8& \mbox{if }  g\geq 4.
															\end{cases}$$
\item\label{H^2 Sp Z/4 Z/2} $$H^2(\Sp(2g, \bZ/4); \bZ/2)=\begin{cases}   \bZ/2\oplus\bZ/2\oplus\bZ/2& \mbox{if }  g=2,\\
												         	\bZ/2\oplus\bZ/2& \mbox{if }  g=3,\\
													         \bZ/2& \mbox{if }  g\geq 4.
															\end{cases}$$
\item\label{H^2 Sp Z/4 Z/8} $$H^2(\Sp(2g, \bZ/4); \bZ/8)=\begin{cases} \bZ/2\oplus\bZ/2 \oplus \bZ/2& \mbox{if }  g=2,\\
													\bZ/2\oplus\bZ/2 & \mbox{if }  g=3,\\
													\bZ/2& \mbox{if }  g\geq 4.
												\end{cases}$$
\item\label{H^2 Sp image r2} The image of the map $r_2\colon H^2(\Sp(2g, \bZ); \bZ)\rightarrow H^2(\Sp(2g, \bZ); \bZ/2)$ is
$$\begin{cases}
													\bZ/2\oplus\bZ/2 \oplus\{0\}& \mbox{if }  g=2,\\
													\bZ/2 \oplus\{0\} & \mbox{if }  g=3,\\
													\bZ/2& \mbox{if }  g\geq 4.
															\end{cases}$$
\item\label{H^2 Sp image r8} The image of the map $r_8\colon H^2(\Sp(2g, \bZ); \bZ)\rightarrow H^2(\Sp(2g, \bZ); \bZ/8)$ is
$$\begin{cases}
													\bZ/8\oplus\bZ/2 \oplus\{0\}& \mbox{if }  g=2,\\
													\bZ/8 \oplus\{0\} & \mbox{if }  g=3,\\
													\bZ/8& \mbox{if }  g\geq 4.
															\end{cases}$$
\item\label{H_2 map Gamma_g Sp} The map $p_*\colon H_2(\Gamma_g; \bZ)\rightarrow H_2(\Sp(2g, \bZ); \bZ)$ induced in homology by the canonical surjection $p\colon\Gamma_g\rightarrow \Sp(2g, \bZ)$ is an isomorphism for $g\geq 3$.
\item\label{H^2 map Gamma_g Sp} The map $p^*\colon H^2(\Sp(2g, \bZ); \bZ)\rightarrow H^2(\Gamma_g; \bZ)$ induced in cohomology by the canonical surjection $p\colon\Gamma_g\rightarrow \Sp(2g, \bZ)$ is an isomorphism for $g\geq 3$.
\item\label{H^2 map Gamma_g Sp Z/2} The map $p^*\colon H^2(\Sp(2g, \bZ); \bZ/2)\rightarrow H^2(\Gamma_g; \bZ/2)$ induced in cohomology by the canonical surjection $p\colon\Gamma_g\rightarrow \Sp(2g, \bZ)$ is an isomorphism for $g\geq 3$.

\end{enumerate}
\end{lemma}
\begin{proof}
\begin{enumerate}
\item[\rm (i)]
For $g=2$, we have $H_1(\Gamma_2; \bZ)=\bZ/10$ (see for example \cite[Theorem 5.1]{korkmaz}). For $H_2(\Gamma_2; \bZ)$ see \cite[after Theorem 6.1]{korkmaz}.
The universal coefficient theorem and the properties of $\mathsf{Ext}$ allow to easily compute $H^2(\Gamma_2; \bZ)$ and $H^2(\Gamma_2; \bZ/2)$.

For $g\geq 3$, we have $H_1(\Gamma_g; \bZ)=\{0\}$ (see for example \cite[Theorem 5.2]{Farb/Margalit:2012a}). The group $H_2(\Gamma_3; \bZ)$ is computed in \cite[Corollary 4.10]{sakasai}. For $g\geq 4$, we have $H_2(\Gamma_g; \bZ)=\bZ$ (see for example \cite[Theorem 6.1]{korkmaz}).
From that we easily obtain $H^2(\Gamma_g; \bZ)$ and $H^2(\Gamma_g; \bZ/2)$ using the universal coefficient theorem.
\item[\rm (ii)] The symplectic group $\Sp(2g, \bZ)$, and thus all of its quotients, is a quotient of the mapping class group $\Gamma_g$, which is perfect for $g\geq 3$ (see for example \cite[Theorems 5.2, 6.4]{Farb/Margalit:2012a}).

For $g=2$, one can abelianize the presentation of $\Sp(4, \bZ)$ given in \cite{bender} or \cite[Theorem 2]{lu}.

For $\Sp(4, \bZ/4)$, one can compute it using GAP.
\item[\rm (iii)] See for example \cite[Lemma 6.11]{BensonCampagnoloRanickiRovi:2017a} for $g\geq3$.

For $g=2$ it was computed using GAP.
\item[\rm (iv)] See for example \cite[Lemma 6.12]{BensonCampagnoloRanickiRovi:2017a} for $g\geq 3$.

For $g=2$, it can be computed using GAP.
\item[\rm (v)] For $g\geq3$, perfection of $\Sp(2g, \bZ)$ together with the universal coefficient theorem gives an isomorphism
$$
\xymatrix
{
H^2(\Sp(2g,\bZ); \bZ)\ar[r]^-\cong &\mathsf{Hom}(H_2(\Sp(2g, \bZ); \bZ), \bZ).
}
$$
Hence using \eqref{H_2 Sp Z Z} we obtain the stated cohomology group.

For $g=2$, the universal coefficient theorem gives a split exact sequence
$$
\xymatrix
{
\mathsf{Ext}^1_{\bZ}(H_1(\Sp(4,\bZ); \bZ), \bZ)\ar[r]& H^2(\Sp(4,\bZ); \bZ)\ar[r]&\mathsf{Hom}(H_2(\Sp(4, \bZ); \bZ), \bZ).
}
$$
Using \eqref{H_1 Sp}, \eqref{H_2 Sp Z Z} together with the properties of $\mathsf{Ext}$, one obtains $H^2(\Sp(4, \bZ); \bZ)\cong \bZ\oplus\bZ/2$ as stated.
\item[\rm (vi)] For $g\geq3$, perfection of $\Sp(2g, \bZ)$ together with the universal coefficient theorem gives an isomorphism
$$
\xymatrix
{
H^2(\Sp(2g,\bZ); \bZ/2)\ar[r]^-\cong &\mathsf{Hom}(H_2(\Sp(2g, \bZ); \bZ), \bZ/2).
}
$$
Hence using \eqref{H_2 Sp Z Z} we obtain the stated cohomology groups.

For $g=2$, the universal coefficient theorem yields the split exact sequence
$$
\xymatrix
{
\mathsf{Ext}^1_{\bZ}(H_1(\Sp(2g,\bZ); \bZ), \bZ/2)\ar[r]& H^2(\Sp(2g,\bZ); \bZ/2)\ar[r]&\mathsf{Hom}(H_2(\Sp(2g, \bZ); \bZ), \bZ/2).
}
$$
Using \eqref{H_1 Sp} and \eqref{H_2 Sp Z Z} and the properties of $\mathsf{Ext}$ we obtain the required results.
\item[\rm (vii)] For $g\geq3$, perfection of $\Sp(2g, \bZ)$ together with the universal coefficient theorem gives an isomorphism
$$
\xymatrix
{
H^2(\Sp(2g,\bZ); \bZ/8)\ar[r]^-\cong &\mathsf{Hom}(H_2(\Sp(2g, \bZ); \bZ), \bZ/8).
}
$$
Hence using \eqref{H_2 Sp Z Z} we obtain the stated cohomology groups.

For $g=2$, the universal coefficient theorem yields the split exact sequence
$$
\xymatrix
{
\mathsf{Ext}^1_{\bZ}(H_1(\Sp(2g,\bZ); \bZ), \bZ/8)\ar[r]& H^2(\Sp(2g,\bZ); \bZ/8)\ar[r]&\mathsf{Hom}(H_2(\Sp(2g, \bZ); \bZ), \bZ/8).
}
$$
Using \eqref{H_1 Sp} and \eqref{H_2 Sp Z Z} and the properties of $\mathsf{Ext}$ we obtain the required results.
\item[\rm (viii)] For $g\geq3$, perfection of $\Sp(2g, \bZ/4)$ together with the universal coefficient theorem gives an isomorphism
$$
\xymatrix
{
H^2(\Sp(2g,\bZ/4); \bZ/2)\ar[r]^-\cong &\mathsf{Hom}(H_2(\Sp(2g, \bZ/4); \bZ), \bZ/2).
}
$$
Hence using \eqref{H_2 Sp Z/4 Z} we obtain the stated cohomology groups.

For $g=2$, the universal coefficient theorem yields the split exact sequence
$$
\xymatrix
{
\mathsf{Ext}^1_{\bZ}(H_1(\Sp(2g,\bZ/4); \bZ), \bZ/2)\ar[r]& H^2(\Sp(2g,\bZ/4); \bZ/2)\ar[r]&\mathsf{Hom}(H_2(\Sp(2g, \bZ/4); \bZ), \bZ/2).
}
$$ 
Points \eqref{H_1 Sp} and \eqref{H_2 Sp Z/4 Z} and the properties of the $\mathsf{Ext}$-functor yield the stated results.
\item[\rm (ix)] For $g\geq3$, perfection of $\Sp(2g, \bZ/4)$ together with the universal coefficient theorem gives an isomorphism
$$
\xymatrix
{
H^2(\Sp(2g,\bZ/4); \bZ/8)\ar[r]^-\cong &\mathsf{Hom}(H_2(\Sp(2g, \bZ/4); \bZ), \bZ/8).
}
$$
Hence using \eqref{H_2 Sp Z/4 Z} we obtain the stated cohomology groups.

For $g=2$, the universal coefficient theorem yields the split exact sequence
$$
\xymatrix
{
\mathsf{Ext}^1_{\bZ}(H_1(\Sp(2g,\bZ/4); \bZ), \bZ/8)\ar[r]& H^2(\Sp(2g,\bZ/4); \bZ/8)\ar[r]&\mathsf{Hom}(H_2(\Sp(2g, \bZ/4); \bZ), \bZ/8).
}
$$ 
Points \eqref{H_1 Sp} and \eqref{H_2 Sp Z/4 Z} and the properties of the $\mathsf{Ext}$-functor yield the stated results.
\item[\rm (x)] Naturality of the universal coefficient theorem and perfection of $\Sp(2g,\bZ)$ for $g\geq3$ give a commutative diagram
$$
\xymatrix
{
H^2(\Sp(2g,\bZ); \bZ)\ar[r]^-\cong\ar[d]^{r_2} &\mathsf{Hom}(H_2(\Sp(2g, \bZ); \bZ), \bZ)\ar[d]\phantom{.}\\
H^2(\Sp(2g,\bZ); \bZ/2)\ar[r]^-\cong& \mathsf{Hom}(H_2(\Sp(2g, \bZ); \bZ), \bZ/2).
}
$$
For $g=2$, the universal coefficient theorem yields the split exact sequence
$$
\xymatrix
{
\mathsf{Ext}^1_{\bZ}(H_1(\Sp(4,\bZ); \bZ), \bZ)\ar[r]\ar[d]& H^2(\Sp(4,\bZ); \bZ)\ar[r]\ar[d]^{r_2}&\mathsf{Hom}(H_2(\Sp(4, \bZ); \bZ), \bZ)\ar[d]\phantom{.}\\
\mathsf{Ext}^1_{\bZ}(H_1(\Sp(4,\bZ); \bZ), \bZ/2)\ar[r]& H^2(\Sp(4,\bZ); \bZ/2)\ar[r]&\mathsf{Hom}(H_2(\Sp(4, \bZ); \bZ), \bZ/2).
}
$$
It is easy to understand the map induced by $r_2$ on the $\mathsf{Hom}$-part: it is just the concatenation with the reduction morphism on the coefficients $\bZ\rightarrow\bZ/2$.

For the $\mathsf{Ext}$-part, consider the projective resolution of $H_1(\Sp(4, \bZ), \bZ)=\bZ/2$ given by 
$$0\longrightarrow \bZ\stackrel{\cdot 2}{\longrightarrow} \bZ\longrightarrow \bZ/2\longrightarrow 0.$$
Apply $\mathsf{Hom}$ to obtain
\[
\xymatrix
{
0& \mathsf{Hom}(\bZ, \bZ)\ar[l]\ar[d]^r& \mathsf{Hom}(\bZ, \bZ)\ar[d]\ar[l]_-{\cdot 2}\phantom{.}\\
0& \mathsf{Hom}(\bZ, \bZ/2)\ar[l]& \mathsf{Hom}(\bZ, \bZ/2)\ar[l]_-{\cdot 2}.
}
\]
The groups we are interested in, $\mathsf{Ext}^1_{\bZ}(H_1(\Sp(4,\bZ); \bZ), \bZ)$ and $\mathsf{Ext}^1_{\bZ}(H_1(\Sp(4,\bZ); \bZ), \bZ/2)$, are the cohomology groups of this part of the chain complexes. We notice that the map induced in cohomology is surjective, as the map $r$ is.

The map $r_2$ thus has image $$\begin{cases}
													\bZ/2\oplus\bZ/2 \oplus\{0\}& \mbox{if }  g=2,\\
													\bZ/2 \oplus\{0\}& \mbox{if }  g=3,\\
													\bZ/2 & \mbox{if }  g\geq 4.
															\end{cases}$$
\item[\rm (xi)] Naturality of the universal coefficient theorem and perfection of $\Sp(2g,\bZ)$ for $g\geq3$ give a commutative diagram
$$
\xymatrix
{
H^2(\Sp(2g,\bZ); \bZ)\ar[r]^-\cong\ar[d]^{r_8} &\mathsf{Hom}(H_2(\Sp(2g, \bZ); \bZ), \bZ)\ar[d]\phantom{.}\\
H^2(\Sp(2g,\bZ); \bZ/8)\ar[r]^-\cong& \mathsf{Hom}(H_2(\Sp(2g, \bZ); \bZ), \bZ/8).
}
$$
For $g=2$, the  universal coefficient theorem yields the split exact sequence
$$
\xymatrix
{
\mathsf{Ext}^1_{\bZ}(H_1(\Sp(4,\bZ); \bZ), \bZ)\ar[r]\ar[d]& H^2(\Sp(4,\bZ); \bZ)\ar[r]\ar[d]^{r_8}&\mathsf{Hom}(H_2(\Sp(4, \bZ); \bZ), \bZ)\ar[d]\phantom{.}\\
\mathsf{Ext}^1_{\bZ}(H_1(\Sp(4,\bZ); \bZ), \bZ/8)\ar[r]& H^2(\Sp(4,\bZ); \bZ/8)\ar[r]&\mathsf{Hom}(H_2(\Sp(4, \bZ); \bZ), \bZ/8).
}
$$
It is easy to understand the map induced by $r_8$ on the $\mathsf{Hom}$-part: it is just the concatenation with the reduction morphism on the coefficients $\bZ\rightarrow\bZ/8$.

For the $\mathsf{Ext}$-part, consider the projective resolution of $H_1(\Sp(4, \bZ); \bZ)=\bZ/2$ given by 
$$0\longrightarrow \bZ\stackrel{\cdot 2}{\longrightarrow} \bZ\longrightarrow \bZ/2\longrightarrow 0.$$
Apply $\mathsf{Hom}$ to obtain
\[
\xymatrix
{
0& \mathsf{Hom}(\bZ, \bZ)\ar[l]\ar[d]^r& \mathsf{Hom}(\bZ, \bZ)\ar[d]\ar[l]_-{\cdot 2}\phantom{.}\\
0& \mathsf{Hom}(\bZ, \bZ/8)\ar[l]& \mathsf{Hom}(\bZ, \bZ/8)\ar[l]_-{\cdot 2}.
}
\]
The groups we are interested in, $\mathsf{Ext}^1_{\bZ}(H_1(\Sp(4,\bZ); \bZ), \bZ)$ and $\mathsf{Ext}^1_{\bZ}(H_1(\Sp(4,\bZ); \bZ), \bZ/8)$, are the cohomology groups of this part of the chain complexes. We notice that the map induced in cohomology is surjective, as the map $r$ is.

The map $r_8$ thus has image $$\begin{cases}
													\bZ/2\oplus\bZ/8 \oplus\{0\}& \mbox{if }  g=2,\\
													\bZ/8 \oplus\{0\}& \mbox{if }  g=3,\\
													\bZ/8 & \mbox{if } g\geq 4.
															\end{cases}$$
\item[\rm (xii)]
Meyer shows in \cite[Satz 2]{Meyer:1973a} that the image of the cocycle $\tau_g$ under the composition of 
\[
\xymatrix
{
 \mathsf{Hom}_\bZ(H_2(\Sp(2g, \bZ); \bZ), \bZ)\cong H^2(\Sp(2g, \bZ); \bZ)\ar[r]& H^2(\Gamma_g; \bZ)\ar[r]^-\cong&\mathsf{Hom}_\bZ(H_2(\Gamma_g; \bZ), \bZ)
}
\] is a map $\hat{k}\in \mathsf{Hom}_\bZ(H_2(\Gamma_g; \bZ), \bZ) $ with image $4\bZ$. On the other hand, Barge-Ghys show in \cite[Proposition 3.1, p. 258]{barge-ghys} that $\tau_g$ is four times the cocycle defined by the central extension $\widetilde{\Sp(2g, \bZ)}$, the restriction of the universal cover $\widetilde{\Sp(2g, \bR)}$ of $\Sp(2g, \bR)$ to $\Sp(2g, \bZ)$. This means that $\tau_g$ is four times the generator of $H^2(\Sp(2g, \bZ); \bZ)$, which maps to four times the generator of $H^2(\Gamma_g; \bZ)$. Now this means that the restriction of $p_*$ to the $\bZ$-summand of $H_2(\Gamma_g; \bZ)$ is an isomorphism to the $\bZ$-summand of $H_2(\Sp(2g, \bZ); \bZ)$.

The case $g\geq 4$ is thus settled. For $g=3$, we use Theorem 4.9, Corollary 4.10 and the paragraph before in \cite{sakasai}, together with \cite[Proposition 1.5]{Korkmaz/Stipsicz:2003a}: the map $H_2(\Gamma_{3, 1}; \bZ)\rightarrow H_2(\Gamma_3; \bZ)$ is surjective, thus an isomorphism, and this implies that the $\bZ/2$-summand of $H_2(\Gamma_3; \bZ)$ is sent isomorphically by $p_*$ to the $\bZ/2$-summand of $H_2(\Sp(6, \bZ); \bZ)$. The proof of the isomorphism is now complete.
\item[\rm (xiii)]
Meyer shows in \cite[Satz 2]{Meyer:1973a} that the image of the cocycle $\tau_g$ under the composition of 
\[
\xymatrix
{
H^2(\Sp(2g, \bZ); \bZ)\ar[r]& H^2(\Gamma_g; \bZ)\ar[r]&\mathsf{Hom}_\bZ(H_2(\Gamma_g; \bZ), \bZ)
}
\] is a map $\hat{k}\in \mathsf{Hom}_\bZ(H_2(\Gamma_g; \bZ), \bZ) $ with image $4\bZ$. On the other hand, Barge-Ghys show in \cite[Proposition 3.1, p. 258]{barge-ghys} that $\tau_g$ is four times the cocycle defined by the central extension $\widetilde{\Sp(2g, \bZ)}$, the restriction of the universal cover $\widetilde{\Sp(2g, \bR)}$ of $\Sp(2g, \bR)$ to $\Sp(2g, \bZ)$. This means that $\tau_g$ is four times the generator of $H^2(\Sp(2g, \bZ); \bZ)$, which maps to four times the generator of $H^2(\Gamma_g; \bZ)$, showing that $p^*$ is an isomorphism.
\item[\rm (xiv)]
We showed above that $p_*\colon H_2(\Gamma_g; \mathbb{Z})\rightarrow H_2(\Sp(2g, \bZ); \bZ)$ is an isomorphism for $g\geq 3$. The universal coefficient theorem and perfection of $\Gamma_g$ and $\Sp(2g, \bZ)$ yield a commutative diagram
\[
\xymatrix
{
H^2(\Sp(2g, \bZ); \bZ/2)\ar[r]^-\cong\ar[d]^{p^*} & \mathsf{Hom}(H_2(\Sp(2g, \bZ); \bZ), \bZ/2)\ar[d]^{\mathsf{Hom}(p_*)}\phantom{.}\\
 H^2(\Gamma_g; \bZ/2)\ar[r]^-\cong &\mathsf{Hom}(H_2(\Gamma_g; \bZ), \bZ/2).
}
\]
The left vertical arrow must therefore be an isomorphism as well.

For $g=2$, recall that by \eqref{mapping class group} and \eqref{H^2 Sp Z Z/2} the map $p^*$ cannot possibly be an isomorphism. We use the following commutative diagram to study its kernel:
$$
\xymatrix
{
H^2(\Sp(4, \bZ); \bZ)\ar[r]^{r_2}\ar[d]^{p^*}&H^2(\Sp(4, \bZ); \bZ/2)\ar[d]^{p^*}\phantom{.}\\
H^2(\Gamma_2; \bZ)\ar[r] & H^2(\Gamma_2; \bZ/2).
}
$$
The class $[\tau_2]\in H^2(\Sp(4, \bZ); \bZ)$ has non-zero image $$r_2([\tau_2])=[\tau_2]\pmod 2\in H^2(\Sp(4, \bZ); \bZ/2)$$
by \eqref{H^2 Sp image r2}, from which it follows also that this image has order $2$. By \cite[Satz 2]{Meyer:1973a}, the order of $p^*([\tau_2])$ is $5$ in $H^2(\Gamma_2;\bZ)$. But then its image in $H^2(\Gamma_2; \bZ/2)$ has to be $0$, by \eqref{mapping class group}. So $p^*\colon H^2(\Sp(4, \bZ); \bZ/2)\rightarrow H^2(\Gamma_2; \bZ/2)$ is not injective and has $r_2([\tau_2])$ in its kernel. \qedhere
\end{enumerate}
\end{proof}


\section{Werner Meyer (by Winfried Scharlau)}\label{app: biography}

We are grateful to Winfried Scharlau for allowing us to publish here his photo of Meyer, and the following English translation of an extract from his 2017 biography of Hirzebruch, \cite{Scharlau:2017a}.

\bigskip

\begin{minipage}{0.35\textwidth}
\centering
\includegraphics[width=5cm]{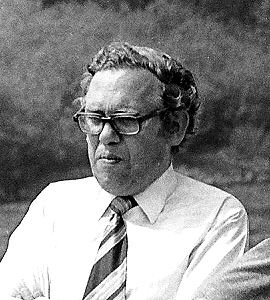}
\end{minipage}
\begin{minipage}{0.62\textwidth}

\medskip
Werner Meyer (1929--1991) was disabled; he was born with a cleft palate, which was difficult to operate on. His speech difficulties ruled out an academic career involving teaching. His parents were Jehovah's Witnesses: they were persecuted by the Nazis, and killed in the Ravensbr\"uck concentration camp in 1944.
Werner Meyer's brother was executed during the war, on account of being a conscientious objector. 
All these circumstances led to Meyer not getting the education his talents deserved.
\end{minipage}

\medskip

When Meyer started working at the Max Planck Institute for Mathematics in June 1983 he submitted the following autobiographical sketch. 

\textit{After the war, early in 1948, I completed an apprenticeship as an electrical fitter.  After that I worked for several companies as a fitter, a locksmith and a technical draftsman. From 1959 to 1961 I attended the local technical college in Rheydt and obtained a high school diploma on 31.3.1961. 
From the spring semester 1962 to the spring semester 1963 I studied mechanical engineering at the state school of Engineering in Aachen, taking the preliminary examination for a higher degree in engineering. In order to  become a university student I applied for a scholarship in Duesseldorf, which I obtained on 9th October 1963.} 

One should add that Werner Meyer was in a sense ``discovered" by Wilhelm Junkers, a mathematician who later became a professor in Duisburg. 
Junkers was once visiting a firm, when some of Meyer's colleagues  told Junkers about that eccentric who was always working on mathematics. Junkers made sure that Meyer could catch up on his high school diploma at evening classes and that he would enroll in engineering studies in Aachen. It was also Junkers who introduced Meyer to Hirzebruch so that Meyer could study mathematics in Bonn from the winter semester 1963/64. It soon became clear that he was extraordinarily gifted. He had very broad interests and acquired an extensive mathematical knowledge. He remains in the memories of many students and collaborators as being especially knowledgeable, competent and helpful. Meyer obtained his doctorate in 1971 under Hirzebruch's supervision and was awarded (jointly with Zagier) the  Hausdorff Memorial Prize for his dissertation.

After that he stayed on in Bonn under Hirzebruch's protective wing. He obtained long term positions at the Sonderforschungsbereich and later at the Max Planck Institute. 
In 1979 he obtained the ``Habilitation" with his work on ``Signature defects, Teichm\"uller groups and hyperelliptic fibrations" 
. 
In 1986 he was appointed by Bonn University as adjunct professor. The last years before his premature death  were overshadowed by ill health  (kidney failure) -- Mathematics was without doubt his greatest happiness in life.

 On 15.1.1993  a memorial colloquium took place in Meyer's honor. Hirzebruch gave a speech, talking extensively about Meyer's life and gave a lecture about his work on the signature of fibre bundles based on the paper ``Die Signatur von Fl\"achenb\"undeln". Wilhelm Plesken gave another lecture about ``Constructive methods in representation theory".

In his report for the Max Planck Society for the period from October 1991 to September 1994 Hirzebruch  mentioned Meyer on the very first page:

\textit{``He was able to go deep into the matter. In some lectures and seminars about new research I used to set problems that I couldn't solve myself, promising a bottle of wine to whoever could give a solution. Many times the prize went to Werner Meyer. With his great knowledge he was an outstanding advisor for Masters and Doctoral candidates. He helped many students even if they were not writing their dissertation officially with him."}




\subsection*{Acknowledgements}
The authors thank Christophe Pittet for a helpful correspondence about Theorem \ref{th:Borel}. We also want to thank the anonymous referee for providing very useful comments on this paper.
Campagnolo acknowledges support by the Swiss National Science Foundation, grant number PP00P2-128309/1, under which this research was started; Campagnolo and Rovi acknowledge support by the Deutsche Forschungsgemeinschaft (DFG, German Research Foundation) – 281869850 (RTG 2229), under which this research was continued and completed.

\medskip 

The writing of this paper was overshadowed by the death of Andrew Ranicki in February 2018. With deep sorrow, we dedicate this work to his memory.






\begin{thebibliography}{00}


  \bibitem{barge-ghys}
J.~Barge, E.~Ghys, \emph{Cocycles d'Euler et de Maslov}, Math. Ann. \textbf{294} (1992), no.~2, 235--265.

\bibitem{baaj-skandalis-vaes}
S.~Baaj, G.~Skandalis, and S.~Vaes,
\emph{Measurable Kac Cohomology for Bicrossed Products},
Trans. Amer. Math. Soc. \textbf{357} (November 2004), no.~4, 1497--1524.

\bibitem{bender}
P.~Bender, \emph{Eine Pr\"asentation der symplektischen Gruppe $\Sp(4, \bZ)$ mit $2$ Erzeungenden und $8$ definierenden Relationen},
J. Algebra \textbf{65} (1980), no.~2, 328--331.

\bibitem{Benson:theta}
D.~J.~Benson, \emph{Theta functions and a presentation of
  $2^{1+(2g+1)}\mathsf{Sp}(2g,2)$}, J. Algebra \textbf{527} (2019), 204--240.
 
\bibitem{BensonCampagnoloRanickiRovi:2017a}
D. J. Benson, C. Campagnolo, A. Ranicki, and C. Rovi, \emph{Cohomology of symplectic groups and Meyer's signature theorems}, Algebr. Geom. Topol. \textbf{18} (2018), no.~7, 4069--4091.

\bibitem{beyl}
F.~R.~Beyl, \emph{The Schur multiplicator of $\rm {SL}(2,{\bf Z}/m{\bf Z})$ and the Congruence Subgroup Property}, Math. Z. \textbf{191} (1986), no.~1, 23--42.
    
\bibitem{brown}
	K.~S.~Brown, \emph{Cohomology of Groups},
	Springer (1982).

\bibitem{CampagnoloThesis}
C.~Campagnolo,
\emph{Surface bundles over surfaces: A study through signature and simplicial volume},
10.13097/archive-ouverte/unige:96380, 2016.

\bibitem{ChatterjiMislindeCornulierPittet:2013a}
I.~Chatterji, Y.~de Cornulier, G.~Mislin, and C.~Pittet,
\emph{Bounded characteristic classes and flat bundles},
J. Differential. Geom. \textbf{95} (2013), no.~1, 39--51.

\bibitem{ChatterjiMislinPittetSaloff-Coste}
I.~Chatterji, G.~Mislin, C.~Pittet, and L.~Saloff-Coste,
\emph{A geometric criterion for the boundedness of characteristic classes},
Math. Ann. \textbf{351} (2011), no.~3, 541--569.

\bibitem{chern-hirzebruch-serre}
	S. S. Chern, F. Hirzebruch, J.-P. Serre,
	\emph{On the Index of a Fibered Manifold},
	Proceedings of the American Mathematical Society \textbf{8} (Jun., 1957), no.~3, 587--596.

\bibitem{atlas}
J.~H.~Conway, R.~Turner~Curtis, S.~P.~Norton, R.~A.~Parker, and R.~Arnott Wilson,
\emph{Atlas of Finite Groups. Maximal Subgroups and Ordinary Characters for Simple Groups}, 
Oxford University Press, 1985.

\bibitem{Deligne:1978a}
P.~Deligne, \emph{Extensions centrales non r\'esiduellement finies des groupes
  arithm\'etiques}, C. R. Acad. Sci. Paris, S\'er. A-B \textbf{287}
  (1978), no.~4, A203--A208.

\bibitem{Farb/Margalit:2012a}
B.~Farb and D.~Margalit, \emph{A primer on mapping class groups},
Princeton (2012).

\bibitem{galatius-randal-williams}
S.~Galatius and O.~Randal-Williams, \emph{Abelian quotients of mapping class groups of highly connected manifolds}, Math. Ann. \textbf{365} (2016), no.~1-2, 857--879.

\bibitem{Gilmer/Masbaum:2013a}
P.~Gilmer and G.~Masbaum, \emph{Maslov index, lagrangians, mapping class groups and TQFT}, Forum Math. \textbf{25} (2013), no.~5, 1067--1106.

\bibitem{Griess:1973a}
R.~L.~Griess, \emph{Automorphisms of extra special groups and nonvanishing
  degree $2$ cohomology}, Pacific J. Math. \textbf{48} (1973), no.~2,
  403--422.

\bibitem{Hambleton/Korzeniewski/Ranicki:2007a}
I.~Hambleton, A.~Korzeniewski, and A.~Ranicki, \emph{The signature of a fibre
  bundle is multiplicative mod $4$}, Geom. Topol. \textbf{11} (2007),
  251--314.
  
  
 
\bibitem{Hochschild:1965a}
G.~Hochschild, \emph{The Structure of Lie Groups},
Holden-Day (1965).

\bibitem{Hopf:1942a}
H.~Hopf, \emph{Fundamentalgruppe und zweite Bettische Gruppe},
Comment. Math. Helv. \textbf{14} (1942), 257--309.

\bibitem{kirby-melvin}
R.~Kirby and P.~Melvin, \emph{Dedekind sums, $\mu$-invariants and the signature cocycle}, Math. Ann. \textbf{299} (1994), no.~2, 231--267.

\bibitem{korkmaz}
M.~Korkmaz, \emph{Low-dimensional homology groups of mapping class groups: a survey}, Turkish J. Math. \textbf{26} (2002), no.~1, 101--114.
  
\bibitem{Korkmaz/Stipsicz:2003a}
M.~Korkmaz and A.~I.~Stipsicz,
\emph{The second homology groups of mapping class groups of orientable surfaces},
Math. Proc. Cambridge Philos. Soc. \textbf{134} (2003), no.~3, 479--489.
  
\bibitem{Leray:1981a}
J.~Leray, \emph{Lagrangian analysis and quantum mechanics},
MIT Press (1981).

\bibitem{lu}
N.~Lu, \emph{A Simple presentation of the Siegel Modular Groups}, Linear Algebra Appl. \textbf{166} (1992), 185--194.

\bibitem{Mackey:1957a}
G.~W.~Mackey, \emph{Borel structure in groups and their duals},
Trans. Amer. Math. Soc. \textbf{85} (1957), 134--165.

\bibitem{Maslov:1965a}
V.~P.~Maslov, \emph{Th\'eorie des perturbations et m\'ethodes asymptotiques}, Dunod (1972 (transl. of 1965 Russian edition)).

\bibitem{Meyer:1972a}
W.~Meyer, \emph{Die Signatur von lokalen Koeffizientensystemen und Faserb\"undeln.},
Bonn. Math. Schr. No. 53 (1972), viii+59,
\href{http://www.maths.ed.ac.uk/~aar/papers/meyerthesis.pdf}{http://www.maths.ed.ac.uk/~aar/papers/meyerthesis.pdf}


\bibitem{Meyer:1973a}
W.~Meyer, \emph{Die Signatur von Fl\"achenb\"undeln}, Math.\ Ann.
  \textbf{201} (1973), 239--264.

\bibitem{Milnor:1983a}
J.~Milnor, \emph{On the homology of Lie groups made discrete},
Comment. Math. Helv. \textbf{58} (1983), no.~1, 72--85.

\bibitem{moore}
C.~C.~Moore,
\emph{Group Extensions and Cohomology for Locally Compact Groups. III},
Trans. Amer. Math. Soc. \textbf{221} (July 1976), no.~1, 1--33.

\bibitem{Py:2005a}
P.~Py, \emph{Indice de Maslov et Th\'eor\`eme de Novikov--Wall},
Bol.\ Soc.\ Mat.\ Mexicana (3) \textbf{11} (2005), no.~2, 303--331.

\bibitem{Rovi:AGT}
C.~Rovi, \emph{The non-multiplicativity of the signature modulo $8$ of a fibre
  bundle is an Arf--Kervaire invariant}, Algebr. Geom. Topol. \textbf{18} (2018), 1281--1322.

\bibitem{Sah/Wagoner:1977a}
C.-H.~Sah and J. B.~Wagoner,
\emph{Second homology of Lie groups made discrete},
Comm. Algebra \textbf{5} (1977), no.~6, 611--642.

\bibitem{sakasai}
T.~Sakasai, \emph{Lagrangian mapping class groups from a group cohomological point of view}, Algebr. Geom. Topol. \textbf{12} (2012), no.~1, 267--291.

\bibitem{Scharlau:2017a}
W.~Scharlau, \emph{Das Gl\"uck, Mathematiker zu sein. Friedrich Hirzebruch und seine Zeit}, 
Heidelberg: Springer Spektrum (2017).

\bibitem{Schur:1904a}
J.~Schur, \emph{\"Uber die Darstellung der endlichen Gruppen durch gebrochene lineare Substitutionen},
J. reine angew. Math. \textbf{127} (1904), 20--50.

\bibitem{Thurston:1974a}
W.~Thurston, \emph{Foliations and groups of diffeomorphisms},
Bull. Amer. Math. Soc. \textbf{80} (1974), 304--307.

\bibitem{Wall:1969a}
C.~T.~C.~Wall, \emph{Non-additivity of the signature},
Invent. Math, \textbf{7} (1969), 269--274.

\bibitem{Weil:1940a}
A.~Weil, \emph{L'int\'egration dans les groupes topologiques et ses applications}, Hermann, Paris (1940).

\bibitem{wigner}
D.~Wigner,
\emph{Algebraic Cohomology of Topological Groups},
Trans. Amer. Math. Soc. \textbf{178} (April 1973), 83--93.

                     \end{thebibliography}


\end{document}